\def\todaysdate{3\textsuperscript{rd} January 2022}
\documentclass[reqno,a4paper]{article}
\usepackage{etex}
\usepackage[T1]{fontenc}
\usepackage[utf8]{inputenc}
\usepackage{lmodern}
\usepackage{amsmath,amscd}

\usepackage[style=alphabetic,firstinits=true,backref,backend=biber,isbn=false,url=false,maxbibnames=99]{biblatex}

\renewbibmacro{in:}{}
\newbibmacro{string+doi}[1]{\iffieldundef{doi}{\iffieldundef{url}{#1}{\href{\thefield{url}}{#1}}}{\href{http://dx.doi.org/\thefield{doi}}{#1}}}
\DeclareFieldFormat*{title}{\usebibmacro{string+doi}{\emph{#1}}}
\DefineBibliographyStrings{english}{%
  backrefpage = {$\uparrow$},
  backrefpages = {$\uparrow$},
}

\usepackage{amssymb,amsmath,amsthm}
\usepackage[dvipsnames]{xcolor}
\usepackage{thmtools}
\usepackage{units}
\usepackage{graphicx}
\usepackage[all]{xy}
\usepackage{tikz}
\usetikzlibrary{arrows,calc,positioning,decorations.pathreplacing,shapes.misc,cd}
\usepackage{relsize}
\usepackage{mhsetup}
\usepackage{mathtools}
\usepackage{stmaryrd}
\usepackage{tocloft}
\usepackage{paralist} % For compactitem environment
\usepackage[left=3cm,right=3cm,top=3cm,bottom=3cm]{geometry} % Decrease the margins %
\usepackage[bottom]{footmisc}
\usepackage[colorlinks,citecolor=blue,linkcolor=blue,urlcolor=blue,filecolor=blue,breaklinks]{hyperref}
\usepackage{footnotebackref}
\usepackage[format=plain,indention=1em,labelsep=quad,labelfont={up},textfont={footnotesize},margin=2em]{caption}
\usepackage[mathscr]{euscript}
\usepackage{accents}
\usepackage{marginnote}
\usepackage{pinlabel}

\definecolor{lightblue}{rgb}{0.8,0.8,1}

\cftpagenumbersoff{part}

\numberwithin{equation}{section}
\numberwithin{figure}{section}

\declaretheoremstyle[
  spaceabove=\topsep,
  spacebelow=\topsep,
  headpunct=,
  numbered=no,
  postheadspace=1ex,
  headfont=\normalfont\bfseries,
  bodyfont=\normalfont\itshape,
]{italic}
\declaretheoremstyle[
  spaceabove=\topsep,
  spacebelow=\topsep,
  headpunct=,
  numbered=no,
  postheadspace=1ex,
  headfont=\normalfont\bfseries,
  bodyfont=\normalfont\upshape,
]{upright}
\declaretheorem[style=italic,name=Theorem,numbered=yes,numberwithin=section]{thm}
\declaretheorem[style=italic,name=Lemma,numbered=yes,numberlike=thm]{lem}
\declaretheorem[style=italic,name=Proposition,numbered=yes,numberlike=thm]{prop}
\declaretheorem[style=italic,name=Corollary,numbered=yes,numberlike=thm]{coro}
\declaretheorem[style=upright,name=Definition,numbered=yes,numberlike=thm]{defn}
\declaretheorem[style=upright,name=Remark,numbered=yes,numberlike=thm]{rmk}
\declaretheorem[style=upright,name=Example,numbered=yes,numberlike=thm]{eg}
\declaretheorem[style=upright,name=Notation,numbered=yes,numberlike=thm]{notation}
\declaretheorem[style=upright,name=Terminology,numbered=yes,numberlike=thm]{terminology}
\declaretheorem[style=upright,name=Construction,numbered=yes,numberlike=thm]{construction}

\makeatletter
\renewcommand*{\@seccntformat}[1]{\upshape\csname the#1\endcsname.\hspace{1ex}}
\renewcommand*{\part}{\@startsection{part}{0}{\z@}%
	{2.5ex \@plus 1ex \@minus 0.2ex}%
	{1.5ex \@plus 0.2ex}%
	{\normalfont\large\bfseries\centering}}
\renewcommand*{\section}{\@startsection{section}{1}{\z@}%
	{2.5ex \@plus 1ex \@minus 0.2ex}%
	{1.5ex \@plus 0.2ex}%
	{\normalfont\large\bfseries}}
\renewcommand*{\subsection}{\@startsection{subsection}{2}{\z@}%
	{2.5ex \@plus 1ex \@minus 0.2ex}%
	{-1.5ex \@plus -0.2ex}%
	{\normalfont\normalsize\bfseries}}
\renewcommand*{\subsubsection}{\@startsection{subsubsection}{3}{\z@}%
	{2.5ex \@plus 1ex \@minus 0.2ex}%
	{-1.5ex \@plus -0.2ex}%
	{\normalfont\normalsize\bfseries}}
\renewcommand*{\paragraph}{\@startsection{paragraph}{4}{\z@}%
	{2.5ex \@plus 1ex \@minus 0.2ex}%
	{-1.5ex \@plus -0.2ex}%
	{\normalfont\normalsize\bfseries}}
\renewcommand*{\subparagraph}{\@startsection{subparagraph}{5}{\z@}%
	{2.5ex \@plus 1ex \@minus 0.2ex}%
	{-1.5ex \@plus -0.2ex}%
	{\normalfont\normalsize\slshape}}
\makeatother

\newcommand{\cf}{\textit{cf}.\ }

\newcommand{\Diff}{\ensuremath{\mathrm{Diff}}}
\newcommand{\Map}{\ensuremath{\mathrm{Map}}}
\newcommand{\twoheadrightarrowlong}{\ensuremath{\longrightarrow\mathrel{\mkern-14mu}\rightarrow}}
\newcommand{\longhookrightarrow}{\ensuremath{\lhook\joinrel\longrightarrow}}
\newcommand{\too}{\ensuremath{\longrightarrow}}

\newcommand{\incl}[3][right]%
{%
\draw[<-,>=#1 hook] #2 to ($ #2!0.5!#3 $);
\draw[->] ($ #2!0.5!#3 $) to #3;%
}

\newenvironment{itemizeb}%
{\begin{compactitem}

}%
{\end{compactitem}}

\newcommand{\bR}{\mathbb{R}}

\newcommand{\bZ}{\mathbb{Z}}

\renewcommand{\geq}{\geqslant}
\renewcommand{\leq}{\leqslant}
\renewcommand{\footnoterule}{%
  \kern -3pt
  \hrule width \textwidth height 0.4pt
  \kern 2.6pt
}
\newcommand{\cmap}[4]{\ensuremath{\mathrm{CMap}_{#1}^{#2}(#3;#4)}}

\begin{document}
\title{\Large\bfseries Point-pushing actions for manifolds with boundary}
\author{\normalsize Martin Palmer and Ulrike Tillmann}
\date{\small\todaysdate}
\maketitle
{
\makeatletter
\renewcommand*{\BHFN@OldMakefntext}{}
\makeatother
\footnotetext{2020 \textit{Mathematics Subject Classification}: 55R80, 57N65, 55R10, 20F38}
\footnotetext{\textit{Key words and phrases}: Monodromy actions, point-pushing actions, Birman exact sequence, configuration-mapping spaces.}
\footnotetext{The authors would like to thank the Isaac Newton Institute for Mathematical Sciences, Cambridge, for support and hospitality during the programme \emph{Homotopy Harnessing Higher Structures} (EPSRC grant number EP/R014604/1), where this work was initiated. We also would also like to thank the IAS/Park City Mathematics Institute for an excellent research environment during the summer 2019 programme \emph{Quantum field theory and manifold invariants} (NSF Grant 1441467). The first author was partially supported by a grant of the Romanian Ministry of Education and Research, CNCS - UEFISCDI, project number PN-III-P4-ID-PCE-2020-2798, within PNCDI III.}
}
\begin{abstract}
Given a manifold $M$ and a point in its interior, the point-pushing map describes a diffeomorphism that pushes the point along a closed path. This defines a homomorphism from the fundamental group of $M$ to the group of isotopy classes of diffeomorphisms of $M$ that fix the basepoint. This map is well-studied in dimension $d=2 $ and is part of the Birman exact sequence. Here we study, for any $d\geq 3$ and $k\geq 1$, the map  from the $k$-th braid group of $M$  to the group of homotopy classes of homotopy equivalences of the $k$-punctured manifold $M \smallsetminus z$, and analyse its injectivity. Equivalently, we describe the monodromy of the universal bundle that associates to a configuration $z$ of size $k$ in $M$ its complement, the space $M\smallsetminus z$.
Furthermore, motivated by our work in \cite{PalmerTillmann2020homologicalstabilityconfigurationsection}, we describe the action of the braid group of $M$ on the fibres of configuration-mapping spaces.
\end{abstract}
%%%%%%%%%%%%%%%%%%%%%%%%%%%%%%%%%%%%%%%%%%%%%%%%%%%%%%%%%%%%%%%%%%%%%%%%%%%%%%%%%%
%%%%%%%%%%%%%%%%%%%%%%%%%%%%%%%%%%%%%%%%%%%%%%%%%%%%%%%%%%%%%%%%%%%%%%%%%%%%%%%%%%

%%%%%%%%%%%%%%%%%%%%%%%%%%%%%%%%%%%%%%%%%%%
%%%%%%%%%%%%%%%%%%%%%%%%%%%%%%%%%%%%%%%%%%%
\section{Introduction}\label{s:intro}
%%%%%%%%%%%%%%%%%%%%%%%%%%%%%%%%%%%%%%%%%%%
%%%%%%%%%%%%%%%%%%%%%%%%%%%%%%%%%%%%%%%%%%%

Let $M$ be a based, connected (smooth) manifold of dimension $d\geq 2$ and denote by $C_k(\mathring{M})$ the configuration space of $k$ unordered distinct points in its interior. We may think of it as the \textit{moduli space} of $k$ distinct points in $M$. Its \textit{universal bundle} is the fiber bundle $U_k(M)$ that associates to each $k$-tuple $z \in C_k(\mathring{M})$ the $k$-punctured manifold $M\smallsetminus z$:
\[\begin{tikzcd}
M \smallsetminus z \ar[r] & U_k(M) \ar[d,swap,"u"] \\
& C_k(\mathring{M}).
\end{tikzcd}\]
The primary goal of this paper is to describe the monodromy action (up to homotopy) of the above fibre bundle
\[
\mathrm{push}_{(M,z)} \colon \pi_1 (C_k(\mathring{M}), z) \longrightarrow \pi_0 (\text{hAut} (M \smallsetminus z))
\]
where $\text{hAut} (M \smallsetminus z) $ denotes the homotopy equivalences of the complement of $z$ in $M$; when $M$ has boundary we will consider the relative homotopy equivalences, and any base point is on the boundary of $M$.

Let $(X, *)$ be a fixed connected based space and assume that $M$ has boundary and a basepoint. Applying the continuous functor $\Map_* ( \quad ; X)$ (based maps to $X$) fibrewise to $u$ defines a new fibre bundle:
\[\begin{tikzcd}
\mathrm{Map}_*(M \smallsetminus z,X) \ar[r] & \cmap{k}{*}{M}{X} \ar[d,swap,"p"] \\
& C_k(\mathring{M}).
\end{tikzcd}\]
Our second goal is to give explicit formulas for the monodromy action for $p$ (up to homotopy). 
The total space is an example of the configuration-mapping  spaces studied in  \cite{EllenbergVenkateshWesterland2012HomologicalstabilityHurwitz, PalmerTillmann2020homologicalstabilityconfigurationsection}. 
Indeed, our interest in the monodromy actions  was motivated by our study of  homology stability  for configuration-mapping spaces. 

When $z$ is just a single point the monodromy map can be defined in terms of the point-pushing map: it sends an element $[\alpha] \in \pi_1(M, z)$ to the pointed isotopy class of the diffeomorphism that pushes the point $z$ along the curve $\alpha$ and is the identity outside  a small neighbourhood of the image of $\alpha$. It is not difficult to see that the point pushing map and more generally 
$\mathrm{push}_{(M,z)}$ factors through the (smooth) mapping class group:
\[
\mathrm{push}_{(M,z)}^{\mathsf{sm}} \colon \pi _1 (C_k(\mathring{M}), z )\longrightarrow
\pi_0 (\Diff (M; z));
\]
here $\Diff(M;z)$ denotes the group of (smooth) diffeomorphisms of $M$ that permute the points in $z$. If the boundary of $M$ is non-empty we will consider those diffeomorphisms that fix the boundary.

There is a possibly more familiar alternative description of $\mathrm{push}_{(M,z)}^{\mathsf{sm}}$. For $z$ a single point in $M$, consider the fibration
\[
\Diff (M ; z  ) \longrightarrow \Diff (M) \overset {eval} \longrightarrow M
\]
where $eval$ denotes the map that evaluates a diffeomorphism at $z$. As $M$ is path-connected, this gives rise to the exact sequence
\[
0 \longrightarrow L \longrightarrow \pi_1 (M, z) \longrightarrow \pi_0 (\Diff (M ; z )) \longrightarrow \pi_0 (\Diff (M)) \longrightarrow 0,
\]
where $L$ is by definition the kernel of the middle map.

For $M=S$ a surface of negative Euler characteristic, the connected components of $\Diff (M)$ are contractible \cite{EarleEells} \cite{EarleSchatz} and hence the fibration gives rise to the
Birman exact sequence \cite{Birman1969}
\[
0 \longrightarrow \pi_1 (S,z) \longrightarrow \pi_0 (\Diff (S ; z )) \longrightarrow \pi_0 (\Diff (S)) \longrightarrow 0.
\]
When $\alpha$ is represented by a simple curve that has a two-sided neighbourhood in $S$, its image is a product of the two Dehn twists around the two curves (oriented oppositely) that form the boundary of a tubular neighbourhood of $\alpha$. On the other hand, when $S=T$ is the torus, $\Diff (T) \simeq T \rtimes \mathrm {SL}_2 (\mathbb Z)$ \cite{Gramain1973} and $eval$ induces an isomorphism on fundamental groups:
\[
\pi_1 (\Diff (T); \mathrm{id}_T) \cong K = \pi_1 (T , z) \cong  \mathbb Z^2.
\]
Thus the smooth point-pushing map (and hence also the non-smooth version) is well-understood when $d=2$. Recently, Banks \cite{Banks2017} completely determined the kernel $L$ also when $d=3$. In particular she shows that $L$ is trivial unless the manifold $M$ is prime and Seifert fibered via an $S^1$ action.
In a different direction, Tshishiku~\cite{Tshishiku2015} studies the Nielsen realisation problem for the point-pushing map, i.e.~asks when the point-pushing map can be factored through $\Diff(M,z)$.
However, little seems to be known about the image of the point-pushing map in higher dimensions. Here we give a complete description, up to homotopy, of the induced self-map of $M\smallsetminus z$ for any element of the fundamental group when $M$ has non-empty boundary. As an example, in section \ref{s:examples}, we study the manifolds $M_{g,1}^d = \sharp^g (S^1 \times S^{d-1}) \smallsetminus \mathring{D}^d$ for $d\geq 3$ and $g\geq 0$ and show that the point-pushing map is injective for these examples. Inspired by our calculations in these examples, we discuss injectivity more generally in \S\ref{s:kernel-point-pushing}. We note that for these examples $M^d_{g,1}$, the Nielsen realisation problem is solvable as the fundamental group is free.

\paragraph{Outline and results.}

The paper is organised as follows. Section \ref{s:monodromy-actions} contains basic recollections about (relative) monodromy actions associated to fibrations and Section \ref{s:point-pushing-actions} discusses equivalent definitions of the point-pushing map (see Figure \ref{fig:point-pushing}), and considers the induced actions for associated fibre bundles obtained from the universal bundle $u$ by applying a continuous functor. Restricting from now on to manifolds with boundary and dimension $d\geq 3$, in Section \ref{s:formulas} we note that for a $k$-tuple $z$, up to homotopy, $M\smallsetminus z$ decomposes as a wedge of $M$ with a $k$-fold wedge sum of spheres $S^{d-1}$,
\[
M\smallsetminus z \simeq M \vee  W_k \qquad \text{ where } \qquad W_k := \bigvee _{k } S^{d-1}
\]
and $\pi_1 (C_k(\mathring{M}),z)$ is isomorphic to the wreath product 
\[
\pi_1(M)^k \rtimes \Sigma_k.
\] 
Thus the task of understanding the monodromy action is divided into understanding (on each of the terms $M$ and $W_k$) the action of the symmetric group elements, which is done in Section \ref{s:symmetric}, and the more complicated action of the loop elements, considered in Section \ref{s:loop}. The elements of the symmetric group act, up to homotopy, by the identity on $M$ and by permuting the $k$ summands in the wedge product $W_k$; compare Proposition \ref{p:formula-sigma}. The precise action of a loop $\alpha \in \pi_1 M$ is the content of Propositions \ref{p:formula-alpha-1} and \ref{p:formula-alpha-2}. Roughly, when $\alpha$ is in the $i$-th factor of the wreath product, it acts on the summand $W_k$ by taking the $i$-th sphere $S^{d-1}$ and mapping  a neighbourhood of its base point around $\alpha$ before covering itself by a degree $\pm 1$ map depending on whether $\alpha$ lifts to a loop in the orientation double cover of $M$. The other factors of $W_k$ are mapped by the inclusion. This completely describes the monodromy action of $\alpha$ on $W_k \to M\vee W_k$. The action of $\alpha$ on $M$ depends only on the sequence of intersections of $\alpha$ with the $(d-1)$-cells of $M$, or more precisely those of an embedded CW-complex $K$ of dimension at most $d-1$ such that $M$ deformation retracts onto it; compare formula \eqref{eq:pf2} and Figure \ref{fig:alpha-tau}. So, if there are no such intersections, for example when $K$ has no $(d-1)$-cells, then the action on $M$ is simply given by the inclusion. However, if $\alpha$ intersects a $(d-1)$-cell $\tau$ of $K$ with intersection number $\sharp (\tau, \alpha)$ then in addition to the inclusion of $M$, the monodromy action of $\alpha$ takes the cell $\tau$ to the $i$-th factor of $W_k$ by a degree $\sharp (\tau, \alpha)$ map. These assemble to give a map:
\[
M \simeq K \twoheadrightarrowlong K/K^{(d-2)} \simeq \bigvee_\tau S^{d-1} \longrightarrow S^{d-1} \subseteq W_k
\]
where  $K^{(d-2)}$ denotes the $(d-2)$-skeleton of $K$. This completely describes the monodromy action of $\alpha$ on $M \to M\vee W_k$ after projection to each factor $M$ and $W_k$. The full description of this action in Definition \ref{defn:pf2} takes into account the precise sequence of intersections of $\alpha$ and the $(d-1)$-cells. We illustrate this latter more complicated action of $\alpha$ with  examples in Section \ref{s:examples}. In Section \ref{s:kernel-point-pushing} we discuss the general question of injectivity for the point-pushing map. We show that, up to isomorphism, the kernel of the point-pushing map is independent of $k$ regardless whether diffeomorphisms, homeomorphisms or homotopy equivalences are considered. In particular, it is always injective when the manifold has non-empty boundary. Our main result in this direction is contained in Proposition \ref{p:kernel-of-pk}. Finally in Section \ref{s:mapping-spaces} the induced action on the fibres of $p$ for configuration mapping spaces is described. As a further application we compute the number of connected components for configuration mapping spaces in Corollary \ref{cor:connected_components}.

\paragraph{Acknowledgements.}
The authors are grateful to the anonymous referee for detailed corrections and suggestions for improvements to an earlier draft of this article.

\tableofcontents

%%%%%%%%%%%%%%%%%%%%%%%%%%%%%%%%%%%%%%%%%%%
%%%%%%%%%%%%%%%%%%%%%%%%%%%%%%%%%%%%%%%%%%%
\section{Monodromy actions}\label{s:monodromy-actions}
%%%%%%%%%%%%%%%%%%%%%%%%%%%%%%%%%%%%%%%%%%%
%%%%%%%%%%%%%%%%%%%%%%%%%%%%%%%%%%%%%%%%%%%

We first recall the \emph{monodromy action} associated to a fibration. Let $f \colon E \to B$ be a continuous map and write $F = f^{-1}(b)$ for a point $b \in B$. Assume that $f$ satisfies the \emph{homotopy lifting property} (\emph{covering homotopy property}) (\cf \cite[\S 4.2]{Hatcher2002} or \cite[\S 7]{May1999}) with respect to the spaces $F$ and $F \times [0,1]$. For example, this holds if $f$ is a Hurewicz fibration, or if $f$ is a Serre fibration and $F$ is a CW-complex. In particular it holds whenever $f$ is a fibre bundle and either $B$ is paracompact or $F$ is a CW-complex.

\begin{defn}
\label{d:haut}
For a space $F$, write $\mathrm{hAut}(F) \subseteq \mathrm{Map}(F,F)$ for the space of continuous self-maps $F \to F$, with the compact-open topology, that admit a homotopy inverse. This is a topological monoid under composition, and \emph{grouplike}, i.e.~the discrete monoid $\pi_0(\mathrm{hAut}(F))$ is a group (it is the automorphism group of $F$ in the homotopy category).

For a pair of spaces $(F,F_0)$, we write $\mathrm{End}(F|F_0)$ for the topological monoid (with the compact-open topology) of self-maps of $F$ that are the identity on $F_0$ and we write $\mathrm{hAut}(F|F_0) \subseteq \mathrm{End}(F|F_0)$ for the union of those path-components of $\mathrm{End}(F|F_0)$ corresponding to the invertible elements of the discrete monoid $\pi_0(\mathrm{End}(F|F_0))$. Note that $\mathrm{hAut}(F|\varnothing) = \mathrm{hAut}(F)$. See also Remark \ref{rmk:relative-hAut}.
\end{defn}

\begin{defn}[\emph{Monodromy actions.}]
\label{d:monodromy-action}
Under the above assumptions, the \emph{monodromy action} associated to $f$ is the action-up-to-homotopy
\begin{equation}
\label{eq:monodromy-action}
\mathrm{mon}_f \colon \pi_1(B,b) \longrightarrow \pi_0(\mathrm{hAut}(F))
\end{equation}
of $\pi_1(B,b)$ on $F$ defined as follows. For an element $[\gamma] \in \pi_1(B,b)$ represented by a loop $\gamma \colon [0,1] \to B$, let $g \colon F \times [0,1] \to E$ be a choice of lift in the diagram:
\begin{equation}
\label{eq:monodromy-lifting}
\centering
\begin{split}
\begin{tikzpicture}
[x=1mm,y=1mm]
\node (tl) at (0,15) {$F$};
\node (tr) at (40,15) {$E$};
\node (bl) at (0,0) {$F \times [0,1]$};
\node (bm) at (25,0) {$[0,1]$};
\node (br) at (40,0) {$B$};
\draw[->] (tl) to node[above,font=\small]{$\mathrm{incl}$} (tr);
\draw[->>] (bl) to (bm);
\draw[->] (bm) to node[below,font=\small]{$\gamma$} (br);
\draw[->] (tl) to node[left,font=\footnotesize]{$(-,0)$} (bl);
\draw[->] (tr) to node[right,font=\small]{$f$} (br);
\draw[->,densely dashed] (bl) to (tr);
\end{tikzpicture}
\end{split}
\end{equation}
and define
$\mathrm{mon}_f([\gamma]) = [g(-,1)]$.
\end{defn}

There is also a relative version of this construction. Let $F_0 \subseteq F$ be a subspace and assume that $f$ satisfies the \emph{relative homotopy lifting property} with respect to the pairs of spaces $(F,F_0)$ and $(F,F_0) \times [0,1]$. For example, this holds if $f$ is a Hurewicz fibration, or if $f$ is a Serre fibration and $(F,F_0)$ is a relative CW-complex. Also assume that we have a topological embedding $i \colon F_0 \times B \hookrightarrow E$ such that $f \circ i$ is the projection onto the second factor and $i(-,b)$ is the inclusion $F_0 \subseteq F \subseteq E$. (This says, essentially, that $f$ contains the trivial fibration over $B$ with fibre $F_0$ as a sub-fibration.)

\begin{defn}[\emph{Relative monodromy actions.}]
\label{d:relative-monodromy-action}
Under these assumptions, the \emph{relative monodromy action} associated to $f$ and $F_0$ is the action-up-to-homotopy
\begin{equation}
\label{eq:relative-monodromy-action}
\mathrm{mon}_f \colon \pi_1(B,b) \longrightarrow \pi_0(\mathrm{hAut}(F|F_0))
\end{equation}
constructed as follows. For an element $[\gamma] \in \pi_1(B,b)$ represented by a loop $\gamma \colon [0,1] \to B$, let $g \colon F \times [0,1] \to E$ be a choice of lift in the diagram:
\begin{equation}
\label{eq:relative-monodromy-lifting}
\centering
\begin{split}
\begin{tikzpicture}
[x=1mm,y=1mm]
\node (tl) at (-20,15) {$(F_0 \times [0,1]) \cup (F \times \{0\})$};
\node (tr) at (40,15) {$E$};
\node (bl) at (-20,0) {$F \times [0,1]$};
\node (bm) at (20,0) {$[0,1]$};
\node (br) at (40,0) {$B$};
\draw[->] (tl) to node[above,font=\small]{$(i \circ (\mathrm{id}_{F_0} \times \gamma)) \cup \mathrm{incl}$} (tr);
\draw[->>] (bl) to (bm);
\draw[->] (bm) to node[below,font=\small]{$\gamma$} (br);
\draw[->] (tl) to node[left,font=\small]{$\mathrm{incl}$} (bl);
\draw[->] (tr) to node[right,font=\small]{$f$} (br);
\draw[->,densely dashed] (bl) to (tr);
\end{tikzpicture}
\end{split}
\end{equation}
and define $\mathrm{mon}_f([\gamma]) = [g(-,1)]$.
\end{defn}

\begin{lem}
\label{l:relative-monodromy-action}
The monodromy action \eqref{eq:monodromy-action} and relative monodromy action \eqref{eq:relative-monodromy-action} are well-defined.
\end{lem}
\begin{proof}
For the monodromy action \eqref{eq:monodromy-action}, the proof is given in \cite[Lemma 5.3]{PalmerTillmann2020homologicalstabilityconfigurationsection}. The proof for the relative monodromy action \eqref{eq:relative-monodromy-action} is similar.
\end{proof}

%%%%%%%%%%%%%%%%%%%%%%%%%%%%%%%%%%%%%%%%%%%
%%%%%%%%%%%%%%%%%%%%%%%%%%%%%%%%%%%%%%%%%%%
\section{Point-pushing actions}\label{s:point-pushing-actions}
%%%%%%%%%%%%%%%%%%%%%%%%%%%%%%%%%%%%%%%%%%%
%%%%%%%%%%%%%%%%%%%%%%%%%%%%%%%%%%%%%%%%%%%

This section defines the \emph{point-pushing action} associated to a manifold $M$ and a finite subset $z \subset \mathring{M}$ of its interior. This is given in Definition \ref{d:point-pushing} via the monodromy action of the ``universal'' bundle \eqref{eq:point-pushing-bundle}. This may be refined (Remark \ref{rmk:smooth-point-pushing}) to a smooth version, and it has a simple geometric description (Lemma \ref{l:geometric-point-pushing}) for manifolds of dimension at least $3$. We then describe point-pushing actions on mapping spaces and other spaces associated functorially to the complement $M \smallsetminus z$ (see Definitions \ref{d:associated-point-pushing} and \ref{d:point-pushing-action-mapping}).

\begin{defn}
Let $U_k(M) \coloneqq \bar{C}_{1,k}(M)$ be the configuration space of $k$ unordered green points in the interior of $M$ and one red point in $M$, which may lie on the boundary. There is a fibre bundle
\begin{equation}
\label{eq:point-pushing-bundle}
u \colon U_k(M) \longrightarrow C_k(\mathring{M}),
\end{equation}
given by forgetting the red point, whose fibres are $u^{-1}(z) = M \smallsetminus z$. This is the \emph{universal bundle} referred to in the introduction.
\end{defn}

\begin{defn}[\emph{The point-pushing action.}]
\label{d:point-pushing}
For a manifold-with-boundary $M$ (we allow $\partial M = \varnothing$) and a finite subset $z \subseteq \mathring{M}$ of cardinality $k$, the \emph{point-pushing action} of $\pi_1(C_k(\mathring{M}),z)$ on $M \smallsetminus z$ is defined as the relative monodromy action of \eqref{eq:point-pushing-bundle}. More precisely, we write $F = u^{-1}(z)$, let $F_0 = \partial M \subseteq M \smallsetminus z$ and note that $(M \smallsetminus z,\partial M)$ is a relative CW-complex, since it is a (smooth) manifold with boundary. There is an embedding
\[
i \colon \partial M \times C_k(\mathring{M}) \longhookrightarrow U_k(M),
\]
given by colouring the point in $\partial M$ red and the $k$ points in the interior green, which satisfies the conditions of Definition \ref{d:relative-monodromy-action}. By Definition \ref{d:relative-monodromy-action} and Lemma \ref{l:relative-monodromy-action}, there is therefore a well-defined relative monodromy action
\begin{equation}
\label{eq:point-pushing-definition}
\mathrm{push}_{(M,z)} \colon \pi_1(C_k(\mathring{M}),z) \longrightarrow \pi_0(\mathrm{hAut}(M \smallsetminus z | \partial M)).
\end{equation}
This is, by definition, the \emph{point-pushing action} of $\pi_1(C_k(\mathring{M}),z)$ on $M \smallsetminus z$. For $[\gamma] \in \pi_1(C_k(\mathring{M}),z)$, the homotopy class of maps
\[
\mathrm{push}_\gamma = \mathrm{push}_{(M,z)}([\gamma]) \colon M \smallsetminus z \longrightarrow M \smallsetminus z
\]
(fixing $\partial M$ pointwise) is called the \emph{point-pushing map} of $[\gamma]$ on $M \smallsetminus z$.
\end{defn}

\begin{rmk}
\label{rmk:smooth-point-pushing}
The monodromy action \eqref{eq:point-pushing-definition} may be refined to an action by \emph{isotopy classes of diffeomorphisms}, as in the following diagram:
\begin{equation}
\label{eq:two-pushing-maps}
\centering
\begin{split}
\begin{tikzpicture}
[x=1mm,y=1mm]
\node (tl) at (0,6) {$\pi_1(C_k(\mathring{M}),z)$};
\node (bl) at (0,-6) {$\pi_1(C_k(\mathring{M}),z)$};
\node (tr) at (50,6) {$\pi_0(\mathrm{Diff}_\partial(M,z))$};
\node (br) at (50,-6) {$\pi_0(\mathrm{hAut}(M \smallsetminus z | \partial M)),$};
\draw[->] (tl) to node[above,font=\small]{$\mathrm{push}_{(M,z)}^{\mathsf{sm}}$} (tr);
\draw[->] (bl) to node[below,font=\small]{$\mathrm{push}_{(M,z)}$} (br);
\node at (0,0) {\rotatebox{90}{$=$}};
\draw[->] (tr) to node[right,font=\small]{$i$} (br);
\end{tikzpicture}
\end{split}
\end{equation}
where $\mathrm{Diff}_\partial(M)$ denotes the topological group of diffeomorphisms of $M$ fixing $\partial M$ pointwise, in the smooth Whitney topology, the topology on $\mathrm{hAut}(-)$ is the compact-open topology and $i$ is induced by the continuous injection $\mathrm{Diff}_\partial(M,z) \hookrightarrow \mathrm{hAut}(M \smallsetminus z | \partial M)$ given by $\varphi \mapsto \varphi|_{M \smallsetminus z}$.

To see this, recall (\cf \cite{Palais1960Localtrivialityof,Cerf1961Topologiedecertains,Lima1963localtrivialityof}) that there is a fibre bundle $\mathrm{Diff}_\partial(M) \to C_k(\mathring{M})$ given by evaluating a diffeomorphism at a finite set of points, whose fibre over $z$ is the subgroup $\mathrm{Diff}_\partial(M,z)$ of diffeomorphisms fixing $z$ as a subset. The map $\mathrm{push}_{(M,z)}^{\mathsf{sm}}$ in the diagram above is the connecting homomorphism in the long exact sequence of homotopy groups of this fibre bundle. We call this action the \emph{smooth point-pushing action} of $\pi_1(C_k(\mathring{M}))$ on $M \smallsetminus z$, and we call the map $\mathrm{push}_{\gamma}^{\mathsf{sm}} = \mathrm{push}_{(M,z)}^{\mathsf{sm}}([\gamma]) \colon (M,z) \to (M,z)$ the \emph{smooth point-pushing map} of $[\gamma]$ on $(M,z)$.
\end{rmk}

If $d=\mathrm{dim}(M) \geq 3$, there is a useful geometric description of the smooth point-pushing action, which we will use later. An element $\gamma \in \pi_1(C_k(\mathring{M}),z)$ is represented by a certain number of oriented loops $\gamma_1,\ldots,\gamma_j$ in $M$, each passing through at least one point of $z$, such that, for each point of $z$, exactly one of the loops passes through it. (The number $j \leq k$ of such loops is the number of cycles in the cycle decomposition of the permutation of $z$ induced by $\gamma$.) Choose representatives of the loops $\gamma_1,\ldots,\gamma_j$ that are smoothly embedded and have pairwise disjoint images (using the fact that $d\geq 3$ for disjointness). Also choose pairwise disjoint closed tubular neighbourhoods $T_1,\ldots,T_j$ of these loops, which we assume to be contained in the interior of $M$. Define a diffeomorphism
\[
\varphi_{(T_1,\ldots,T_j)} \colon (M,z) \longrightarrow (M,z)
\]
fixing $\partial M$ pointwise and $z$ setwise as follows. On the complement of the tubular neighbourhoods, $\varphi_{(T_1,\ldots,T_j)}$ is the identity. Suppose that the tubular neighbourhood $T_i$ contains $k_i$ of the points of $z$ (so $k_1 + \cdots + k_j = k$) and identify $T_i \smallsetminus (z \cap T_i)$ with
\[
((D^{d-1} \times \bR) \smallsetminus (\{0\} \times \bZ)) / {\sim},
\]
where $\sim$ is either
\begin{itemizeb}
    \item the equivalence relation given by $(x,t) \sim (x,t+k_i)$, or
    \item the equivalence relation given by $(x,t) \sim (r(x),t+k_i)$, where $r \colon D^{d-1} \to D^{d-1}$ is a fixed reflection in a hyperplane passing through $0$,
\end{itemizeb}
depending on whether or not the loop $\gamma_i$ lifts to a loop in the orientation double cover of $M$. We moreover arrange that this identification restricts to an identification of the cores of these two tubes. Choose a smooth function $\lambda \colon [0,1] \to [0,1]$ that takes the value $1$ on $[0,\epsilon]$ and the value $0$ on $[1-\epsilon,1]$ for some $\epsilon > 0$. Then the restriction of $\varphi_{(T_1,\ldots,T_j)}$ to $T_i$, under this identification, is defined by
\[
\varphi_{(T_1,\ldots,T_j)}(x,t) = (x,t+\lambda(\lvert x \rvert)).
\]
See Figure \ref{fig:point-pushing} for an illustration. We record this geometric description in the following lemma.

\begin{lem}[\emph{Geometric point-pushing.}]
\label{l:geometric-point-pushing}
Let $M$ be a smooth manifold-with-boundary of dimension $d \geq 3$ and let $[\gamma] \in \pi_1(C_k(\mathring{M}),z)$. Choose a collection of smoothly embedded loops $\gamma_1,\ldots,\gamma_j$ and tubular neighbourhoods $T_1,\ldots,T_j$ as described above. Then
\[
[\varphi_{(T_1,\ldots,T_j)}] = \mathrm{push}_{(M,z)}^{\mathsf{sm}}([\gamma]) \in \pi_0(\mathrm{Diff}_\partial(M,z)).
\]
\end{lem}

\begin{figure}[t]
\centering
\labellist
\small\hair 2pt
 \pinlabel {$\mathrm{id}$} [t] at 168 604
 \pinlabel {$0$} [t] at 77 525
 \pinlabel {$4$} [t] at 257 525
 \pinlabel {$T_1$} [tr] at 112 223
 \pinlabel {$T_2$} [tr] at 256 183
 \pinlabel {$r$} [t] at 168 84
 \pinlabel {$0$} [t] at 122 0
 \pinlabel {$2$} [t] at 212 0
 \pinlabel {\footnotesize non-orientable loop} [ ] at 340 225
 \pinlabel {$M$} [ ] at 196 419
\endlabellist
\includegraphics[scale=0.5]{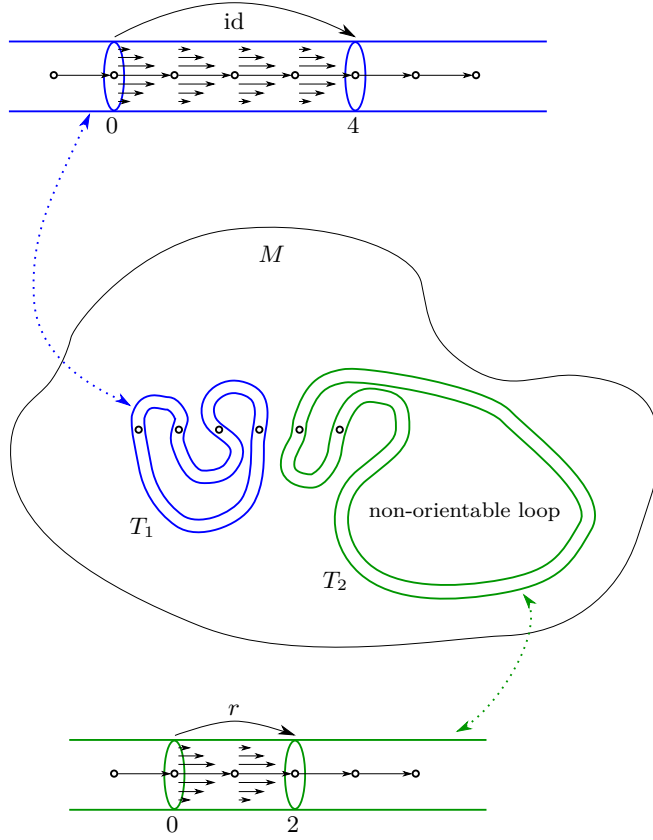}
\caption{An example of the point-pushing action for $\lvert z \rvert = 6$ and where the loop $\gamma \in \pi_1(C_6(\mathring{M}),z)$ induces a permutation of $z$ with one $4$-cycle and one $2$-cycle.}
\label{fig:point-pushing}
\end{figure}

\paragraph{Associated point-pushing actions.}

We have so far described the ``universal'' point-pushing action of $\pi_1(C_k(\mathring{M}),z)$ on the complement $M \smallsetminus z$, for a subset $z \subset \mathring{M}$ with $\lvert z \rvert = k$. We now discuss induced point-pushing actions associated to continuous endofunctors $T \colon \mathrm{Top} \to \mathrm{Top}$ or $T \colon \mathrm{Top}_* \to \mathrm{Top}_*$ (or, more generally, to a continuous functor of the form \eqref{eq:functor-to-Top}).

\begin{defn}[\emph{Associated fibre bundles.}]
\label{d:associated-bundle}
We first recall that, if $f \colon E \to B$ is a fibre bundle with fibre $F$ (and structure group $\mathrm{Homeo}(F)$ in the compact-open topology), and if $T \colon \mathrm{Top} \to \mathrm{Top}$ is a continuous endofunctor (covariant or contravariant) of the topologically-enriched category of spaces, there is an associated fibre bundle
\begin{equation}
\label{eq:induced-fibre-bundle}
f_T \colon T_{\mathrm{fib}}(E) \too B
\end{equation}
with fibre $T(F)$, constructed by ``applying $T$ fibrewise'' to $E$. More precisely, the functor $T$ restricts to a continuous group (anti-)homomorphism
\begin{equation}
\label{eq:action-on-TF}
\mathrm{Homeo}(F) \too \mathrm{Homeo}(T(F)),
\end{equation}
and we define \eqref{eq:induced-fibre-bundle} to be the Borel construction $\mathrm{Prin}(E) \times_{\mathrm{Homeo}(F)} T(F)$, where $\mathrm{Prin}(E) \to B$ is the principal $\mathrm{Homeo}(F)$-bundle associated to $f$, and where $\mathrm{Homeo}(F)$ acts on $T(F)$ via \eqref{eq:action-on-TF}. (See \cite[\S\S 8--9]{Steenrod1951} for more details.)

There is an exactly analogous construction if $f$ is a fibre bundle in the pointed category $\mathrm{Top}_*$ (i.e.~with structure group $\mathrm{Homeo}_*(F)$) and $T \colon \mathrm{Top}_* \to \mathrm{Top}_*$ is a continuous endofunctor of the topologically-enriched category of based spaces $\mathrm{Top}_*$. In this case $T$ restricts to a continuous group homomorphism 
\begin{equation}
\label{eq:action-on-TF-based}
\mathrm{Homeo}_*(F) \too \mathrm{Homeo}_*(T(F)),
\end{equation}
so we may define \eqref{eq:induced-fibre-bundle} to be the Borel construction $\mathrm{Prin}_*(E) \times_{\mathrm{Homeo}_*(F)} T(F)$, where $\mathrm{Prin}_*(E) \to B$ is the principal $\mathrm{Homeo}_*(F)$-bundle associated to $f$, and where $\mathrm{Homeo}_*(F)$ acts on $T(F)$ via \eqref{eq:action-on-TF-based}.
\end{defn}

\begin{defn}[\emph{Configuration-mapping spaces.}]
\label{d:cmap-1}
Let $X$ be any space and consider the (contravariant) continuous functor
\[
T = \mathrm{Map}(-,X) \colon \mathrm{Top} \too \mathrm{Top}.
\]
The fibre bundle associated by $T$ to the bundle \eqref{eq:point-pushing-bundle} is denoted by
\[
\cmap{k}{}{M}{X} \coloneqq T_{\mathrm{fib}}(U_k(M)) \too C_k(\mathring{M}),
\]
and its total space is the $k$-th \emph{configuration-mapping space} of $M$ and $X$. A point in $\cmap{k}{}{M}{X}$ consists of a configuration $z \subset \mathring{M}$ in the interior of $M$ and a continuous map $M \smallsetminus z \to X$.

If $\partial M \neq \varnothing$, the fibre bundle \eqref{eq:point-pushing-bundle} admits a canonical section given by $z \mapsto (z,*)$, where $* \in \partial M$ is a choice of basepoint, allowing us to reduce its structure group to the based homeomorphism group $\mathrm{Homeo}_*(M \smallsetminus z)$, where $z$ is a basepoint of $C_k(\mathring{M})$. Thus, choosing a basepoint for $X$, we may also consider the fibre bundle associated to \eqref{eq:point-pushing-bundle} by the continuous functor $T = \mathrm{Map}_*(-,X) \colon \mathrm{Top}_* \to \mathrm{Top}_*$, which is denoted by
\[
\cmap{k}{*}{M}{X} \coloneqq T_{\mathrm{fib}}(U_k(M)) \too C_k(\mathring{M}).
\]
A point in $\cmap{k}{*}{M}{X}$ consists of a configuration $z \subset \mathring{M}$ in the interior of $M$ together with a \emph{based} continuous map $M \smallsetminus z \to X$.
\end{defn}

\begin{defn}[\emph{Associated fibre bundles, II.}]
\label{d:associated-bundle-2}
The structure group of the bundle \eqref{eq:point-pushing-bundle} may be reduced further to $\mathrm{Homeo}_{\partial M}(M,z)$, the group of self-homeomorphisms of $M$ that fix $z$ setwise and $\partial M$ pointwise. Hence any continuous functor
\begin{equation}
\label{eq:functor-to-Top}
T \colon \mathrm{Homeo}_{\partial M}(M,z) \too \mathrm{Top}
\end{equation}
(i.e., any space with a continuous action of $\mathrm{Homeo}_{\partial M}(M,z)$) associates to \eqref{eq:point-pushing-bundle} a new fibre bundle
\begin{equation}
\label{eq:associated-bundle-2}
T_{\mathrm{fib}}(U_k(M)) \too C_k(\mathring{M})
\end{equation}
by taking the Borel construction of the associated principal $\mathrm{Homeo}_{\partial M}(M,z)$-bundle.
\end{defn}

\begin{rmk}
For comparison, the associated fibre bundles of Definition \ref{d:associated-bundle} above correspond to continuous functors \eqref{eq:functor-to-Top} that are of the form
\[
\mathrm{Homeo}_{\partial M}(M,z) \xrightarrow{\; -|_{M \smallsetminus z} \;} \mathrm{Homeo}(M \smallsetminus z) \subset \mathrm{Top} \too \mathrm{Top},
\]
in other words, that factor through an endofunctor of $\mathrm{Top}$. However, there are interesting (and more subtle) examples that do not extend in this way, as we show in the next example.
\end{rmk}

\begin{defn}[\emph{Configuration-mapping spaces, II.}]
\label{d:cmap-2}
Fix a basepoint $* \in \partial M$, a based space $X$ and a subset $c \subseteq [S^{d-1},X]$ of unbased homotopy classes of maps $S^{d-1} \to X$. If $M$ is non-orientable we assume that $c$ consists of fixed points under the involution of $[S^{d-1},X]$ given by a reflection of $S^{d-1}$. There is a continuous functor
\begin{equation}
\label{eq:map-c-functor}
\mathrm{Map}_*^c(-,X) \colon \mathrm{Homeo}_{\partial M}(M,z) \too \mathrm{Top}
\end{equation}
defined as follows. The unique object on the left-hand side is sent to the space (with the compact-open topology) of based, continuous maps $f \colon M \smallsetminus z \to X$ with ``monodromy'' valued in $c$. The last condition means that, if $e \colon D^d \to M$ is an embedding such that $z \cap e(D^d)$ is a single point in the interior of $e(D^d)$, then the homotopy class of $f \circ e|_{\partial D^d}$ lies in $c$. (If $M$ is orientable, we fix an orientation and require that $e$ is orientation-preserving in the preceding sentence.) One may then check that the natural action of $\varphi \in \mathrm{Homeo}_{\partial M}(M,z)$ on the mapping space $\mathrm{Map}_*(M \smallsetminus z , X)$ preserves the subspace $\mathrm{Map}_*^c(M \smallsetminus z , X)$. The fibre bundle associated by \eqref{eq:map-c-functor} to the bundle \eqref{eq:point-pushing-bundle} is denoted by
\begin{equation}
\label{eq:configuration-mapping-space-charge}
\cmap{k}{c,*}{M}{X} \too C_k(\mathring{M}),
\end{equation}
and its total space is the $k$-th \emph{based configuration-mapping space} of $M$ and $X$ with ``\emph{monodromy}'' or ``\emph{charge}'' in $c$.
\end{defn}

\begin{rmk}
Configuration-mapping spaces are discussed in more detail in \cite[\S 2]{PalmerTillmann2020homologicalstabilityconfigurationsection}, and may be generalised to \emph{configuration-section spaces}, which are defined in \cite[\S 3]{PalmerTillmann2020homologicalstabilityconfigurationsection}. There are also many other natural continuous functors $T \colon \mathrm{Top} \to \mathrm{Top}$ or $T \colon \mathrm{Homeo}_{\partial M}(M,z) \to \mathrm{Top}$ that may be used to construct interesting fibre bundles associated to the ``universal'' bundle \eqref{eq:point-pushing-bundle}. For example, one could take $T$ to be suspension $\Sigma^k (-)$, symmetric powers $SP^k(-)$ or configuration spaces $C_k(-)$, each of which lead to a certain flavour of bicoloured configuration spaces. Other interesting examples are co-representable functors, such as the based and free loop-space functors $\Omega (-)$ and $L(-)$, which lead to spaces of configurations equipped with (based or free) continuous loops in their complement.
\end{rmk}

\begin{defn}[\emph{Associated point-pushing action.}]
\label{d:associated-point-pushing}
For a space $T$ with a continuous action of $\mathrm{Homeo}_{\partial M}(M,z)$, viewed as a continuous functor $T \colon \mathrm{Homeo}_{\partial M}(M,z) \to \mathrm{Top}$, we have from Definition \ref{d:associated-bundle-2} a fibre bundle \eqref{eq:associated-bundle-2}
\[
T_{\mathrm{fib}}(U_k(M)) \too C_k(\mathring{M})
\]
with fibre $T$. The \emph{associated point-pushing action} of $\pi_1(C_k(\mathring{M}),z)$ on $T$ is then the monodromy action of this fibre bundle, denoted by
\begin{equation}
\label{eq:associated-point-pushing}
\mathrm{push}_{(M,z,T)} \colon \pi_1(C_k(\mathring{M}),z) \too \pi_0(\mathrm{hAut}(T)).
\end{equation}
\end{defn}

\begin{defn}[\emph{Point-pushing action on mapping spaces.}]
\label{d:point-pushing-action-mapping}
In particular, if we specialise to the case $T = \mathrm{Map}_*^c(M \smallsetminus z , X)$ for a based space $X$ and a subset $c \subseteq [S^{d-1},X]$, as in Definition \ref{d:cmap-2}, we have an associated point-pushing action
\[
\mathrm{push}_{(M,z,X,c)} \colon \pi_1(C_k(\mathring{M}),z) \too \pi_0(\mathrm{hAut}(\mathrm{Map}_*^c(M \smallsetminus z , X))).
\]
which is the monodromy action of the fibre bundle \eqref{eq:configuration-mapping-space-charge}. This can be generalised to a point-pushing action of $\pi_1(C_k(\mathring{M}),z)$ on $\mathrm{Map}^c((M \smallsetminus z,D),(X,*))$ for any subset $D \subseteq \partial M$.
\end{defn}

The following elementary lemma relates the \emph{point pushing action} of $\pi_1(C_k(\mathring{M}),z)$ on $M \smallsetminus z$ (Definition \ref{d:point-pushing}) and its \emph{associated point-pushing action} on the mapping space $\mathrm{Map}^c((M \smallsetminus z,D),(X,*))$ (Definition \ref{d:point-pushing-action-mapping}). Choose $k$ pairwise disjoint balls in $M$ centred at the points $z$ and let
\[
s \colon S^{d-1} \times \{1,\ldots,k\} \lhook\joinrel\longrightarrow M \smallsetminus z
\]
be the inclusion of their boundaries. Denote by $\mathrm{hAut}^s(M \smallsetminus z | \partial M) \subseteq \mathrm{hAut}(M \smallsetminus z | \partial M)$ the subspace of homotopy automorphisms $f$ of $M\smallsetminus z$ such that $f \circ s \simeq s \circ g$ for some homotopy automorphism $g$ of $S^{d-1} \times \{1,\ldots,k\}$. Note that the point-pushing action \eqref{eq:point-pushing-definition} takes values in $\pi_0$ of this subspace.

\begin{lem}
\label{l:associated-point-pushing-comparison}
The point-pushing action of $\pi_1(C_k(\mathring{M}),z)$ on $\mathrm{Map}^c((M \smallsetminus z,D),(X,*))$ is obtained from its point-pushing action on $M \smallsetminus z$ by pre-composition. In other words, the following diagram commutes:
\begin{equation}
\label{eq:two-pushing-maps2}
\centering
\begin{split}
\begin{tikzpicture}
[x=1mm,y=1mm]
\node (tl) at (0,6) {$\pi_1(C_k(\mathring{M}),z)$};
\node (bl) at (0,-6) {$\pi_1(C_k(\mathring{M}),z)$};
\node (tr) at (70,6) {$\pi_0(\mathrm{hAut}^s(M \smallsetminus z | \partial M))$};
\node (br) at (70,-6) {$\pi_0\bigl(\mathrm{hAut}\bigl(\mathrm{Map}^c((M \smallsetminus z,D),(X,*))\bigr)\bigr),$};
\draw[->] (tl) to node[above,font=\small]{$\mathrm{push}_{(M,z)}$} (tr);
\draw[->] (bl) to node[below,font=\small]{$\mathrm{push}_{(M,z,X,c)}$} (br);
\node at (0,0) {\rotatebox{90}{$=$}};
\draw[->] (tr) to node[right,font=\small]{$\circ$} (br);
\end{tikzpicture}
\end{split}
\end{equation}
where the right vertical homomorphism $\circ$ is defined by composition. In particular, the action up to homotopy of $\pi_0(\mathrm{hAut}^s(M \smallsetminus z | \partial M))$ on the mapping space $\mathrm{Map}((M \smallsetminus z,D),(X,*))$ preserves the subspace $\mathrm{Map}^c((M \smallsetminus z,D),(X,*))$ for each subset $c \subseteq [S^{d-1},X]$, assuming, if $M$ is non-orientable, that $c$ is closed under the involution given by reflecting in $S^{d-1}$.
\end{lem}

\begin{rmk}
We have focused in this section (except in Remark \ref{rmk:smooth-point-pushing} and Lemma \ref{l:geometric-point-pushing}) on monodromy actions -- by \emph{homotopy automorphisms} -- of fibrations (as discussed abstractly in \S\ref{s:monodromy-actions}). This is because our main result is an explicit description of the monodromy action \emph{by homotopy automorphisms} of the universal bundle \eqref{eq:point-pushing-bundle} (and, as a corollary, of the configuration-mapping bundle \eqref{eq:configuration-mapping-space-charge}). However, the constructions of this section also have direct analogues for monodromy actions by homeomorphisms (diffeomorphisms) of fibre bundles (smooth fibre bundles). See also \S\ref{s:kernel-point-pushing}, where we discuss kernels of point-pushing actions in all three settings.
\end{rmk}

%%%%%%%%%%%%%%%%%%%%%%%%%%%%%%%%%%%%%%%%%%%
%%%%%%%%%%%%%%%%%%%%%%%%%%%%%%%%%%%%%%%%%%%
\section{Formulas for point-pushing actions}\label{s:formulas}
%%%%%%%%%%%%%%%%%%%%%%%%%%%%%%%%%%%%%%%%%%%
%%%%%%%%%%%%%%%%%%%%%%%%%%%%%%%%%%%%%%%%%%%

Let $M$ be a connected manifold of dimension $d \geq 3$, let $z \subset \mathring{M}$ be a $k$-point configuration in its interior, $D \subseteq \partial M$ an embedded $(d-1)$-dimensional disc in its boundary, $X$ a based space and $c \subseteq [S^{d-1},X]$ a non-empty set of unbased homotopy classes of maps $S^{d-1} \to X$. Our goal is to give explicit formulas for the point-pushing action of $\pi_1(C_k(\mathring{M}),z)$ on $M \smallsetminus z$ (Definition \ref{d:point-pushing}). These will be given in the following two sections; in this section we first fix notation and the identifications that we will use.

\begin{notation}
Let $W_k$ denote a wedge $\bigvee^k S^{d-1}$ of $k$ copies of the $(d-1)$-sphere.
\end{notation}

\begin{construction}
\label{construction:equivalence-of-pairs}
Let us choose an explicit homotopy equivalence of pairs
\begin{equation}
\label{eq:equivalence-of-pairs}
(M \smallsetminus z , D) \;\simeq\; ( M \vee W_k , * ),
\end{equation}
as follows (see Figure \ref{fig:decomposition} for an illustration). Choose a $d$-dimensional closed disc $B$ in $M$ containing the configuration $z$ in its interior and such that $B \cap \partial M$ is a $(d-1)$-dimensional disc in $\partial M$ containing (but not equal to) $D$, and such that the closure of the complement $(B \cap \partial M) \smallsetminus D$ is also a disc. (In Figure \ref{fig:decomposition}, we may assume that $D = \partial M \cap B'$.) Note that the closure $M'$ of $M \smallsetminus B$ in $M$ is also homeomorphic to $M$. Also note that we have $M' \cap (B \cap \partial M) = \partial (B \cap \partial M)$ and $D \cap \partial (B \cap \partial M) \neq \varnothing$ by the condition that the closure of the complement $(B \cap \partial M) \smallsetminus D$ is a disc. Thus $D \cap M' \neq \varnothing$, so we may choose a basepoint $*$ of $M$ in $D \cap M'$. Choose also $k$ embedded $(d-1)$-spheres in $B$ such that each sphere intersects $\partial B$ at the basepoint $*$ and nowhere else, the spheres are pairwise disjoint except for $*$ and each sphere ``wraps once around each of the points of $z$'' (this is more formally expressed by the condition that $B \smallsetminus z$ must deformation retract onto the union of the spheres). The union of $M'$ and the spheres is homeomorphic to the wedge sum on the right-hand side of \eqref{eq:equivalence-of-pairs}, and there is a deformation retraction of $M \smallsetminus z$ onto this subspace, supported in $B \smallsetminus z$, fixing the basepoint $*$ and sending $D$ onto $\{*\}$.
\end{construction}

\begin{figure}[t]
\centering
\begin{tikzpicture}
[x=1mm,y=1mm]
\draw[black!30,line width=2mm] (11,11) to[out=20,in=180] (25,14) .. controls (35,14) and (40,20) .. (40,30) arc (180:0:10) -- (60,17) arc (0:-90:4) -- (50,13) .. controls (35,13) and (35,4) .. (30,4) .. controls (25,4) and (25,6) .. (20,6) to[out=180,in=-20] (11,9);
\node at (59.5,20.5) [anchor=east,font=\small,black!80] {$T$};
\fill[blue!10] (20,15) -- (20,50) -- (10,50) arc (90:180:10) -- (0,15) -- cycle;
\fill[blue!20] (20,0) -- (20,15) -- (0,15) -- (0,10) arc (180:270:10) -- cycle;
\draw[blue!50] (0,15) -- (20,15);
\fill[green!20,opacity=0.5] (20,0) -- (70,0) arc (-90:0:10) -- (80,40) arc (0:90:10) -- (20,50) -- cycle;
\draw (20,0) -- (70,0) arc (-90:0:10) -- (80,40) arc (0:90:10) -- (10,50) arc (90:180:10) -- (0,10) arc (180:270:10) -- cycle;
\draw (20,0) -- (20,50);
\node at (10,10) [draw,circle,inner sep=0.5mm,fill=white] {};
\node at (10,20) [draw,circle,inner sep=0.5mm,fill=white] {};
\node at (10,30) [draw,circle,inner sep=0.5mm,fill=white] {};
\node at (10,40) [draw,circle,inner sep=0.5mm,fill=white] {};
\draw (20,0) .. controls (20,5) and (19,5) .. (19,10) -- (19,40) arc (0:90:3) -- (10,43) arc (90:270:3) -- (15,37) arc (90:0:3) -- (18,10) .. controls (18,5) and (20,5) .. (20,0);
\draw (20,0) .. controls (20,5) and (17,5) .. (17,10) -- (17,30) arc (0:90:3) -- (10,33) arc (90:270:3) -- (13,27) arc (90:0:3) -- (16,10) .. controls (16,5) and (20,5) .. (20,0);
\draw (20,0) .. controls (20,5) and (15,5) .. (15,10) -- (15,20) arc (0:90:3) -- (10,23) arc (90:270:3) -- (11,17) arc (90:0:3) -- (14,10) .. controls (14,5) and (20,5) .. (20,0);
\draw (20,0) .. controls (20,5) and (13,5) .. (13,10) arc (0:270:3) .. controls (13,7) and (20,3) .. (20,0);
\node (b) at (10,-9) [fill=blue!10,inner sep=2mm,draw=black] {};
\node (bb) at (40,-9) [fill=blue!20,inner sep=2mm,draw=black] {};
\node (g) at (65,-9) [fill=green!10,inner sep=2mm,draw=black] {};
\node (b1) at (b.east) [anchor=west] {$\cup$};
\node (b2) at (b1.east) [anchor=west,fill=blue!20,inner sep=2mm,draw=black] {};
\node (b3) at (b2.east) [anchor=west] {$=B$};
\node at (bb.east) [anchor=west] {$= B'$};
\node at (g.east) [anchor=west] {$= M'$};
\node (basepoint) at (20,0) [fill=black,inner sep=0.5mm] {};
\node at (basepoint.south) [anchor=north] {$*$};
\node at (80,25) [anchor=west] {$=M$};
\node at (0,25) [anchor=east] {$\phantom{M=}$};
\draw[green!50!black] (11,11) to[out=20,in=180] (25,14) .. controls (35,14) and (40,20) .. (40,30) arc (180:0:10) -- (60,17) arc (0:-90:4) -- (50,13) .. controls (35,13) and (35,4) .. (30,4) .. controls (25,4) and (25,6) .. (20,6) to[out=180,in=-20] (11,9);
\draw[->,green!50!black] (60,27);
\node at (60.5,18.5) [anchor=west,font=\small,green!50!black] {$\delta$};
\end{tikzpicture}
\caption{An embedding of $M \vee ( \textstyle\bigvee^k S^{d-1} )$ into $M \smallsetminus z$ as a deformation retract, together with a loop $\delta$ in $B' \cup M'$ based at $z \cap B'$ and a tubular neighbourhood $T$ of its intersection with $M'$. The disc $D \subseteq \partial M$ is the intersection $\partial M \cap B'$.}
\label{fig:decomposition}
\end{figure}
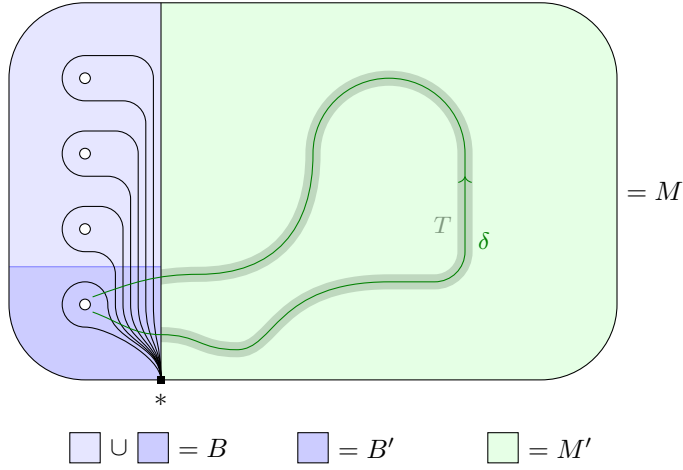

\begin{notation}
From now on, we will write $\pi_1(C_k(\mathring{M}),z)$ just as $\pi_1(C_k(M))$, leaving the basepoint $z$ implicit, and using the fact that the inclusion $C_k(\mathring{M}) \hookrightarrow C_k(M)$ is a homotopy equivalence.
\end{notation}

\begin{notation}
By the smooth version of the point-pushing action (see Remark \ref{rmk:smooth-point-pushing}), an element $\gamma \in \pi_1(C_k(M))$ induces (an isotopy class of) a self-diffeomorphism $\mathrm{push}_{\gamma}^{\mathsf{sm}} \colon M \to M$, fixing $\partial M$ pointwise and $z$ setwise, which has an explicit geometric representative $\varphi_{(T_1,\ldots,T_j)}$ given by Lemma \ref{l:geometric-point-pushing} if $\mathrm{dim}(M) \geq 3$. We denote its restriction to a self-diffeomorphism of $M \smallsetminus z$ by
\[
\pi_\gamma \colon M \smallsetminus z \longrightarrow M \smallsetminus z.
\]
By abuse of notation, we also denote by $\pi_\gamma$ the (homotopy class of a) homotopy self-equivalence of $M \vee W_k$ fixing $*$ induced via the deformation retraction \eqref{construction:equivalence-of-pairs}:
\begin{equation}
\label{eq:pi-gamma}
\centering
\begin{split}
\begin{tikzpicture}
[x=1mm,y=1mm]
\node (tl) at (0,15) {$M \smallsetminus z$};
\node (tr) at (40,15) {$M \smallsetminus z$};
\node (bl) at (0,0) {$M \vee W_k$};
\node (br) at (40,0) {$M \vee W_k .$};
\draw[->] (tl) to node[above,font=\small]{$\pi_\gamma$} (tr);
\draw[->] (bl) to node[above,font=\small]{$\pi_\gamma$} (br);
\draw[->] (bl) to node[right,font=\small]{$\simeq$} node[left,font=\small]{$\mathrm{incl}$} (tl);
\draw[->] (tr) to node[left,font=\small]{$\simeq$} node[right,font=\small]{\eqref{construction:equivalence-of-pairs}} (br);
\end{tikzpicture}
\end{split}
\end{equation}
\end{notation}

\phantomsection
\label{description-of-braid-decomposition}
Recall that, for $\mathrm{dim}(M) \geq 3$, the fundamental group $\pi_1(C_k(M))$ decomposes as the semi-direct product $\pi_1(M)^k \rtimes \Sigma_k$. (See \cite[Theorem~9]{FadellNeuwirth1962Configurationspaces}, \cite[Theorem~1]{Birman1969Onbraidgroups}, or \cite[Lemma 4.1]{Tillmann2016Homologystabilitysymmetric} for a generalisation.) Concretely, the isomorphism 
\begin{equation}
\label{equation:wreath-product}
\Upsilon \colon \pi_1(C_k(M),z) \cong \pi_1(M,z_0)^k \rtimes \Sigma_k
\end{equation}
is given as follows, and depends on the choice of a contractible ball $B$ containing the base configuration $z$ and the point $z_0$. Any loop $\gamma$ of $k$-point configurations in $M$ based at $z$ consists of an ordered tuple of paths in $M$ given by the motions of the individual points. (The paths are ordered according to their ordering at time $t=0$, and this is determined by a fixed ordering of the base configuration $z$.) If we collapse $B$ to a point, we obtain a $k$-tuple $(\alpha_1,\ldots,\alpha_k)$ of based loops in $M/B$. Together with the permutation $\sigma(\gamma)$ of $z \cong \{1,\ldots,k\}$ induced by $\gamma$, and using the isomorphism $\pi_1(M) \cong \pi_1(M/B)$ induced by the collapse map $M \to M/B$, this determines an element $\Upsilon(\gamma) = (\alpha_1,\ldots,\alpha_k;\sigma(\gamma))$ of the semidirect product $\pi_1(M,z_0)^k \rtimes \Sigma_k$.

In the next two sections we give explicit formulas for the bottom horizontal map of \eqref{eq:pi-gamma} for $\gamma = (\alpha_1,\ldots,\alpha_k;\sigma) \in \pi_1(M)^k \rtimes \Sigma_k$ under this decomposition.

\begin{notation}
\label{notation}
We collect here some additional notation that will be used in the following two sections.

\begin{itemizeb}
\item For a wedge $A \vee B$, we write $\mathrm{inc}_A$ (resp.\ $\mathrm{inc}_B$) for the inclusion of the first (resp.\ second) summand, and similarly we write $\mathrm{pr}_A$ (resp.\ $\mathrm{pr}_B$) for the projection onto the first (resp.\ second) summand.
\item For pointed spaces $A,B,C$ and a pointed map $f \colon A \vee B \to C$, we will sometimes write $f$ as a $(1 \times 2)$-matrix:
\[
f = \left( \begin{array}{c|c}
f_A & f_B
\end{array} \right) ,
\]
where $f_A = f \circ \mathrm{inc}_A$ and $f_B = f \circ \mathrm{inc}_B$. Note that $f_A$ and $f_B$ jointly determine $f$, since $\vee$ is the coproduct in the category of pointed spaces.
\item Similarly, for pointed spaces $A,B,C,D$ and a pointed map $f \colon A \vee B \to C \vee D$, we will sometimes write $f$ as a $(2 \times 2)$-matrix:
\[
f = \left( \begin{array}{c|c}
f_A & f_B
\end{array} \right) \;\rightsquigarrow\; \left( \begin{array}{c|c}
{}_C f_A & {}_C f_B \\ \hline
{}_D f_A & {}_D f_B
\end{array} \right) ,
\]
where ${}_C f_A = \mathrm{pr}_C \circ f \circ \mathrm{inc}_A$, etc. Note that the pair of ${}_C f_A$ and ${}_D f_A$ does \emph{not} determine $f_A$ (since $\vee$ is not a product), so the $(2 \times 2)$-matrix-notation loses information. (This is why we write ``$\rightsquigarrow$'' instead of ``$=$'' in this case.)
\item As mentioned above, we have for $\mathrm{dim}(M) \geq 3$ a splitting $\pi_1(C_k(M)) \cong \pi_1(M)^k \rtimes \Sigma_k$. Thus, for each $\sigma \in \Sigma_k$ and $\alpha \in \pi_1(M)$, we have elements
\[
(1,\ldots,1;\sigma) \text{ and } (\alpha,1,\ldots,1;\mathrm{id}) \in \pi_1(C_k(M)),
\]
which we will denote simply by $\sigma$ and $\alpha$ by abuse of notation. We will always use these letters for elements of these two subgroups of $\pi_1(C_k(M))$, and we will denote a general element of $\pi_1(C_k(M))$ by $\gamma$.
\item We take the basepoint of $S^{d-1}$ to be the south pole, and write
\[
\mathrm{pinch} \colon S^{d-1} \longrightarrow S^{d-1} \vee S^{d-1}
\]
for the map that collapses the equator of $S^{d-1}$ to a point. The wedge sum on the right-hand side identifies the north pole of the left summand with the south pole of the right summand. We take the basepoint of $S^{d-1} \vee S^{d-1}$ to be the south pole of the left summand (in particular, \emph{not} the point at which the wedge sum is taken); with this choice, $\mathrm{pinch}$ is a based map.
\item We write
\[
\mathrm{coll} \colon S^{d-1} \longrightarrow [0,1]
\]
for the ``collapse'' map that projects $S^{d-1} \subset \bR^d$ onto the $d$-th coordinate (so the south pole goes to $-1$ and the north pole goes to $1$) and then linearly reparametrises by $x \mapsto \frac12 (x+1)$.
\end{itemizeb}
\end{notation}

\begin{rmk}
Since $\pi_1(C_k(M))$ is generated by elements of the form $(1,\ldots,1;\sigma)$ and $(\alpha,1,\ldots,1;\mathrm{id})$ (which we henceforth denote simply by $\sigma$ and $\alpha$) for $\sigma \in \Sigma_k$ and $\alpha \in \pi_1(M)$, it will suffice to give explicit formulas for
\[
\pi_\sigma \text{ and } \pi_\alpha \colon M \vee W_k \longrightarrow M \vee W_k
\]
up to basepoint-preserving homotopy, for all $\sigma \in \Sigma_k$ and $\alpha \in \pi_1(M)$. This will be done in sections \ref{s:symmetric} and \ref{s:loop} respectively.
\end{rmk}

\begin{terminology}
The elements $\sigma = (1,\ldots,1;\sigma)$ will be called \emph{symmetric generators} of $\pi_1(C_k(M))$ and the elements $\alpha = (\alpha,1,\ldots,1;\mathrm{id})$ will be called \emph{loop generators} of $\pi_1(C_k(M))$.
\end{terminology}

%%%%%%%%%%%%%%%%%%%%%%%%%%%%%%%%%%%%%%%%%%%
%%%%%%%%%%%%%%%%%%%%%%%%%%%%%%%%%%%%%%%%%%%
\section{Symmetric generators}\label{s:symmetric}
%%%%%%%%%%%%%%%%%%%%%%%%%%%%%%%%%%%%%%%%%%%
%%%%%%%%%%%%%%%%%%%%%%%%%%%%%%%%%%%%%%%%%%%

The action of the \emph{symmetric generators} of $\pi_1(C_k(M))$ on $M \vee W_k$ is fairly easy to describe.

\begin{prop}
\label{p:formula-sigma}
For any element $\sigma \in \Sigma_k$ we have
\begin{equation}
\label{eq:formula-sigma}
\pi_\sigma = \mathrm{id}_M \vee \sigma_\sharp = \left( \begin{array}{c|c}
\mathrm{inc}_M & \mathrm{inc}_{W_k} \circ \sigma_\sharp
\end{array} \right) \;\rightsquigarrow\; \left( \begin{array}{c|c}
\mathrm{id}_M & * \\ \hline
* & \sigma_\sharp
\end{array} \right) ,
\end{equation}
where $\sigma_\sharp$ denotes the obvious self-map of $W_k = \bigvee^k S^{d-1}$ determined by the permutation $\sigma$, and $*$ denotes the constant map to the basepoint.
\end{prop}
\begin{proof}
In the geometric model $\varphi_{(T_1,\ldots,T_j)}$ (see Lemma \ref{l:geometric-point-pushing}) for the point-pushing diffeomorphism of $(M,z)$ induced by $\gamma = (1,\ldots,1;\sigma)$, we may assume that the tubular neighbourhoods $T_1,\ldots,T_j$ are all contained in the codimension-zero ball $B \subset M$ (see Figure \ref{fig:decomposition}). This follows from the concrete description of the isomorphism $\eqref{equation:wreath-product} \colon \pi_1(C_k(M)) \cong \pi_1(M)^k \rtimes \Sigma_k$ given on page \pageref{description-of-braid-decomposition}. Since $\varphi_{(T_1,\ldots,T_j)}$ is the identity outside of the tubular neighbourhoods, this implies that $\pi_\sigma \simeq \mathrm{id}_M \vee \psi$, for some automorphism $\psi$ of $W_k$.

To see that $\psi \simeq \sigma_\sharp$, first consider a collection of $k$ small, unbased $(d-1)$-spheres surrounding the points of $z$, contained in the union of tubular neighbourhoods $T_1 \cup \cdots \cup T_j$. It follows from its explicit description in Lemma \ref{l:geometric-point-pushing} that $\varphi_{(T_1,\ldots,T_j)}$ permutes the homotopy classes of these spheres according to $\sigma$. Since these spheres form a free basis for the the homology group $H_{d-1}(B \smallsetminus z) \cong \bZ^k$, the effect of $\varphi_{(T_1,\ldots,T_j)}$ on $H_{d-1}(M\smallsetminus z) \cong H_{d-1}(B\smallsetminus z) \oplus H_{d-1}(M')$ is to permute the $k$ different $\bZ$ factors of $H_{d-1}(B\smallsetminus z)$ according to $\sigma$. Identifying $W_k$ with the wedge of embedded $(d-1)$-spheres in Figure \ref{fig:decomposition}, we have a canonical isomorphism $H_{d-1}(W_k) \cong H_{d-1}(B\smallsetminus z) \cong \bZ^k$. It follows that the effect of $\psi$ on $H_{d-1}(W_k)$ is to permute the $k$ factors of $H_{d-1}(W_k) \cong \bZ^k$ according to $\sigma$. By the Hurewicz theorem, we have $H_{d-1}(W_k) \cong \pi_{d-1}(W_k)$. Since $W_k$ is a wedge of spheres, $\psi \colon W_k \to W_k$ is determined up to based homotopy by its effect on $\pi_{d-1}(W_k)$; thus $\psi \simeq \sigma_\sharp$.
\end{proof}

%%%%%%%%%%%%%%%%%%%%%%%%%%%%%%%%%%%%%%%%%%%
%%%%%%%%%%%%%%%%%%%%%%%%%%%%%%%%%%%%%%%%%%%
\section{Loop generators}\label{s:loop}
%%%%%%%%%%%%%%%%%%%%%%%%%%%%%%%%%%%%%%%%%%%
%%%%%%%%%%%%%%%%%%%%%%%%%%%%%%%%%%%%%%%%%%%

For any $\alpha \in \pi_1(M,*)$, the point-pushing map $\pi_\alpha \colon M \smallsetminus z \to M \smallsetminus z$ may be assumed (up to basepoint-preserving homotopy) to be supported in a tubular neighbourhood of an embedded loop $\alpha'$ in $M$, based at one of the points of the configuration $z$, in the homotopy class determined by conjugating $\alpha$ with a path in $B$ from $*$ to this point (see Figure \ref{fig:decomposition}). We may choose $\alpha'$ and its tubular neighbourhood $T$ to be contained in $M' \cup B'$, so the support of $\pi_\alpha \colon M \smallsetminus z \to M \smallsetminus z$ is contained in $M' \cup B'$. Under the identification \eqref{eq:equivalence-of-pairs}, this implies the following.

\begin{lem}
\label{l:reduction-to-one-sphere}
For any $\alpha \in \pi_1(M)$, up to based homotopy, $\pi_\alpha \colon M \vee W_k \to M \vee W_k$ is of the form
\[
\pi_\alpha = \bar{\pi}_\alpha \vee \mathrm{id}_{W_{k-1}},
\]
where $\bar{\pi}_\alpha$ is a self-map of $M \vee S^{d-1}$, unique up to based homotopy.
\end{lem}

We therefore just have to describe the map $\bar{\pi}_\alpha$ for each $\alpha \in \pi_1(M)$. We first do this under an additional assumption on the manifold $M$. Recall that the \emph{handle-dimension} of a manifold is the smallest $i$ such that $M$ may be constructed using handles of degree at most $i$. Using the cores of such a handle decomposition, this implies that $M$ deformation retracts onto an embedded CW-complex of dimension equal to the handle dimension of $M$. Since $M$, in our situation, is connected and has non-empty boundary, its handle-dimension is necessarily at most $\mathrm{dim}(M) - 1$.

\begin{prop}
\label{p:formula-alpha-1}
Suppose that the handle dimension of $M$ is at most $\mathrm{dim}(M)-2$. Then, for any element $\alpha \in \pi_1(M)$ we have
\begin{equation}
\label{eq:formula-alpha-1}
\bar{\pi}_\alpha = \left( \begin{array}{c|c}
\mathrm{inc}_M & ((\alpha \circ \mathrm{coll}) \vee \mathrm{sgn}(\alpha)) \circ \mathrm{pinch}
\end{array} \right) \;\rightsquigarrow\; \left( \begin{array}{c|c}
\mathrm{id}_M & \alpha \circ \mathrm{coll} \simeq * \\ \hline
* & \mathrm{sgn}(\alpha)
\end{array} \right) ,
\end{equation}
where $\mathrm{sgn}(\alpha) \colon S^{d-1} \to S^{d-1}$ has degree $+1$ if $\alpha$ lifts to a loop in the orientation double cover of $M$ and degree $-1$ otherwise. The other notation is explained in Notation \ref{notation}.
\end{prop}

If the handle dimension of $M$ is equal to $\mathrm{dim}(M) - 1$ (the maximum possible), the formula for $\bar{\pi}_\alpha$ is more complicated. The following proposition gives the general formula.

\begin{prop}
\label{p:formula-alpha-2}
For any element $\alpha \in \pi_1(M)$ we have
\begin{equation}
\label{eq:formula-alpha-2}
\bar{\pi}_\alpha = \left( \begin{array}{c|c}
\overline{\pitchfork}_\alpha & ((\alpha \circ \mathrm{coll}) \vee \mathrm{sgn}(\alpha)) \circ \mathrm{pinch}
\end{array} \right) \;\rightsquigarrow\; \left( \begin{array}{c|c}
\mathrm{id}_M & \alpha \circ \mathrm{coll} \simeq * \\ \hline
\pitchfork_\alpha & \mathrm{sgn}(\alpha)
\end{array} \right) ,
\end{equation}
where $\mathrm{sgn}(\alpha)$ is as in Proposition \ref{p:formula-alpha-1} and the maps $\overline{\pitchfork}_\alpha$ and $\pitchfork_\alpha$ are described in \S\ref{subs:loopII} below.
\end{prop}

In \S\ref{subs:loopI} we prove Proposition \ref{p:formula-alpha-1}. In \S\ref{subs:loopII} we first define the maps $\overline{\pitchfork}_\alpha$ and $\pitchfork_\alpha$ in the statement of Proposition \ref{p:formula-alpha-2} (Definitions \ref{defn:pf1} and \ref{defn:pf2}) and then prove Proposition \ref{p:formula-alpha-2}.

In each case we prove the descriptions on the left-hand side of \eqref{eq:formula-alpha-1} and of \eqref{eq:formula-alpha-2}, and those on the right-hand side in terms of $(2 \times 2)$ matrices follow as a consequence. We note that in each case the top-right entry of the matrix is \emph{a priori} equal to $\alpha \circ \mathrm{coll} \colon S^{d-1} \to M$, but this is nullhomotopic as a based map, so it may be replaced with $*$. In contrast, the appearance of $\alpha \circ \mathrm{coll}$ in the formulas on the left-hand side of \eqref{eq:formula-alpha-1} and of \eqref{eq:formula-alpha-2} may not be replaced by $*$, since it is part of a description of a map $S^{d-1} \to M \vee S^{d-1}$ where the sphere is first collapsed to $[0,1] \vee S^{d-1}$, so in this case the interval may not be deformation retracted to its basepoint $0$, since its other endpoint $1$ is attached to the sphere $S^{d-1}$, which is wrapped with sign $\pm 1$ around the $S^{d-1}$ summand of $M \vee S^{d-1}$.

\subsection{Below the maximal handle dimension.}\label{subs:loopI}

In this subsection we prove Proposition \ref{p:formula-alpha-1}. Let us write
\begin{itemizeb}
\item $\bar{\pi}^M_\alpha \colon M \to M \vee S^{d-1}$ for the restriction of $\bar{\pi}_\alpha$ to the $M$ summand of $M \vee S^{d-1}$;
\item $\bar{\pi}^S_\alpha \colon S^{d-1} \to M \vee S^{d-1}$ for the restriction of $\bar{\pi}_\alpha$ to the $S^{d-1}$ summand of $M \vee S^{d-1}$.
\end{itemizeb}
In this notation, to prove Proposition \ref{p:formula-alpha-1}, we need to show that
\begin{equation}
\label{eq:two-formulas}
\bar{\pi}^M_\alpha \simeq \mathrm{inc}_M \qquad\text{and}\qquad \bar{\pi}^S_\alpha \simeq ((\alpha \circ \mathrm{coll}) \vee \mathrm{sgn}(\alpha)) \circ \mathrm{pinch}.
\end{equation}

We first prove the right-hand side of \eqref{eq:two-formulas}. This may in fact be seen purely geometrically from Figure \ref{fig:decomposition}. We need to describe the effect of $\pi_\alpha$ on the loop (representing a $(d-1)$-sphere) pictured in the bottom-left corner of that figure. As mentioned at the beginning of this section, $\pi_\alpha$ may be assumed to be supported in a tubular neighbourhood $T$ of a loop based at the puncture $z \cap B'$ and supported in $M' \cup B'$, as pictured in Figure \ref{fig:decomposition}. To see the effect of point-pushing along the tube $T$ on the $(d-1)$-sphere based at $*$ pictured in the figure, it is easier first to replace it, up to homotopy equivalence, by a $(d-1)$-sphere encircling the puncture $z \cap B'$ together with a ``tether'' connecting this sphere to the basepoint $*$ (this corresponds to the pinch and collapse maps in the formula \eqref{eq:two-formulas}). Point-pushing along $T$ has the effect on the tether of sending it around a loop homotopic to $\alpha$. On the $(d-1)$-sphere encircling the puncture, it acts by a map of degree $\pm 1$ depending on whether the tubular neighbourhood $T$ is orientable or not, in other words, whether or not $\alpha$ lifts to a loop in the orientation double cover of $M$, which is exactly $\mathrm{sgn}(\alpha)$. Putting this all together, we obtain the desired formula on the right-hand side of \eqref{eq:two-formulas}.

We prove the left-hand side of \eqref{eq:two-formulas} in two steps:
\begin{itemizeb}
\item $\bar{\pi}^M_\alpha \simeq \mathrm{inc}_M \circ \theta_\alpha$ for some self-map $\theta_\alpha \colon M \to M$;
\item $\theta_\alpha \simeq \mathrm{id}_M$.
\end{itemizeb}

Since the handle dimension of $M$ is at most $d-2$, there is an embedded CW-complex $K \subset M$ of dimension at most $d-2$, such that $M$ deformation retracts onto $K$. (Constructed, for example, using the cores of a handle decomposition of $M$ with handles of index at most $d-2$.) The restriction of $\bar{\pi}^M_\alpha$ to $K$ is a map of the form
\[
K \longrightarrow M \vee S^{d-1}.
\]
Choose a CW-complex structure on $M$ extending that of $K$ and give $S^{d-1}$ the unique CW-complex structure with a single $0$-cell and a single $(d-1)$-cell. With respect to these choices, we may homotope the map above to be \emph{cellular}, so that every $r$-cell of $K$ is mapped into a cell of dimension at most $r$. This implies that the image of the map must intersect $S^{d-1}$ only in the basepoint, so we have a factorisation up to homotopy
\[
\bar{\pi}^M_\alpha|_K \colon K \longrightarrow M \lhook\joinrel\longrightarrow M \vee S^{d-1},
\]
for some map $K \to M$. Since the inclusion of $K$ into $M$ is a homotopy equivalence, this implies also that $\bar{\pi}^M_\alpha$ itself factorises up to homotopy as a self-map $\theta_\alpha$ of $M$ followed by the inclusion into $M \vee S^{d-1}$. This establishes the first claim above.

We next have to prove that $\theta_\alpha$ is homotopic to the identity. Consider the following diagram.
\begin{equation}
\label{eq:proof-thetaalpha}
\centering
\begin{split}
\begin{tikzpicture}
[x=1mm,y=1mm]
\node (tl) at (0,30) {$M$};
\node (tr) at (40,30) {$M$};
\node (ml) at (0,15) {$M \vee S^{d-1}$};
\node (mr) at (40,15) {$M \vee S^{d-1}$};
\node (bl) at (0,0) {$M$};
\node (br) at (40,0) {$M$};
\incl{(tl)}{(ml)}
\incl{(tr)}{(mr)}
\incl{(ml)}{(bl)}
\incl{(mr)}{(br)}
\draw[->] (tl) to node[above,font=\small]{$\theta_\alpha$} (tr);
\draw[->] (ml) to node[above,font=\small]{$\bar{\pi}_\alpha$} (mr);
\draw[->] (bl) to node[above,font=\small]{$\mathrm{id}$} (br);
\end{tikzpicture}
\end{split}
\end{equation}
The upper vertical inclusions are both the inclusion of the $M$ summand into $M \vee S^{d-1}$. The lower vertical inclusions are both the embedding of $M \vee S^{d-1}$ into $M$ illustrated in Figure \ref{fig:decomposition}. The bottom square commutes up to homotopy since any point pushing map becomes homotopic to the identity once the puncture(s) have been filled in. The top square commutes up to homotopy by what we have just proven: that $\bar{\pi}^M_\alpha$ factors through $\theta_\alpha$ up to homotopy. The composition of the left-hand vertical maps is homotopic to the identity $M \to M$, and similarly for the right-hand side. Hence three out of the four sides of the outer square of \eqref{eq:proof-thetaalpha} are homotopic to the identity, so the fourth side $\theta_\alpha$ must also be homotopic to the identity.

This completes the proof of Proposition \ref{p:formula-alpha-1}.

\begin{rmk}
\label{rmk:half-of-second-proposition}
This also proves half of Proposition \ref{p:formula-alpha-2}, since that proposition is equivalent to the two statements
\begin{equation}
\label{eq:two-formulas-2}
\bar{\pi}^M_\alpha \simeq \overline{\pitchfork}_\alpha \qquad\text{and}\qquad \bar{\pi}^S_\alpha \simeq ((\alpha \circ \mathrm{coll}) \vee \mathrm{sgn}(\alpha)) \circ \mathrm{pinch},
\end{equation}
and in the proof above we did not use the hypothesis on the handle-dimension of $M$ when proving the right-hand side of \eqref{eq:two-formulas}, which is the same as the right-hand side of \eqref{eq:two-formulas-2}.
\end{rmk}

\subsection{In the maximal handle dimension.}\label{subs:loopII}

In this subsection, we first define the maps $\pitchfork_\alpha$ and $\overline{\pitchfork}_\alpha$ appearing in the statement of Proposition \ref{p:formula-alpha-2}. These depend, a priori, on some additional choices, including a CW-complex $K \subset M$ onto which $M$ deformation retracts. However, Proposition \ref{p:formula-alpha-2} implies that they do not depend on these additional choices up to homotopy (see Remark \ref{rmk:dependence-on-K}).

\begin{defn}
\label{defn:pf1}
Let $K \subset M$ be a CW-complex of dimension at most $d-1$ embedded into $M$ such that $M$ deformation retracts onto $K$. Assume also that $K$ has exactly one $0$-cell and that, for any $i$-cell $\tau$ of $K$, if $\Phi_\tau \colon D^i \to K$ denotes its characteristic map, then the restriction
\[
\Phi_\tau|_{\mathrm{int}(D^i)} \colon \mathrm{int}(D^i) \longrightarrow K \subset M
\]
is a smooth embedding. This exists since $M$ is connected and has non-empty boundary, so its handle-dimension is at most $d-1$: such a CW-complex $K$ may be constructed from the cores of a handle decomposition of $M$ with one $0$-handle. Let $\alpha \in \pi_1(M)$ and choose a representative loop of $\alpha$ that is a smooth embedding, transverse to the interior of every cell of $K$ and also transverse to $\partial M$. (For the assumption that the representative of $\alpha$ may be chosen to be an \emph{embedding}, we are using the fact that $M$ has dimension at least $3$.) Note that the fact that $\alpha$ is transverse to the cells of $K$ implies that it must be disjoint from the $(d-2)$-skeleton $K^{(d-2)}$ of $K$.

Given these choices, we define the map $\pitchfork_\alpha \colon M \to S^{d-1}$ as follows:
\begin{equation}
\label{eq:pf}
\pitchfork_\alpha \colon M \longrightarrow K \twoheadrightarrowlong K/K^{(d-2)} \cong \bigvee_\tau S^{d-1} \longrightarrow S^{d-1},
\end{equation}
where the map $M \to K$ is a homotopy inverse of the inclusion, the index $\tau$ runs over all $(d-1)$-cells of $K$ and the $\tau$-th component of the last map is a map $S^{d-1} \to S^{d-1}$ of degree $\sharp(\tau,\alpha)$, which is the algebraic intersection number of (the interior of) $\tau$ with $\alpha$.

There are two subtleties in this definition: we need to choose the identification of $K/K^{(d-2)}$ with a wedge of $(d-1)$-spheres unambiguously and we need to ensure that the algebraic intersection number $\sharp(\tau,\alpha)$ is well-defined.

For the first point, we simply choose, arbitrarily and once and for all, an orientation of $S^{d-1}$ and an orientation of each open $(d-1)$-cell $\Phi_\tau(\mathrm{int}(D^{d-1}))$ of $K$. The identification of $K/K^{(d-2)}$ with a wedge of copies of $S^{d-1}$ is then well-defined, up to based homotopy, by taking it to be \emph{orientation-preserving} on each open $(d-1)$-cell.

For the second point, to ensure that the algebraic intersection number $\sharp(\tau,\alpha)$ is well-defined, we need an orientation of $\alpha$ and of each open $(d-1)$-cell $\tau$, as well as a local orientation of $M$ at each intersection point of $\alpha$ with the interior of $\tau$, i.e., each point of
\begin{equation}
\label{eq:intersection}
\Phi_\tau(\mathrm{int}(D^{d-1}) \cap \alpha([0,1]).
\end{equation}
We have already chosen orientations of each open $(d-1)$-cell $\tau$, and $\alpha$ is an oriented loop, so it remains to choose local orientations of $M$ at each point of \eqref{eq:intersection}. We do this in several steps:
\begin{itemizeb}
\item We have already chosen an orientation of $S^{d-1}$, which is embedded into $B'$ (see Figure \ref{fig:decomposition}).
\item Let $R$ denote the closure of the connected component of $B' \smallsetminus S^{d-1}$ that is disjoint from $z$, and let $R' = R \smallsetminus \{*\}$. Then $R'$ is a codimension-zero submanifold of $M$ with boundary $\partial R' = (\partial B' \smallsetminus \{*\}) \sqcup (S^{d-1} \smallsetminus \{*\})$. The orientation of $S^{d-1}$ determines an orientation of $R'$ and hence of $\partial B' \smallsetminus \{*\}$.
\item In particular, this restricts to an orientation of $\partial M \cap B' \smallsetminus \{*\} = D \smallsetminus \{*\}$. Choosing a slightly larger disc in $\partial M$ containing $D$ in its interior, this determines a local orientation of $\partial M$ at the basepoint $*$.
\item This, together with $\alpha$, determines a local orientation of $M$ at $*$ as follows: we take it to be the local orientation of $M$ at $*$ such that the algebraic intersection number of $\alpha|_{[1-\epsilon,1]}$ with $\partial M$ at $*$ is $+1$.
\item If $M$ is orientable, this then determines an orientation of $M$, and in particular local orientations of $M$ at each point of \eqref{eq:intersection}.
\item If $M$ is non-orientable, we have to be more careful. Choose $\epsilon > 0$ such that all intersection points \eqref{eq:intersection} are contained in $\alpha([\epsilon,1])$ and choose a closed tubular neighbourhood $T$ of $\alpha|_{[\epsilon,1]}$. Since $T$ is an orientable codimension-zero submanifold of $M$ containing $*$ and each point of \eqref{eq:intersection}, we may use it to transport the local orientation of $M$ at $*$ to a local orientation of $M$ at each point of \eqref{eq:intersection}.
\end{itemizeb}

We note that this definition does \emph{not} depend on our arbitrary choices of orientations for $S^{d-1}$ and for each open $(d-1)$-cell $\tau$ of $K$:
\begin{itemizeb}
\item Suppose that we reverse the orientation of one $(d-1)$-cell $\tau_0$. This affects the identification of $K/K^{(d-2)}$ with the wedge of $(d-1)$-spheres in a way that corresponds to inserting an automorphism of $\bigvee_\tau S^{d-1}$ that sends each sphere to itself, has degree $-1$ on the $\tau_0$ component and has degree $+1$ on all other components. However, it also has the effect of reversing the sign of the algebraic intersection number $\sharp(\tau_0,\alpha)$, so these effects cancel each other out after composing all maps in \eqref{eq:pf}.
\item Suppose that we reverse the orientation of $S^{d-1}$. This affects the identification of $K/K^{(d-2)}$ with the wedge of $(d-1)$-spheres in a way that corresponds to inserting an automorphism of $\bigvee_\tau S^{d-1}$ that sends each sphere to itself and has degree $-1$ on each component. However, it also has the effect of reversing the local orientations of $M$ at each intersection point \eqref{eq:intersection} for each $\tau$, and so it reverses the sign of each algebraic intersection number $\sharp(\tau,\alpha)$. Again, these effects cancel each other out after composing all maps in \eqref{eq:pf}.
\end{itemizeb}

This completes the definition of the map $\pitchfork_\alpha \colon M \to S^{d-1}$.
\end{defn}

For the definition of $\overline{\pitchfork}_\alpha$, we again use an embedded CW-complex $K \subset M$ as in Definition \ref{defn:pf1}, and choose a representative loop of $\alpha \in \pi_1(M)$ as in Definition \ref{defn:pf1}.

\begin{defn}
\label{defn:pf2}
We now define a map $\overline{\pitchfork}_\alpha \colon M \to M \vee S^{d-1}$ whose composition with the projection $\mathrm{pr}_{S^{d-1}} \colon M \vee S^{d-1} \to S^{d-1}$ is $\pitchfork_\alpha$. This is the map
\begin{equation}
\label{eq:pf2}
\overline{\pitchfork}_\alpha \colon M \longrightarrow K \longrightarrow M \vee S^{d-1}
\end{equation}
where the first map is a homotopy inverse of the inclusion and the second map is defined as follows. On the $(d-2)$-skeleton it is defined to be the inclusion $K^{(d-2)} \subset K \subset M \subset M \vee S^{d-1}$. We now extend this to each $(d-1)$-cell of $K$, in other words, for each $(d-1)$-cell $\tau$ of $K$, we define a map
\begin{equation}
\label{eq:pf2tau}
\overline{\pitchfork}_{\alpha,\tau} \colon D^{d-1} \longrightarrow M \vee S^{d-1}
\end{equation}
whose restriction to $\partial D^{d-1}$ is equal to the attaching map $\phi_\tau \colon \partial D^{d-1} \to K^{(d-2)}$ of $\tau$ followed by the inclusion $K^{(d-2)} \subset K \subset M \subset M \vee S^{d-1}$. We define the map \eqref{eq:pf2tau} in several steps:
\begin{itemizeb}
\item Denote the intersection points of $\alpha$ with the interior of $\tau$ by
\[
\Phi_\tau(\mathrm{int}(D^{d-1})) \cap \alpha([0,1]) = \{ y_1,\ldots,y_n \}
\]
and write $x_i = \Phi_{\tau}^{-1}(y_i) \in \mathrm{int}(D^{d-1})$.
\item Let
\[
\epsilon = \tfrac{1}{8}\mathrm{min} \bigl( \{ \lvert x_i - x_j \rvert \text{ for } i,j \in \{1,\ldots,n\}, i \neq j \} \cup \{ 1 - \lvert x_i \rvert \text{ for } i \in \{1,\ldots,n\} \} \bigr)
\]
and write $S_i = \partial B_\epsilon(x_i)$ for the boundary of the ball of radius $\epsilon$ around $x_i$. Let $\sim$ be the equivalence relation that collapses each $S_i \subseteq D^{d-1}$ to a (different) point, for $i \in \{1,\ldots,n\}$. There is a canonical homeomorphism
\begin{equation}
\label{eq:identification-quotient}
D^{d-1} / {\sim} \;\cong\; D^{d-1} \cup_n \bigsqcup_n S^{d-1},
\end{equation}
where the notation $\cup_n$ indicates that we are taking the union along $n$ distinct basepoints, more precisely we identify $x_i \in D^{d-1}$ with the basepoint of the $i$th copy of $S^{d-1}$, for $i \in \{1,\ldots,n\}$. The homeomorphism \eqref{eq:identification-quotient} is given on $B_\epsilon(x_i)/S_i \subseteq D^{d-1}/{\sim}$ by identifying the ball $B_\epsilon(x_i)$ with $D^{d-1}$ by dilatation and translation, and then using the standard (stereographic) identification $D^{d-1} / \partial D^{d-1} \cong S^{d-1}$. It is given by the identity outside of each of the larger balls $B_{2\epsilon}(x_i)$, and on each subspace $(B_{2\epsilon}(x_i) \smallsetminus \mathrm{int}(B_\epsilon(x_i))) / S_i$ it is the homeomorphism
\[
(B_{2\epsilon}(x_i) \smallsetminus \mathrm{int}(B_\epsilon(x_i))) / S_i \longrightarrow B_{2\epsilon}(x_i)
\]
given by $x_i + y \longmapsto x_i + (\lvert y \rvert / \epsilon - 1)y$ (i.e.~``stretching'' inwards by a factor of two). Let
\[
c_n \colon D^{d-1} \longrightarrow D^{d-1} \cup_n \bigsqcup_n S^{d-1}
\]
be the quotient map $D^{d-1} \twoheadrightarrow D^{d-1} / {\sim}$ followed by the identification \eqref{eq:identification-quotient}. Composing this with the ``pinch and collapse map'' $(\mathrm{coll} \vee \mathrm{id}) \circ \mathrm{pinch}$ (see Notation \ref{notation}) on each $S^{d-1}$ factor we obtain a quotient map
\begin{equation}
\label{eq:pf2tau1}
\bar{c}_n \colon D^{d-1} \longrightarrow D^{d-1} \cup_n \bigsqcup_n \; \bigl([0,1] \vee S^{d-1}\bigr).
\end{equation}
See Figure \ref{fig:pf2} for a visual illustration of this construction.
\item Finally, we define \eqref{eq:pf2tau} by $\overline{\pitchfork}_{\alpha,\tau} = \smash{\overline{\pitchfork}}^{\diamond}_{\alpha,\tau} \circ \bar{c}_n$, where the map
\[
\smash{\overline{\pitchfork}}^{\diamond}_{\alpha,\tau} \colon D^{d-1} \cup_n \bigsqcup_n \; \bigl([0,1] \vee S^{d-1}\bigr) \longrightarrow M \vee S^{d-1}
\]
is defined on each piece of the domain as follows.
  \begin{itemizeb}
  \item On the $D^{d-1}$ piece, $\smash{\overline{\pitchfork}}^{\diamond}_{\alpha,\tau}$ is given by the characteristic map $\Phi_\tau \colon D^{d-1} \to K$ followed by the inclusion $K \subset M \subset M \vee S^{d-1}$.
  \item On the $i$-th $[0,1]$ piece, $\smash{\overline{\pitchfork}}^{\diamond}_{\alpha,\tau}$ is the path $\alpha|_{[\alpha^{-1}(y_i),1]}$ in $M$ (rescaled so that its domain is $[0,1]$). Note that this path ends at the basepoint.
  \item On the $i$-th $S^{d-1}$ piece, $\smash{\overline{\pitchfork}}^{\diamond}_{\alpha,\tau}$ is a based map $S^{d-1} \to S^{d-1}$ of degree $\epsilon_i \in \{\pm 1\}$, where the sign $\epsilon_i$ is determined as follows.
    \begin{itemizeb}
    \item[$\bullet$] As in Definition \ref{defn:pf1}, the chosen orientation of $S^{d-1}$ determines a local orientation of $M$ at $*$.
    \item[$\bullet$] We have also chosen an orientation of $D^{d-1}$, and $\Phi_\tau$ is a smooth embedding on the interior of $D^{d-1}$, so we also have an orientation of $\Phi_\tau(\mathrm{int}(D^{d-1}))$. This determines a local orientation of $M$ at the intersection point $y_i$: namely the one with respect to which the intersection number of $\Phi_\tau(\mathrm{int}(D^{d-1}))$ with $\alpha([0,1])$ at $y_i$ is $+1$.
    \item[$\bullet$] If $M$ is orientable, these two local orientations each determine an orientation of $M$, and we set $\epsilon_i$ to be $+1$ if they agree and $-1$ if they disagree.
    \item[$\bullet$] If $M$ is non-orientable, we have to be more careful, just as in Definition \ref{defn:pf1}. Choose $\delta > 0$ such that all intersection points $y_1,\ldots,y_n$ are contained in $\alpha([\delta,1])$ and choose a tubular neighbourhood $T$ of $\alpha|_{[\delta,1]}$. Since $T$ is an orientable codimension-zero submanifold of $M$ containing $*$ and $y_i$, the two local orientations of $M$ (at $*$ and at $y_i$) each determine an orientation of $T$. We set $\epsilon_i = +1$ if they agree and $\epsilon_i = -1$ if they disagree.
    \end{itemizeb}
  \end{itemizeb}
\end{itemizeb}

One may see, as in Definition \ref{defn:pf1}, that this construction of $\overline{\pitchfork}_\alpha$ is independent of the choices of orientation of $S^{d-1}$ and $D^{d-1}$.
\end{defn}

\begin{figure}[t]
\centering
\includegraphics[scale=0.6]{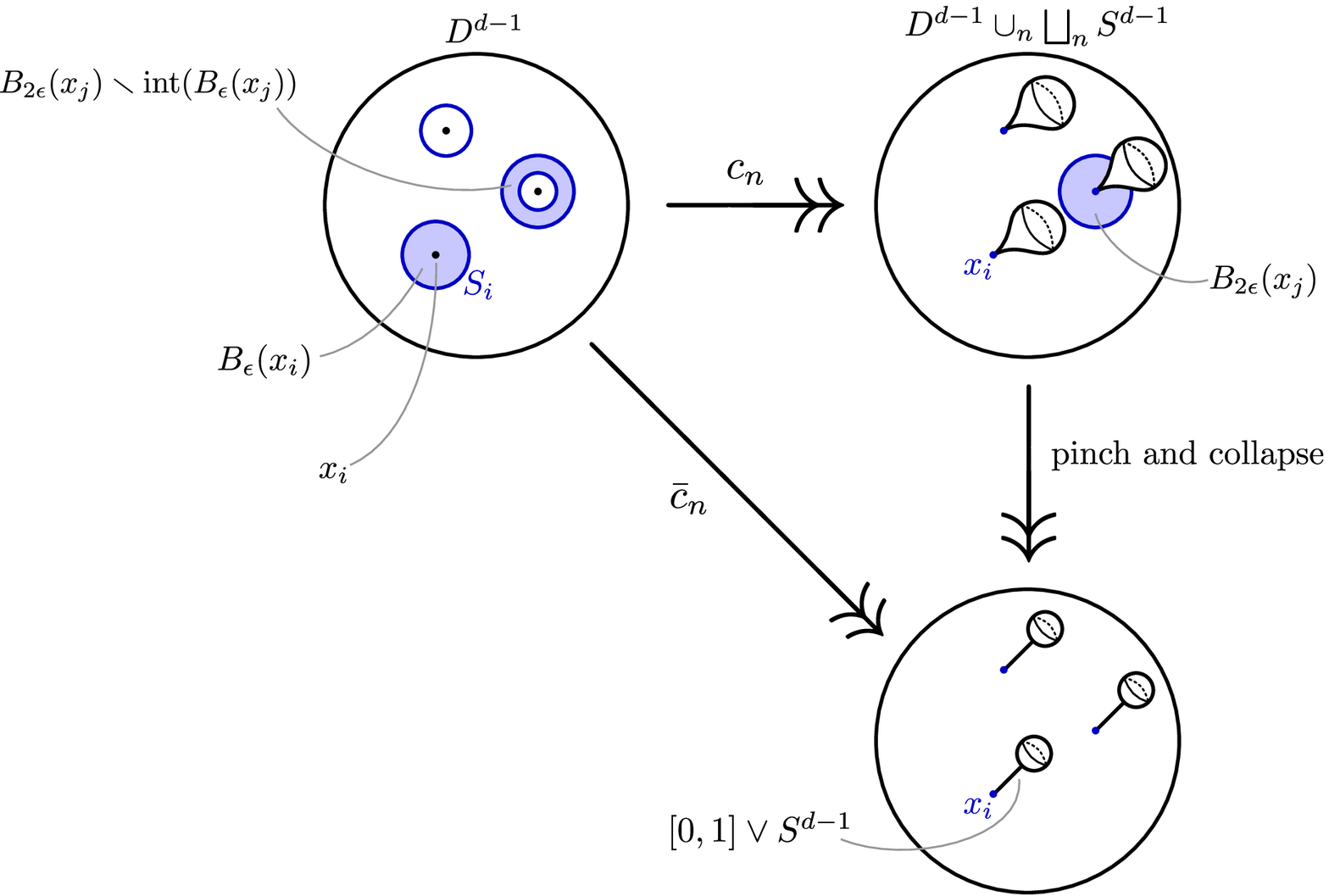}
\caption{The quotient map $\bar{c}_n \colon D^{d-1} \longrightarrow D^{d-1} \cup_n \bigsqcup_n \; \bigl([0,1] \vee S^{d-1}\bigr)$ from Definition \ref{defn:pf2}.}
\label{fig:pf2}
\end{figure}

\begin{rmk}
\label{rmk:dependence-on-K}
A priori, the maps $\pitchfork_\alpha \colon M \to S^{d-1}$ and $\overline{\pitchfork}_\alpha \colon M \to M \vee S^{d-1}$ described in Definitions \ref{defn:pf1} and \ref{defn:pf2} depend on the choice of embedded CW-complex $K$ and the choice of representative of $\alpha \in \pi_1(M)$ that is a smooth embedding and transverse to $\partial M$ and each open cell of $K$. However, a consequence of Proposition \ref{p:formula-alpha-2} is that these maps, up to basepoint-preserving homotopy, do \emph{not} depend on these choices; they depend only on the element $\alpha \in \pi_1(M)$. This is because Proposition \ref{p:formula-alpha-2} identifies these two maps with certain maps derived from the point-pushing map $\pi_\alpha$, which depends up to homotopy only on $\alpha \in \pi_1(M)$.
\end{rmk}

\begin{proof}[Proof of Proposition \ref{p:formula-alpha-2}]
As pointed out in Remark \ref{rmk:half-of-second-proposition}, we have already proven one half of Proposition \ref{p:formula-alpha-2} while proving Proposition \ref{p:formula-alpha-1}. The remaining statement to prove is
\begin{equation}
\label{eq:one-equation}
\bar{\pi}^M_\alpha \simeq \overline{\pitchfork}_\alpha \colon M \too M \vee S^{d-1}.
\end{equation}
We will first prove the two (jointly weaker) statements:
\begin{equation}
\label{eq:jointly-weaker-equations}
\mathrm{pr}_M \circ \bar{\pi}^M_\alpha \simeq \mathrm{id}_M \qquad\text{and}\qquad \mathrm{pr}_{S^{d-1}} \circ \bar{\pi}^M_\alpha \simeq {\pitchfork_\alpha},
\end{equation}
which correspond to the $(2 \times 2)$-matrix description of $\bar{\pi}_\alpha$ on the right-hand side of \eqref{eq:formula-alpha-2}. Consider the following homotopy-commutative diagram.
\begin{equation}
\label{eq:homotopic-to-identity}
\centering
\begin{split}
\begin{tikzpicture}
[x=1mm,y=1mm]
\node (tll) at (-30,15) {$M$};
\node (tl) at (0,15) {$M \vee S^{d-1}$};
\node (tr) at (40,15) {$M \vee S^{d-1}$};
\node (bl) at (0,0) {$M$};
\node (br) at (40,0) {$M$};
\incl{(tll)}{(tl)}
\incl{(tl)}{(bl)}
\incl{(tr)}{(br)}
\draw[->] (tl) to node[above,font=\small]{$\bar{\pi}_\alpha$} (tr);
\draw[->] (bl) to node[above,font=\small]{$\mathrm{id}$} (br);
\draw[->] (tll) to node[below left,font=\small]{$\mathrm{id}$} (bl);
\end{tikzpicture}
\end{split}
\end{equation}
(The square is the same as the bottom square of \eqref{eq:proof-thetaalpha}.) The two vertical inclusions are both the embedding of $M \vee S^{d-1}$ into $M$ illustrated in Figure \ref{fig:decomposition}. But this is homotopic to the projection $\mathrm{pr}_M$ of $M \vee S^{d-1}$ onto its first summand, so $\mathrm{pr}_M \circ \bar{\pi}^M_\alpha$ is the composition from the top-left to the bottom-right of the diagram, and hence homotopic to the identity. This proves the left-hand side of \eqref{eq:jointly-weaker-equations}.

Next, we prove the right-hand side of \eqref{eq:jointly-weaker-equations}. We start by giving another description of the map
\[
w_\alpha = \mathrm{pr}_{S^{d-1}} \circ \bar{\pi}^M_\alpha \colon M \too S^{d-1}
\]
using Figure \ref{fig:decomposition}. Choose a path $p$ in $B'$ from $*$ to the point $z \cap B'$ and choose a loop $\delta$ in $B' \cup M'$, intersecting $\partial M'$ transversely in two points, in the homotopy class of $p \cdot \alpha \cdot \bar{p}$. Also choose a tubular neighbourhood $T$ of $\delta \cap M'$ in $M'$. Geometrically, the map $w_\alpha \colon M \to S^{d-1}$ is then given by starting in $M'$, including into $M \smallsetminus z$, applying the point pushing map along the loop $\delta$ and then collapsing onto the copy of $S^{d-1}$ contained in $B'$. Clearly the complement $M' \smallsetminus T$ of the tubular neighbourhood $T$ is sent to the basepoint under this map. To describe how $w_\alpha$ acts on $T$, we use the following identifications. The intersection $T \cap \partial B'$ consists of two disjoint $(d-1)$-discs $T_0$ and $T_1$, where we assume that $T_0$ contains the intersection point of $\delta \cap \partial B'$ where $\delta$ is pointing into $M'$ and $T_1$ contains the intersection point of $\delta \cap \partial B'$ where $\delta$ is pointing into $B'$. We may then identify $T$ with $T_1 \times [0,1]$, write $\partial_l T = \partial T \smallsetminus (\mathrm{int}(T_0) \cup \mathrm{int}(T_1))$ and describe the map $w_\alpha$ restricted to $T$, as a map of pairs $(T,\partial_l T) \to (S^{d-1},*)$, by
\begin{equation}
\label{eq:walpha}
(T,\partial_l T) \cong (T_1,\partial T_1) \times [0,1] \longrightarrow (T_1,\partial T_1) \longrightarrow (T_1 / \partial T_1 , \partial T_1 / \partial T_1) \cong (S^{d-1},*),
\end{equation}
where the middle two maps are the obvious projections and the identification on the right-hand side is induced by the projection $T_1 \twoheadrightarrow S^{d-1}$ given by
\[
\begin{tikzcd}
T_1 \ar[r,hook] & \partial B' \ar[r,"r"] & \partial B' \ar[r,"{\pi}",two heads] & \partial B' \cong S^{d-1},
\end{tikzcd}
\]
where $r$ is a reflection in the $(d-1)$-sphere $\partial B'$, $\pi$ is a self-surjection of $\partial B'$ with the properties that $\pi^{-1}(*) = \partial B' \smallsetminus \mathrm{int}(T_1)$ and $\pi$ is locally orientation-preserving on $\mathrm{int}(T_1)$, and the homeomorphism $\partial B' \cong S^{d-1}$ is given by a based isotopy in $B'$ between $\partial B'$ and the embedded copy of $S^{d-1}$ in $B'$ in Figure \ref{fig:decomposition}.

We now use this geometric description of $w_\alpha$ to show that it is homotopic to the map $\pitchfork_\alpha$ defined in Definition \ref{defn:pf1}. Let $K$ be a CW-complex of dimension at most $d-1$ embedded into $M'$, such that $M'$ deformation retracts onto $K$. We need to show that the restriction of $w_\alpha$ to $K$ factors as
\begin{equation}
\label{eq:walphaK}
K \twoheadrightarrowlong K/K^{(d-2)} \cong \textstyle\bigvee_\tau S^{d-1} \longrightarrow S^{d-1},
\end{equation}
where the $\tau$-th component of the right-hand map is a map $f_\tau \colon S^{d-1} \to S^{d-1}$ of degree $\sharp(\tau,\delta)$. By smooth approximation and transversality, we may assume that each $(d-1)$-cell $\tau$ of $K$ is smoothly embedded into $M'$ and that $\delta$ and $T$ have been chosen so that (a) each $r$-cell of $K$, for $r \leq d-2$, is disjoint from $T$ and (b) each $\tau \cap T$, for $\tau$ a $(d-1)$-cell of $K$, consists of finitely many $(d-1)$-discs each intersecting $\delta$ transversely in one point.

By property (a), and since $M' \smallsetminus T$ is sent to the basepoint by $w_\alpha$, we see that its restriction to $K$ must factor through the projection $K \twoheadrightarrow K/K^{(d-2)}$. So we just have to show that $f_\tau$ has degree $\sharp(\tau,\delta)$. By property (b) and the description \eqref{eq:walpha} of $w_\alpha|_T$, each component of the disjoint union of $(d-1)$-discs $\tau \cap T$ contributes either $+1$ or $-1$ to $\mathrm{deg}(f_\tau)$. Being careful about (local) orientations as explained in Definition \ref{defn:pf1}, we see that the sum of these $+1$'s and $-1$'s is precisely the algebraic intersection number $\sharp(\tau,\delta)$ of $\tau$ and $\delta$.

This completes the proof that $w_\alpha|_K$ factors as in \eqref{eq:walphaK}, and hence that $w_\alpha \simeq {\pitchfork_\alpha}$, in other words, the right-hand side of \eqref{eq:jointly-weaker-equations}.

The proof of \eqref{eq:one-equation} is similar to the proof above of the right-hand side of \eqref{eq:jointly-weaker-equations}: looking at Figure \ref{fig:decomposition} and using a geometric model for the point-pushing map supported in a tubular neighbourhood of an embedded loop representing $\alpha$, one checks carefully that the definition of $\overline{\pitchfork}_\alpha$ from Definition \ref{defn:pf2} is a correct description of $\bar{\pi}^M_\alpha$ up to homotopy. This is explained in Figure \ref{fig:alpha-tau}, which depicts the map $\bar{\pi}^M_\alpha$ induced by point-pushing along $\alpha$ and compares it to the definition of $\overline{\pitchfork}_\alpha$.
\end{proof}

\definecolor{tubular}{RGB}{0,150,0}
\begin{figure}[t]
\labellist
\small\hair 2pt
 \pinlabel {$M'$} [ ] at 164 642
 \pinlabel {$B$} [ ] at 89 585
 \pinlabel {$y_1$} [ ] at 176 601
 \pinlabel {$y_2$} [ ] at 260 595
 \pinlabel {$y_3$} [ ] at 247 524
 \pinlabel {\textcolor{blue}{$\tau$}} [b] at 228 563
 \pinlabel {\textcolor{tubular}{$T$}} [b] at 202 607
 \pinlabel {\textcolor{red}{$\alpha$}} [l] at 138 545
 \pinlabel {$*$} [t] at 125 527
 \pinlabel {\footnotesize (2)} [ ] at 360 654
 \pinlabel {\footnotesize $\alpha \cap \tau$ viewed in $D^{d-1}_\tau$} [ ] at 522 644
 \pinlabel {\footnotesize $\mathrm{inc} \circ \Phi_\tau$} [b] at 487 483
 \pinlabel {\footnotesize $\alpha|_{[\alpha^{-1}(x),1]}$} [b] at 517 434
 \pinlabel {\footnotesize $\pm \mathrm{id}$} [b] at 517 415
 \pinlabel {$M$} [l] at 535 486
 \pinlabel {$M$} [l] at 551 435
 \pinlabel {$S^{d-1}$} [l] at 550 417
 \pinlabel {\footnotesize (1)} [ ] at 57 483
 \pinlabel {\footnotesize $\alpha \cap \tau$ viewed in $M$} [ ] at 183 483
 \pinlabel {\textcolor{red}{$\alpha$}} [b] at 194 448
 \pinlabel {$y_1$} [t] at 105 428
 \pinlabel {$y_2$} [t] at 166 428
 \pinlabel {$y_3$} [t] at 225 428
 \pinlabel {$S^{d-1}$} [t] at 305 435
\endlabellist
\centering
\includegraphics[scale=0.7]{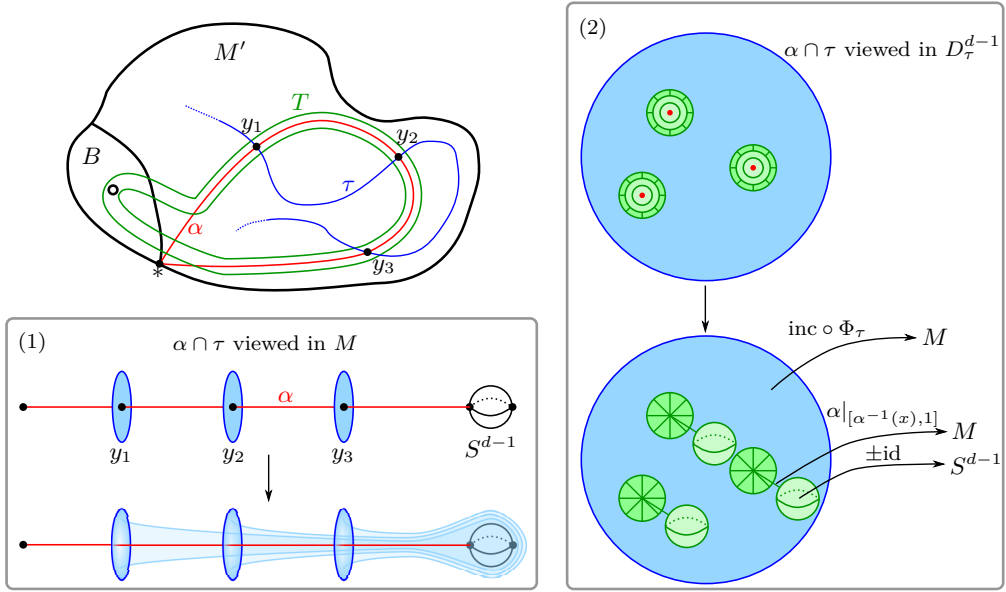}
\caption{Two views of the map $\bar{\pi}^M_\alpha \colon M \to M \vee S^{d-1}$ induced by point-pushing along an embedded loop $\alpha$, and in particular its effect on a $(d-1)$-cell $\tau$. In $(1)$, we view the loop $\alpha$ (in red) and a neighbourhood of its intersections $\{y_1,\ldots,y_n\}$ with $\tau$ (in blue) from within $M$. From the geometric description of point-pushing (Lemma \ref{l:geometric-point-pushing}), the result of point-pushing is as depicted in the bottom half of $(1)$: the $n$ small disc neighbourhoods of $\{y_1,\ldots,y_n\}$ in $\tau$ are pulled along $\alpha$ and wrapped around the $S^{d-1}$ summand of $M \vee S^{d-1}$. From this, we may deduce the description of $\bar{\pi}^M_\alpha$ up to homotopy given in part $(2)$ of the figure. At the top of $(2)$, we view the disc $D_{\tau}^{d-1}$ (whose image under the characteristic map $\Phi_\tau$ is the cell $\tau$) in blue and its intersection $\{y_1,\ldots,y_n\}$ with $\alpha$ as a configuration of red points. We also choose small disc neighbourhoods of each of these points (now depicted in green), divided into three concentric regions. Translating the depiction of $\bar{\pi}^M_\alpha$ from $(1)$ into this viewpoint, we see that the blue region (the complement of the small green disc neighbourhoods) is fixed by $\bar{\pi}^M_\alpha$, in other words, it is simply mapped into $M$ by the characteristic map $\Phi_\tau$ of the cell. For each small green disc neighbourhood, its image under $\bar{\pi}^M_\alpha$ is illustrated as a light blue surface in $(1)$; projecting this onto $\tau \cup \alpha \cup S^{d-1}$ does not change it up to homotopy, and this may then be described in $(2)$ as follows: the outer region of each green disc is ``stretched'' to cover the whole green disc (and then mapped into $M$ via the characteristic map $\Phi_\tau$); the intermediate region is collapsed to an interval and then mapped into $M$ via a terminal segment of the loop $\alpha$; the central region is collapsed to a sphere and then mapped with degree $\pm 1$, depending on local orientations, to the $S^{d-1}$ summand of $M \vee S^{d-1}$. This is precisely the map \eqref{eq:pf2tau} from Definition \ref{defn:pf2} (see in particular Figure \ref{fig:pf2}), which is the restriction of $\overline{\pitchfork}_\alpha$ to the cell $\tau$. Thus for each $(d-1)$-cell $\tau$, the restrictions of $\bar{\pi}^M_\alpha$ and of $\overline{\pitchfork}_\alpha$ to $\tau$ are homotopic relative to its boundary; hence $\bar{\pi}^M_\alpha \simeq \overline{\pitchfork}_\alpha$.}
\label{fig:alpha-tau}
\end{figure}

%%%%%%%%%%%%%%%%%%%%%%%%%%%%%%%%%%%%%%%%%%%
%%%%%%%%%%%%%%%%%%%%%%%%%%%%%%%%%%%%%%%%%%%
\section{Examples}\label{s:examples}
%%%%%%%%%%%%%%%%%%%%%%%%%%%%%%%%%%%%%%%%%%%
%%%%%%%%%%%%%%%%%%%%%%%%%%%%%%%%%%%%%%%%%%%

To illustrate the more complicated setting where $M$ is non-simply-connected and has maximal handle dimension, we discuss some explicit examples, namely
\[
M = (S^1 \times S^2) \smallsetminus \mathrm{int}(D^3)
\]
and more generally
\[
M = {\underbrace{(S^1 \times S^2) \sharp (S^1 \times S^2) \sharp \cdots \sharp (S^1 \times S^2)}_{g \text{ copies}}} \smallsetminus \mathrm{int}(D^3)
\]
which all have maximal handle-dimension $\mathrm{dim}(M) - 1 =2$ and which have fundamental groups $\bZ$ and $F_g$, the free group on $g$ generators, respectively. Indeed, the following computations generalise to all
\[
M = M^d_{g,1}= {\underbrace{(S^1 \times S^{d-1}) \sharp (S^1 \times S^{d-1}) \sharp \cdots \sharp (S^1 \times S^{d-1})}_{g \text{ copies}}} \smallsetminus \mathrm{int}(D^d)
\]
for $d\geq 3$ and $g>0$.

\definecolor{dgreen}{RGB}{0,180,0}
\definecolor{dred}{RGB}{148,0,0}
\begin{figure}[t]
\labellist
\small\hair 2pt
 \pinlabel {$\textcolor{red}{\tau}$} [ ] at 38 182
 \pinlabel {$\textcolor{red}{\alpha}$} [ ] at 181 117
 \pinlabel {$\textcolor{dgreen}{T}$} [ ] at 240 159
 \pinlabel {$\textcolor{dred}{p}$} [ ] at 180 84
\endlabellist
\centering
\includegraphics[scale=1]{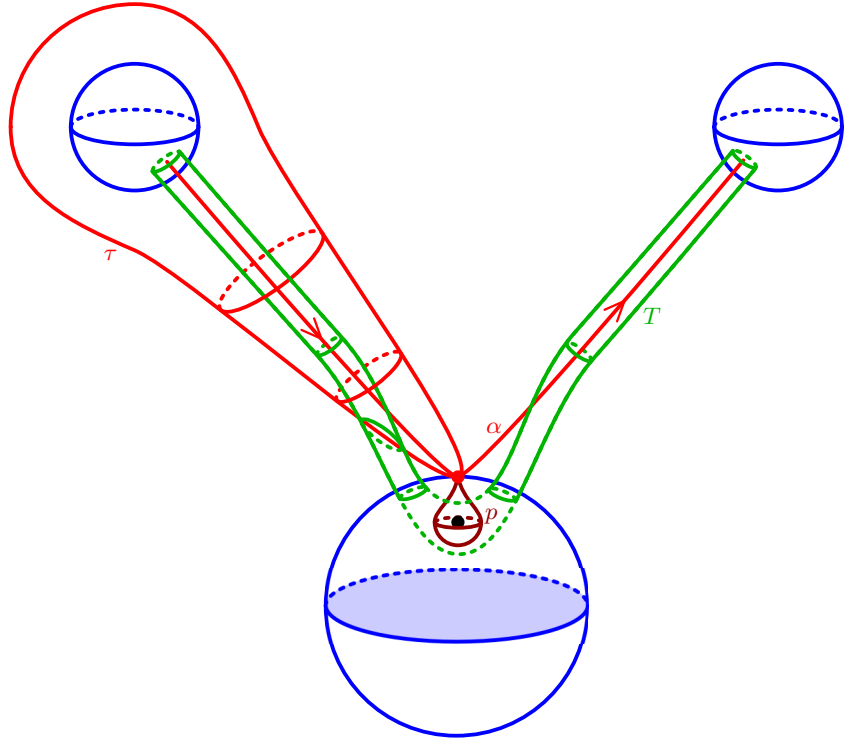}
\caption{The picture is to be thought of as $S^3$ with three open balls (in blue) cut out, and the boundaries of two of them (the top two) identified by a reflection. This is a model for the manifold $M = (S^1 \times S^2) \smallsetminus \mathrm{int}(D^3)$. The embedded copy of $S^1 \vee S^2$ is drawn in red, consisting of a $1$-sphere called $\alpha$ and a $2$-sphere called $\tau$. The manifold $M$ deformation retracts onto this subspace. As a model for $M$ with a puncture removed, we glue back in the top half of the lower $3$-ball (so the boundary now consists of the light blue shaded $2$-disc together with the southern hemisphere of the lower blue $2$-sphere) and then remove the black point. This manifold (let us call it $M'$) deformation retracts onto the embedded wedge sum $S^1 \vee S^2 \vee S^2$ consisting of $\alpha$, $\tau$ and the dark red $2$-sphere called $p$. The green solid cylinder called $T$ is a tubular neighbourhood of $\alpha$, isotoped slightly so that it contains the puncture in its interior. Thus, the effect of the point-pushing map on $\alpha$ may be realised explicitly by a diffeomorphism of the manifold $M'$ supported in the interior of $T$, as described in Lemma \ref{l:geometric-point-pushing}.}
\label{fig:s1s2}
\end{figure}

\begin{eg}
First, consider $M = (S^1 \times S^2) \smallsetminus \mathrm{int}(D^3)$ and let $\alpha$ be a generator of $\pi_1(M) \cong \bZ$. By Proposition \ref{p:formula-alpha-2}, the point-pushing map
\[
\bar{\pi}_\alpha \colon M \vee S^2 \too M \vee S^2
\]
has a simple explicit description when restricted to the $S^2$ summand, and is homotopic to the (in general complicated) map $\overline{\pitchfork}_\alpha \colon M \to M \vee S^2$ of Definition \ref{defn:pf2} when restricted to the $M$ summand.

In this example, $M$ is homotopy equivalent to $S^1 \vee S^2$ (see Figure \ref{fig:s1s2} for a picture of an embedded $S^1 \vee S^2$ onto which it deformation retracts). So, under this identification, the point pushing map $\bar{\pi}_\alpha$ is an endomorphism of $S^1 \vee S^2 \vee S^2$. We will label the $1$- and  $2$-spheres with subscripts $\alpha$, $\tau$ and $p$ to indicate which of the spheres they correspond to (light or dark red spheres in Figure \ref{fig:s1s2}). Thus our aim is to describe (up to based homotopy) the map
\[
\bar{\pi}_\alpha \colon S^1_\alpha \vee S^2_p \vee S^2_\tau = X \too X = S^1_\alpha \vee S^2_p \vee S^2_\tau .
\]
This is an element of the homotopy set $\langle X,X \rangle = \pi_0(\mathrm{Map}_*(X,X))$, which becomes a monoid under composition. In fact, we know of course that $\bar{\pi}_\alpha$ must be an \emph{invertible} element of this monoid, i.e.\ an element of $\pi_0(\mathrm{hAut}_*(X))$, but we will describe it as an element of the larger monoid $\langle X,X \rangle$. In order to do this, we first describe the monoid $\langle X,X \rangle$ explicitly.

First, note that there is a  bijection
\[
\langle X,X \rangle \;\cong\; \pi_1(X) \times \pi_2(X) \times \pi_2(X),
\]
and that $\pi_1(X) \cong \bZ\{\alpha\}$, the free (abelian) group generated by $\alpha$. The second homotopy group of $X$ is the same as that of its universal cover, and using Hilton's theorem \cite{Hilton1955} to compute homotopy groups of wedges of spheres, we see that
\[
\pi_2(X) \cong \bZ\{ \alpha^n p , \alpha^n \tau \mid n \in \bZ \},
\]
the free abelian group generated by the symbols $\alpha^n p$ and $\alpha^n \tau$ for each $n \in \bZ$. Moreover, the action of $\pi_1(X) = \bZ\{\alpha\}$ is given by $\alpha . \alpha^n p = \alpha^{n+1} p$ and $\alpha . \alpha^n \tau = \alpha^{n+1} \tau$. This means that we may write $\pi_2(X) \cong \bZ[\alpha^{\pm 1}]\{p,\tau\} = \bZ[\pi_1(X)]\{p,\tau\}$ as a free module over the group-ring of $\pi_1(X)$. Putting these identifications together, we have
\begin{equation}
\label{eq:identification}
\langle X,X \rangle \;\cong\; \bZ\{\alpha\} \times \bZ[\alpha^{\pm 1}]\{p,\tau\} \times \bZ[\alpha^{\pm 1}]\{\tau,p\}
\end{equation}
as a set. To describe the monoid operation (composition) on $\langle X,X \rangle$ under this identification, it is useful to include it into the larger monoid $\langle \widetilde{X},\widetilde{X} \rangle$, where
\[
\widetilde{X} \simeq \bigvee_{i \in \bZ} S^2_{\alpha^i p} \vee \bigvee_{i \in \bZ} S^2_{\alpha^i \tau}
\]
is the universal cover of $X$. Since $\vee$ is the coproduct for pointed spaces, we have
\begin{equation}
\label{eq:identification-universal-cover}
\langle \widetilde{X},\widetilde{X} \rangle \cong M_2(M_{\bZ}^{\mathrm{vf}}(\bZ)),
\end{equation}
the monoid of $2 \times 2$ block matrices whose entries are vertically-finite $\bZ \times \bZ$ matrices with entries in $\bZ$. (Vertically-finite means that each column has only finitely many non-zero entries.) For example, the $(i,j)$ entry in the bottom-left block of the matrix corresponding to $f \colon \widetilde{X} \to \widetilde{X}$ records the degree of the map
\[
S^2_{\alpha^j p} \lhook\joinrel\longrightarrow \widetilde{X} \xrightarrow{\;f\;} \widetilde{X} \relbar\joinrel\twoheadrightarrow S^2_{\alpha^i \tau}.
\]
Once a compatible base point in $\widetilde X$ is fixed, each based self-map of $X$ lifts uniquely up to homotopy to a based self-map of $\widetilde{X}$, so there is an injection $\langle X,X \rangle \hookrightarrow \langle \widetilde{X},\widetilde{X} \rangle$. Under the identifications \eqref{eq:identification} and \eqref{eq:identification-universal-cover}, this is given by
\[
\left( k\alpha , \sum_i \alpha^i(m_i p + n_i \tau) , \sum_i \alpha^i (r_i p + s_i \tau) \right) \;\longmapsto\; \left( \begin{matrix} A & B \\ C & D \end{matrix} \right) ,
\]
where each of the matrices $A,B,C,D$ is a diagonally constant matrix of slope $-k$, in other words its $(i,j)$ entry is equal to its $(i-jk,0)$ entry; in particular it is determined by its $0$th column, and the $0$th columns of $A=(a_{ij})$, $B=(b_{ij})$, $C=(c_{ij})$, $D=(d_{ij})$ are given by
\[
a_{i0} = m_i \qquad b_{i0} = r_i \qquad c_{i0} = n_i \qquad d_{i0} = s_i .
\]
For example, the identity $X \to X$ corresponds to $(\alpha,p,\tau)$, which is sent to $\left( \begin{smallmatrix} I & 0 \\ 0 & I \end{smallmatrix} \right)$, and the map $X \to X$ that is the identity on the two $S^2$ factors and collapses the $S^1$ factor to the basepoint corresponds to $(0,p,\tau)$, which is sent to $\left( \begin{smallmatrix} 1_0 & 0 \\ 0 & 1_0 \end{smallmatrix} \right)$, where $1_0$ is the matrix with $1$s on the $0$th row and $0$s elsewhere.

Since the operation on $\langle \widetilde{X},\widetilde{X} \rangle$ is just multiplication of matrices, one may use this inclusion of monoids to deduce a formula for the operation on $\langle X,X \rangle$ under the identification \eqref{eq:identification}, which is given as follows:
{\small \begin{multline*}
\left( k\alpha , \sum_i \alpha^i(m_i p + n_i \tau) , \sum_i \alpha^i (r_i p + s_i \tau) \right)
\circ
\left( k'\alpha , \sum_j \alpha^j(m'_j p + n'_j \tau) , \sum_j \alpha^j (r'_j p + s'_j \tau) \right) \\
=
\left( kk'\alpha ,
\sum_{i,j} \alpha^{i+jk} ((m_i m'_j + r_i n'_j)p + (n_i m'_j + s_i n'_j)\tau ) ,
\sum_{i,j} \alpha^{i+jk} ((m_i r'_j + r_i s'_j)p + (n_i r'_j + s_i s'_j)\tau ) \right) .
\end{multline*}}

By considering the action on the universal cover, and using Definition \ref{defn:pf2} and Proposition \ref{p:formula-alpha-2}, we may write the element $\bar{\pi}_\alpha \in \langle X,X \rangle$ in terms of these explicit descriptions of $\langle X,X \rangle$ as follows:
\[
\bar{\pi}_\alpha = (\alpha , \alpha p , \tau + p).
\]
Similarly, we may calculate that:
\[
\bar{\pi}_{\alpha^{-1}} = (\alpha , \alpha^{-1} p , \tau - \alpha^{-1} p).
\]
As a sanity check, let us verify that these are indeed inverse elements in the monoid. After including into the larger monoid $\langle \widetilde{X},\widetilde{X} \rangle$, we have
\[
\bar{\pi}_{\alpha} = \left( \begin{matrix} I^{(1)} & I \\ 0 & I \end{matrix} \right) \qquad\text{and}\qquad \bar{\pi}_{\alpha^{-1}} = \left( \begin{matrix} I^{(-1)} & -I^{(-1)} \\ 0 & I \end{matrix} \right) ,
\]
where $A^{(\ell)}$ denotes the matrix obtained by shifting $A$ vertically upwards by $\ell$ steps, and these matrices are clearly inverses. This description also in particular encodes the fact that $\bar{\pi}_\alpha$ acts on $\pi_1(X) \cong \bZ\{\alpha\}$ by the identity and on $H_2(X;\bZ) \cong \bZ\{p,\tau\}$ by $\left( \begin{smallmatrix} 1 & 1 \\ 0 & 1 \end{smallmatrix} \right)$. (The action on $H_2$ is obtained by applying the operation $M_{\bZ}^{\mathrm{vf}}(\bZ) \to \bZ$ that takes the sum of the entries in the $0$th column to each entry of the $2 \times 2$ block matrix.)

The element $\bar{\pi}_\alpha = (\alpha , \alpha p , \tau + p) \in \langle X,X \rangle$ has infinite order: this can be detected by its action on $H_2(-;\bZ)$, but one may also directly calculate:
\[
(\bar{\pi}_{\alpha})^n = \underbrace{(\alpha , \alpha p , \tau + p) \circ \cdots \circ (\alpha , \alpha p , \tau + p)}_{n} = (\alpha , \alpha^n p , \tau + (1 + \alpha + \cdots + \alpha^{n-1})p),
\]
using the inclusion into $\langle \widetilde{X},\widetilde{X} \rangle$ and the identity $I^{(\ell)} I^{(k)} = I^{(\ell + k)}$. Hence the point-pushing homomorphism
\[
\pi_1(M) \cong \bZ\{\alpha\} \too \pi_0(\mathrm{hAut}_*(M \vee S^2)) \subset \langle X,X \rangle
\]
is injective. This factors through the point-pushing homomorphism
\[
\pi_1(M) \too \pi_0(\mathrm{Homeo}_*(M \smallsetminus *)),
\]
which is therefore also injective.
\end{eg}

\begin{eg}
Consider the more general example of
\[
M = {\underbrace{(S^1 \times S^2) \sharp (S^1 \times S^2) \sharp \cdots \sharp (S^1 \times S^2)}_{g \text{ copies}}} \smallsetminus \mathrm{int}(D^3).
\]
Now $M$ is homotopy equivalent to a wedge of $g$ circles (labelled by $\alpha_1,\ldots,\alpha_g$) and $g$ two-spheres (labelled by $\tau_1,\ldots,\tau_g$), so the point-pushing homomorphism is of the form
\begin{equation}
\label{eq:point-pushing-homomorphism-example-g}
\pi_1(M) \cong F_g = \langle \alpha_1,\ldots,\alpha_g \rangle \longrightarrow \pi_0(\mathrm{hAut}_*(X)) \subset \langle X,X \rangle ,
\end{equation}
where $X = S^1_{\alpha_1} \vee \ldots \vee S^1_{\alpha_g} \vee S^2_{\tau_1} \vee \ldots \vee S^2_{\tau_g} \vee S^2_p$. Here $\langle \alpha_1,\ldots,\alpha_g \rangle $ denotes the free group generated by $\alpha_1,\ldots,\alpha_g$ and $\langle X,X \rangle$ denotes the monoid $\pi_0(\mathrm{Map}_*(X,X))$, as before. We would like to describe the point-pushing maps $\bar{\pi}_{\alpha_1},\ldots,\bar{\pi}_{\alpha_g}$ (the images of $\alpha_1,\ldots,\alpha_g$) as elements of this monoid.

Generalising the discussion in the previous example, suppose that $X$ is a wedge of a number of circles indexed by a set $A$ and a number of two-spheres indexed by a set $B$. We then have
\[
\pi_1(X) \cong F_A \qquad\text{and}\qquad \pi_2(X) \cong \bZ[F_A]B,
\]
where $F_A$ is the free group on the set $A$, $\bZ[F_A]$ is its integral group-ring and $\bZ[F_A]B$ is the free $\bZ[F_A]$-module on the set $B$. The underlying set of the monoid $\langle X,X \rangle$ is therefore
\[
\langle X,X \rangle \cong \prod_A F_A \times \prod_B \bZ[F_A]B.
\]
To understand the operation of composition, it is again convenient to embed this into the larger monoid $\langle \widetilde{X},\widetilde{X} \rangle$, by lifting self-maps of $X$ to self-maps of its universal cover
\[
\widetilde{X} \simeq \bigvee_{w \in F_A} \bigvee_{b \in B} S^2_{wb}.
\]
This monoid is isomorphic to the monoid $M_B(M_{F_A}^{\mathrm{vf}}(\bZ))$ of $B \times B$ block matrices whose entries are $F_A \times F_A$ integer matrices that are vertically finite (each column has only finitely many non-zero entries).

Returning to our setting (and writing $\tau_0 = p$ for notational convenience), we have $A = \{\alpha_1,\ldots,\alpha_g\}$ and $B = \{\tau_0,\ldots,\tau_g\}$, so
\begin{align}
\begin{split}
\label{eq:monoids-higher-genus}
\langle X,X \rangle &= \prod_{i=1}^g \langle \alpha_1,\ldots,\alpha_g \rangle \times \prod_{i=0}^g \bZ\langle \alpha_1^{\pm 1},\ldots,\alpha_g^{\pm 1} \rangle \{\tau_0,\ldots,\tau_g\}, \\
\langle \widetilde{X},\widetilde{X} \rangle &= M_{g+1} \bigl( M_{\langle \alpha_1,\ldots,\alpha_g \rangle}^{\mathrm{vf}}(\bZ) \bigr) 
\end{split}
\end{align}
where $\bZ\langle \alpha_1^{\pm 1},\ldots,\alpha_g^{\pm 1} \rangle$ denotes the ring of non-commutative Laurent polynomials with coefficients in $\bZ$ in the variables $\alpha_1,\ldots,\alpha_g$. The embedding of monoids $\langle X,X \rangle \hookrightarrow \langle \widetilde{X},\widetilde{X} \rangle$ is given by
\[
(w_1,\ldots,w_g,f_0,\ldots,f_g) \;\longmapsto\; \left( \begin{matrix} A_{00} & \cdots & A_{0g} \\ \vdots & \ddots & \vdots \\ A_{g0} & \cdots & A_{gg} \end{matrix} \right) ,
\]
where the matrices $A_{ij}$ are determined as follows. First, consider $w=(w_1,\ldots,w_g)$ as the endomorphism of $\langle \alpha_1,\ldots,\alpha_g \rangle$ that sends the letter $\alpha_i$ to the word $w_i$. Each matrix $A_{ij}$ is ``diagonally constant of slope $-w$'', in the sense that if we write $A_{ij} = (a_{u,v})_{u,v \in \langle \alpha_1,\ldots,\alpha_g \rangle}$, then $a_{u,v} = a_{u.w(v)^{-1},1}$. In particular, each of these matrices is determined by its $1$st column. Finally, the $1$st columns of each of these matrices are determined by setting $(a_{u,1})_{ij}$ equal to the coefficient of $u\tau_i$ in $f_j$.

We may now describe the point-pushing maps $\bar{\pi}_{\alpha_1},\ldots,\bar{\pi}_{\alpha_g}$ and their inverses under the identifications \eqref{eq:monoids-higher-genus}. Namely, we have
\begin{align*}
\bar{\pi}_{\alpha_i} &= (\alpha_1,\ldots,\alpha_g , \alpha_i \tau_0 , \tau_1 ,\ldots, \tau_{i-1} , \tau_i + \tau_0 , \tau_{i+1} , \ldots, \tau_g) \\
\bar{\pi}_{\alpha_i^{-1}} &= (\alpha_1,\ldots,\alpha_g , \alpha_i^{-1} \tau_0 , \tau_1 ,\ldots, \tau_{i-1} , \tau_i - \alpha_i^{-1} \tau_0 , \tau_{i+1} , \ldots, \tau_g)
\end{align*}
as elements of $\langle X,X \rangle$, and
\begin{equation}
\label{eq:matrices-for-alpha-i}
\bar{\pi}_{\alpha_i} = \left( \begin{matrix}
I^{(\alpha_i)} & \cdots & I & \cdots \\
\vdots & \ddots & \vdots & \\
0 & \cdots & I & \cdots \\
\vdots & & \vdots & \ddots
\end{matrix} \right)
\qquad
\bar{\pi}_{\alpha_i^{-1}} = \left( \begin{matrix}
I^{(\alpha_i^{-1})} & \cdots & -I^{(\alpha_i^{-1})} & \cdots \\
\vdots & \ddots & \vdots & \\
0 & \cdots & I & \cdots \\
\vdots & & \vdots & \ddots
\end{matrix} \right)
\end{equation}
as elements of $\langle \widetilde{X},\widetilde{X} \rangle$, where unspecified entries agree with the identity matrix, and $A^{(w)}$ denotes the result of shifting the matrix $A$ vertically by $w$, in other words, if $A=(a_{u,v})$ and $A^{(w)} = (b_{u,v})$, then $b_{u,v} = a_{u.w^{-1},v}$.

From this description, we deduce a formula for $\bar{\pi}_w$ for any word $w$ in the generators $\alpha_1,\ldots,\alpha_g$. Note that we are \emph{not} assuming that the word $w$ is reduced.

\begin{prop}
Let $w$ be a word in the generators $\alpha_1,\ldots,\alpha_g$. Then
\[
\bar{\pi}_w = (\alpha_1,\ldots,\alpha_g , w\tau_0 , \tau_1 + f_1(w)\tau_0 , \tau_2 + f_2(w)\tau_0 , \ldots , \tau_g + f_g(w)\tau_0)
\]
as an element of $\langle X,X \rangle$, and
\[
\bar{\pi}_w = \left( \begin{matrix}
I^{(w)} & A_1(w) & A_2(w) & \cdots & A_g(w) \\
0 & I & 0 & \cdots & 0 \\
0 & 0 & I & \cdots & 0 \\
\vdots & \vdots & \vdots & \ddots & \vdots \\
0 & 0 & 0 & \cdots & I
\end{matrix} \right)
\]
as an element of $\langle \widetilde{X},\widetilde{X} \rangle$, where the non-commutative Laurent polynomials $f_i(w)$ and matrices $A_i(w)$ are defined as follows. Write
\[
w = w_1 \alpha_i^{\epsilon_1} w_2 \alpha_i^{\epsilon_2} \cdots w_\ell \alpha_i^{\epsilon_\ell} w_{\ell + 1},
\]
where the $w_j$ are words not involving $\alpha_i^{\pm 1}$ and $\epsilon_j \in \{\pm 1\}$ and let $\bar{w}_j$ be the initial subword
\[
\bar{w}_j = \begin{cases}
w_1 \alpha_i^{\epsilon_1} w_2 \alpha_i^{\epsilon_2} \cdots w_j & \text{if } \epsilon_j = +1 \\
w_1 \alpha_i^{\epsilon_1} w_2 \alpha_i^{\epsilon_2} \cdots w_j \alpha_i^{-1} & \text{if } \epsilon_j = -1
\end{cases}
\]
of $w$. Then
\[
f_i(w) = \sum_{j=1}^\ell \epsilon_j \bar{w}_j
 \qquad\text{and}\qquad 
 A_i(w) = \sum_{j=1}^\ell \epsilon_j I^{(\bar{w}_j)}.
\]
\end{prop}
\begin{proof}
The description of $\bar{\pi}_w$ as an element of $\langle X,X \rangle$ will follow from its description as an element of $\langle \widetilde{X},\widetilde{X} \rangle$ via the embedding of monoids described earlier, so we only have to prove the latter. Multiplying out the matrices \eqref{eq:matrices-for-alpha-i} corresponding to the letters of the word $w$, it is clear that $\bar{\pi}_w$ is of the form
\[
\left( \begin{matrix}
I^{(w)} & \mathrm{?}_1(w) & \mathrm{?}_2(w) & \cdots & \mathrm{?}_g(w) \\
0 & I & 0 & \cdots & 0 \\
0 & 0 & I & \cdots & 0 \\
\vdots & \vdots & \vdots & \ddots & \vdots \\
0 & 0 & 0 & \cdots & I
\end{matrix} \right) ,
\]
so we just have to verify that $\mathrm{?}_i(w) = A_i(w)$. We first note that, directly from the definition, the matrices $A_i(w)$ have the following property: if $w = w' w''$, then
\begin{equation}
\label{eq:property-of-Ai}
A_i(w) = A_i(w') + I^{(w')} A_i(w'').
\end{equation}
We also observe from \eqref{eq:matrices-for-alpha-i} that the equality $\mathrm{?}_i(w) = A_i(w)$ is true if $w$ is the letter $\alpha_i$ or its inverse.

We now prove that $\mathrm{?}_i(w) = A_i(w)$ by induction on the number of letters of $w$ that are equal to $\alpha_i$ or $\alpha_i^{-1}$. If this is zero, i.e.~if $w$ does not contain $\alpha_i^{\pm 1}$, then $A_i(w)=0$ and also $\mathrm{?}_i(w) = 0$, since it is the $(0,i)$ entry in a product of matrices that each have the property that their $i$th rows and columns agree with the identity matrix. This establishes the base case. If there are $\ell \geq 1$ letters of $w$ that are equal to $\alpha_i$ or $\alpha_i^{-1}$, then we may write $w = w' \alpha_i^{\epsilon} w''$, where $w''$ does not contain $\alpha_i^{\pm 1}$. Applying the base case to $w''$, the inductive hypothesis to $w'$ and using the observation above, we already know that
\[
\mathrm{?}_i(w'') = 0 = A_i(w'') \qquad \mathrm{?}_i(w') = A_i(w') \qquad \mathrm{?}_i(\alpha_i^\epsilon) = A_i(\alpha_i^\epsilon).
\]
Writing just the rows and columns indexed by $0$ and $i$, we therefore have
\begingroup
\setlength\arraycolsep{2pt}
\begin{align*}
\left( \begin{matrix}
I^{(w)} & \cdots & \mathrm{?}_i(w) \\
\vdots & \ddots & \vdots & \\
0 & \cdots & I
\end{matrix} \right)
&=
\left( \begin{matrix}
I^{(w')} & \cdots & \mathrm{?}_i(w') \\
\vdots & \ddots & \vdots & \\
0 & \cdots & I
\end{matrix} \right)
\left( \begin{matrix}
I^{(\alpha_i^\epsilon)} & \cdots & \mathrm{?}_i(\alpha_i^\epsilon) \\
\vdots & \ddots & \vdots & \\
0 & \cdots & I
\end{matrix} \right)
\left( \begin{matrix}
I^{(w'')} & \cdots & \mathrm{?}_i(w'') \\
\vdots & \ddots & \vdots & \\
0 & \cdots & I
\end{matrix} \right)
\\
&=
\left( \begin{matrix}
I^{(w')} & \cdots & A_i(w') \\
\vdots & \ddots & \vdots & \\
0 & \cdots & I
\end{matrix} \right)
\left( \begin{matrix}
I^{(\alpha_i^\epsilon)} & \cdots & A_i(\alpha_i^\epsilon) \\
\vdots & \ddots & \vdots & \\
0 & \cdots & I
\end{matrix} \right)
\left( \begin{matrix}
I^{(w'')} & \cdots & A_i(w'') \\
\vdots & \ddots & \vdots & \\
0 & \cdots & I
\end{matrix} \right)
\\
&=
\left( \begin{matrix}
I^{(w)} & \cdots & A_i(w) \\
\vdots & \ddots & \vdots & \\
0 & \cdots & I
\end{matrix} \right) ,
\end{align*}
\endgroup
where we apply the identity \eqref{eq:property-of-Ai} twice to deduce the final equality.
\end{proof}

In particular, we note that the coefficient of the generator $p=\tau_0$ in the middle component of $\bar{\pi}_w$ is exactly $w \in F_g \subseteq \bZ[F_g] = \bZ\langle \alpha_1^{\pm 1} , \ldots , \alpha_g^{\pm 1} \rangle$. This implies that the point-pushing homomorphism \eqref{eq:point-pushing-homomorphism-example-g} is injective.
\end{eg}

The two examples above go through identically if $S^1 \times S^2$ is replaced with $S^1 \times S^{d-1}$ for any $d\geq 3$; we obtain the same formulas for the point-pushing maps $\bar{\pi}_{\alpha}$ and the point-pushing homomorphism $\alpha \mapsto \bar{\pi}_{\alpha}$ is injective. Thus we have seen that, for any manifold of the form
\[
M = M_{g,1}^d = {\underbrace{(S^1 \times S^{d-1}) \sharp (S^1 \times S^{d-1}) \sharp \cdots \sharp (S^1 \times S^{d-1})}_{g \text{ copies}}} \smallsetminus \mathrm{int}(D^d)
\]
for $d \geq 3$ and $g\geq 0$, the point-pushing homomorphism
\begin{equation}
\label{eq:point-pushing-homomorphism}
\mathrm{push}_M \colon \pi_1(M) \too \pi_0(\mathrm{Homeo}_*(M \smallsetminus *)) \too \pi_0(\mathrm{hAut}_*(M \vee S^{d-1}))
\end{equation}
is injective. 
For $d=2$ this is also true: Recall  the point-pushing homomorphism is part of the Birman exact sequence \cite{Birman1969}:
\[
1 \to \pi_1(M_{g,1}^2) = F_{2g} \longrightarrow \Gamma_{g,1}^1 \longrightarrow \Gamma_{g,1} \to 1.
\]

In the next section, we put these facts into context by discussing the kernel of the point-pushing map more generally and for any number of configuration points.

%%%%%%%%%%%%%%%%%%%%%%%%%%%%%%%%%%%%%%%%%%%
%%%%%%%%%%%%%%%%%%%%%%%%%%%%%%%%%%%%%%%%%%%
\section{The kernel of the point-pushing map}
\label{s:kernel-point-pushing}
%%%%%%%%%%%%%%%%%%%%%%%%%%%%%%%%%%%%%%%%%%%
%%%%%%%%%%%%%%%%%%%%%%%%%%%%%%%%%%%%%%%%%%%

Let $M$ be a smooth, connected manifold of dimension $d \geq 3$ and fix a ball $D \subset M$ in the interior of $M$ containing the base configuration $z$. This determines an identification \eqref{equation:wreath-product} of $\pi_1(C_k(M))$ with the semi-direct product $\pi_1(M)^k \rtimes \Sigma_k$. For $\mathrm{Cat} \in \{ \mathrm{Diff} , \mathrm{Homeo} , \mathrm{hAut} \}$, recall from \S\ref{s:point-pushing-actions} that the point-pushing map
\begin{equation}
\label{eq:point-pushing-general}
p_k \colon \pi_1(C_k(M)) \longrightarrow \pi_0(\mathrm{Cat}(M,z))
\end{equation}
is the monodromy of the bundle $C_{k,1}(M) \to C_k(M)$, viewed either as a smooth bundle, a topological bundle or a Serre fibration.\footnote{In \S\ref{s:point-pushing-actions}, we focused on the $\mathrm{Cat} = \mathrm{hAut}$ setting, but the $\mathrm{Cat} = \mathrm{Homeo}$ and $\mathrm{Cat} = \mathrm{Diff}$ settings are exactly parallel.}

Except when $\mathrm{Cat} = \mathrm{hAut}$ and $k \geq 2$, this may equivalently be described as a connecting homomorphism in the long exact sequence of the fibration $\mathrm{Cat}(M) \to C_k(M)$ taking an automorphism $\varphi$ to its evaluation $\varphi(z)$ at the base configuration $z$.  (Note that such a description is impossible for $\mathrm{Cat} = \mathrm{hAut}$ and $k \geq 2$, since homotopy automorphisms need not be injective, so there is no well-defined map $\mathrm{hAut}(M) \to C_k(M)$ in this case.) Thus if $\mathrm{Cat} \in \{\mathrm{Diff} , \mathrm{Homeo}\}$, or if $\mathrm{Cat} = \mathrm{hAut}$ and $k=1$, the point-pushing map fits into an exact sequence of the form
\[
1 \longrightarrow \mathrm{ker}(p_k) \longrightarrow \pi_1(C_k(M)) \longrightarrow \pi_0(\mathrm{Cat}(M,z)) \longrightarrow \pi_0(\mathrm{Cat}(M)) \longrightarrow 1.
\]

When the boundary $\partial M$ is non-empty, the point-pushing map \eqref{eq:point-pushing-general} factors as
\begin{equation}
\label{eq:point-pushing-A}
p_k \colon \pi_1(C_k(M)) \longrightarrow \pi_0(\mathrm{Cat}_\partial(M,z))
\end{equation}
followed by $\pi_0(\mathrm{Cat}_\partial(M,z)) \to \pi_0(\mathrm{Cat}(M,z))$, where $\mathrm{Cat}_\partial(M) \subseteq \mathrm{Cat}(M)$ denotes the subspace of Cat-automorphisms of $M$ that fix $\partial M$ pointwise.

Despite the differences between the categories of $\mathrm{Diff}$ and $\mathrm{Homeo}$ on the one hand and $\mathrm{hAut}$ on the other, the following results hold for all three. Note though that $\mathrm{ker}(p_1)$ and $\mathrm{ker}(p_k)$ may be different groups for the three different categories.

\begin{prop}[{See also \cite[Lemmas 2.4 and 2.5]{Banks2017}.}]
\label{p:kernel-of-p1}
Let $k=1$. Then
\[
\mathrm{ker}(p_1) \subseteq Z(\pi_1(M)),
\]
i.e.~the kernel of \eqref{eq:point-pushing-general} is contained in the centre of $\pi_1(M)$. If $\partial M \neq \varnothing$, then \eqref{eq:point-pushing-A} is injective.
\end{prop}

\begin{prop}
\label{p:kernel-of-pk}
In general, we have that
\[
\mathrm{ker}(p_k) = \Delta(\mathrm{ker}(p_1)),
\]
i.e.~the kernel of \eqref{eq:point-pushing-general} is equal to the diagonal of $\mathrm{ker}(p_1)^k \subseteq \pi_1(M)^k \subseteq \pi_1(C_k(M))$, where we use the identification of $\pi_1(C_k(M))$ with $\pi_1(M)^k \rtimes \Sigma_k$ fixed above. If $\partial M \neq \varnothing$, then \eqref{eq:point-pushing-A} is injective.
\end{prop}

The first proposition is an immediate consequence of the following basic lemma.

\begin{lem}[{\cite[page 40]{Hatcher2002}}]
\label{Hatcher-exercise}
For any space $X$, the image of the map $\pi_1(\mathrm{hAut}(X)) \to \pi_1(X)$ induced by evaluation at some point $x \in X$ has image contained in the centre $Z(\pi_1(X))$.
\end{lem}

\begin{proof}[Proof of Proposition \ref{p:kernel-of-p1}]
By the long exact sequence, the kernel of \eqref{eq:point-pushing-general} is equal to the image of the map $\pi_1(\mathrm{Cat}(M)) \to \pi_1(M)$ induced by evaluation at the point $z_1 \in M$. The first statement then follows from Lemma \ref{Hatcher-exercise}. Similarly, the kernel of \eqref{eq:point-pushing-A} is equal to the image of the map on $\pi_1$ induced by the evaluation map $\mathrm{Cat}_\partial(M) \to M$ at $z_1 \in M$. But evaluation at $z_1$ is homotopic to evaluation at some point in $\partial M \neq \varnothing$ (since $M$ is path-connected), so this map is nullhomotopic.
\end{proof}

\begin{proof}[Proof of Proposition \ref{p:kernel-of-pk} in the smooth or topological setting.]
In this proof we assume that $\mathrm{Cat} \in \{ \mathrm{Diff} , \mathrm{Homeo} \}$, and we use the long exact sequence into which the point-pushing map \eqref{eq:point-pushing-general} fits. We give a separate proof in the setting $\mathrm{Cat} = \mathrm{hAut}$ further below. That proof also works in the smooth or topological category, but it is more involved, so we give a more geometric proof in these categories first.

Using the long exact sequence and the identification of $\pi_1(C_k(M))$ with $\pi_1(M)^k \rtimes \Sigma_n$, we have a diagram
\begin{equation}
\label{eq:kernel-of-pk-proof}
\begin{tikzcd}
\pi_1(\mathrm{Cat}(M)) \ar[rr] \ar[d] && \pi_1(M)^k \rtimes \Sigma_k \ar[rr,"p_k"] \ar[drr] && \pi_0(\mathrm{Cat}(M,z)) \ar[d] \\
\pi_1(\mathrm{hAut}(C_k(M))) \ar[urr] &&&& \Sigma_k
\end{tikzcd}
\end{equation}
whose top row is exact and right vertical map records the permutation of $z$ induced by the automorphism. It follows from the right-hand side of this diagram that $\mathrm{ker}(p_k)$ is contained in $\pi_1(M)^k$. 
A diffeomorphism or homeomorphism $\phi$ of $M$ induces a diffeomorphism or homeomorphism -- in particular a homotopy automorphism -- of $C_k(M)$, and so the top left arrow factors through $\pi_1(\mathrm{hAut}(C_k(M)) \to \pi_1(M)^k \rtimes \Sigma_k$. From Lemma \ref{Hatcher-exercise} it thus follows that $\mathrm{ker}(p_k)$ is contained in the centre $Z(\pi_1(M)^k \rtimes \Sigma_k)$. Together with the fact that $\mathrm{ker}(p_k) \subseteq \pi_1(M)^k$, we deduce that it is contained in the diagonal copy of $Z(\pi_1(M))$ in $Z(\pi_1(M))^k \subseteq \pi_1(M)^k \subseteq \pi_1(M)^k \rtimes \Sigma_k$. (Except when $k=2$ and $\pi_1(M)=1$, the centre $Z(\pi_1(M)^k \rtimes \Sigma_k)$ is precisely this diagonal copy of $Z(\pi_1(M))$. On the other hand, in the somewhat degenerate special case of $k=2$ and $\pi_1(M)=1$, the centre of $\pi_1(M)^k \rtimes \Sigma_k$ is $\Sigma_2$.)
Next, we consider the commutative diagram
\[\begin{tikzcd}
\pi_1(\mathrm{Cat}(M)) \ar[rr,"(*)"] \ar[drr,swap,"(**)"] && \pi_1(M)^k \ar[d,"\mathrm{proj}_1"] \\
&& \pi_1(M)
\end{tikzcd}\]
where the image of $(*)$ is $\mathrm{ker}(p_k)$ and the image of $(**)$ is $\mathrm{ker}(p_1)$ (these identifications follow, again, from the relevant long exact sequences, for general $k$ and for $k=1$ respectively). We know already that $\mathrm{ker}(p_k)$ is equal to $\Delta(G) \subseteq G^k$ for a certain subgroup $G \subseteq Z(\pi_1(M)) \subseteq \pi_1(M)$. Since this is a diagonal subgroup of the product $\pi_1(M)^k$, the projection onto the first factor restricts to an isomorphism of $\Delta(G)$ onto $G \subseteq \pi_1(M)$. By commutativity of the above diagram, it follows that $G = \mathrm{ker}(p_1)$. This concludes the proof of the first statement of the proposition. For the second statement, we repeat the same arguments with $\mathrm{Cat}$ replaced by $\mathrm{Cat}_\partial$ everywhere to obtain a similar formula, and then apply Proposition \ref{p:kernel-of-p1}.
\end{proof}

For the proof of Proposition \ref{p:kernel-of-pk} in the homotopy setting, we will use the following basic lemmas.

\begin{lem}
\label{lem:Serre-fibration}
Let $A \subseteq X$ be a cofibration and $Y$ any space, and assume that $X$ and $A$ are exponentiable, for example locally compact Hausdorff. Then the restriction map $\mathrm{Map}(X,Y) \to \mathrm{Map}(A,Y)$ is a Serre fibration. Moreover, the restriction map $\mathrm{hAut}(X) \to \mathrm{Map}(A,X)$ is also a Serre fibration.
\end{lem}
\begin{proof}
The first step is to prove that, for any space $Z$, the inclusion $Z \times A \hookrightarrow Z \times X$ is also a cofibration. This is most easily seen using the characterisation \cite[Proposition A.18]{Hatcher2002} of cofibrations $A \hookrightarrow X$ as those inclusions for which $X \times [0,1]$ retracts onto $(X \times \{0\}) \cup (A \times [0,1])$. If $r$ is a retraction witnessing that $A \hookrightarrow X$ is a cofibration, then $\mathrm{id}_Z \times r$ is a retraction witnessing that $Z \times A \hookrightarrow Z \times X$ is a cofibration.

Now suppose that we have a homotopy lifting problem as follows:
\begin{equation}
\label{eq:hlp}
\begin{tikzcd}
Z \times \{0\} \ar[rr,"f"] \ar[d] && Y^X \ar[d] \\
Z \times [0,1] \ar[rr,"g"] && Y^A
\end{tikzcd}
\end{equation}
By taking adjoints twice (since $X$ and $A$ are exponentiable), we may rewrite this as:
\begin{equation}
\label{eq:hep}
\begin{tikzcd}
Z \times A \ar[rr,"{g'}"] \ar[d,hook] && Y^{[0,1]} \ar[d,"{\mathrm{ev}_0}"] \\
Z \times X \ar[rr,"{f'}"] && Y
\end{tikzcd}
\end{equation}
This admits a lift $h' \colon Z \times X \to Y^{[0,1]}$, since $Z \times A \hookrightarrow Z \times X$ is a cofibration. Taking adjoints twice again, we obtain a lift $h \colon Z \times [0,1] \to Y^X$ of \eqref{eq:hlp}. Thus we have shown that
\[
\mathrm{Map}(X,Y) = Y^X \longrightarrow Y^A = \mathrm{Map}(A,Y)
\]
is a Hurewicz fibration, so in particular a Serre fibration.

In particular, this says that the restriction map $\mathrm{Map}(X,X) \to \mathrm{Map}(A,X)$ is a Serre fibration. In general, whenever $E \to B$ is a Serre fibration and $E_0 \subseteq E$ is a union of path-components, the restriction $E_0 \to B$ is also a Serre fibration. Since $\mathrm{hAut}(X)$ is a union of path-components of $\mathrm{Map}(X,X)$, this implies the second statement of the lemma.
\end{proof}

\begin{lem}
\label{lem:inclusion-hAut-weq}
Let $A \subseteq B \subseteq X$ be cofibrations of exponentiable spaces such that $B$ admits a strong deformation retraction onto $A$. Then the inclusion $\mathrm{hAut}_B(X) \hookrightarrow \mathrm{hAut}_A(X)$ is a weak homotopy equivalence.
\end{lem}

Here we write $\mathrm{hAut}_A(X)$ for the space of homotopy automorphisms of $X$ that agree with the identity on $A$. We will also use the notation $\mathrm{Map}_A(B,X)$ for the space of maps $B \to X$ that agree with the inclusion on $A$.

We will use this lemma below when $X$ is a manifold, $B \subset X$ is an embedded interval and $A$ is a point in this interval. (Manifolds are locally compact Hausdorff, hence exponentiable.)

\begin{proof}[Proof of Lemma \ref{lem:inclusion-hAut-weq}]
Consider the following commutative diagram
\begin{equation}
\label{eq:big-diagram}
\begin{tikzcd}
&& \mathrm{hAut}_B(X) \ar[rr,hook] \ar[d,hook] && \mathrm{hAut}_A(X) \ar[d,hook] \\
&& \mathrm{hAut}(X) \ar[rr,"{\mathrm{id}}"] \ar[d,"-|_B",swap] && \mathrm{hAut}(X) \ar[d,"-|_A"] \\
\mathrm{Map}_A(B,X) \ar[rr,hook] && \mathrm{Map}(B,X) \ar[rr,"-|_A"] && \mathrm{Map}(A,X)
\end{tikzcd}
\end{equation}
The maps denoted $-|_B$ and $-|_A$ are restrictions to $B$ respectively $A$, and are Serre fibrations by Lemma \ref{lem:Serre-fibration}. Let $r\colon B \to A$ be a retraction of $B$ onto $A$ and let $h_t \colon B \to B$ be a homotopy between $\mathrm{inc}_A^B \circ r$ and $\mathrm{id}_B$ relative to $A$. Then $\mathrm{inc}_A^X \circ r$ is a point in the bottom-left space $\mathrm{Map}_A(B,X)$ and a deformation retraction $H_t$ of this space onto the point $\{ \mathrm{inc}_A^X \circ r \}$ is given by $H_t(f) = f \circ h_t$. From the long exact sequence of homotopy group it follows that the bottom horizontal map $-|_A$ induces isomorphisms on $\pi_*$ for $*\geq 1$ and an injection on $\pi_0$. This map is also clearly surjective since any map $A\to X$ may be extended to $B$ using the retraction $r$, so it is a weak homotopy equivalence. It then follows from the $5$-lemma (and a little extra care in degree $0$) that $\mathrm{hAut}_B(X) \hookrightarrow \mathrm{hAut}_A(X)$ is a weak homotopy equivalence.
\end{proof}

\begin{proof}[Proof of Proposition \ref{p:kernel-of-pk} in the homotopy setting.]
In this setting, we cannot use the long exact sequence, so we give a different argument. First, the right-hand side of diagram \eqref{eq:kernel-of-pk-proof} implies that $\mathrm{ker}(p_k) \subseteq \pi_1(M)^k$. We then consider the commutative square in diagram \eqref{eq:point-pushing-diagonal} below, where the subscript $z$ means that $z$ is fixed \emph{pointwise}. It follows that $\mathrm{ker}(p_k) \subseteq \mathrm{ker}(p_1)^k$.

We next show that $\mathrm{ker}(p_k)$ \emph{contains} the diagonal $\Delta(\mathrm{ker}(p_1))$. Fix an element $a_1 \in \pi_1(M)$, set $a = (a_1,\ldots,a_1) \in \pi_1(M)^k$ and consider the diagram
\begin{equation}
\label{eq:point-pushing-diagonal}
\begin{tikzcd}
&& \pi_0(\mathrm{hAut}_I(M)) \ar[d] \\
\pi_1(M)^k \ar[rr,"p_k"] \ar[d,swap,"\mathrm{proj}_i"] && \pi_0(\mathrm{hAut}_z(M)) \ar[d] \\
\pi_1(M) \ar[rr,"p_1"] && \pi_0(\mathrm{hAut}(M,z_i))
\end{tikzcd}
\end{equation}
for some fixed $i$ (say $i=1$), where $I \subset M$ is an embedded interval containing the configuration $z$ and again the subscript $I$ means that $I$ is fixed \emph{pointwise}. We observe that the element $p_k(a) \in \pi_0(\mathrm{hAut}_z(M))$ may be lifted to an element $\varphi \in \pi_0(\mathrm{hAut}_I(M))$, defined as follows. Choose an isotopy of embeddings $I \hookrightarrow M$ starting at the inclusion, pulling the interval $I$ around the loop $a_1$ and then ending at the inclusion again. This may be constructed similarly to the explicit description of the (smooth) point-pushing map in Lemma \ref{l:geometric-point-pushing} and Figure \ref{fig:point-pushing}, using a tubular neighbourhood of an embedded representative of the loop $a_1$, which is a $D^{d-1}$-bundle over $a_1$, and a choice of trivial sub-$I$-bundle. Extend this by the isotopy extension theorem to a path in $\mathrm{Diff}(M)$ from $\mathrm{id}$ to $\varphi$. Then $\varphi$ is a diffeomorphism (hence homotopy automorphism) of $M$ fixing $I$ pointwise and representing $p_k(a)$ when considered as a homotopy automorphism of $M$ fixing $z \subset I$ pointwise. Now if we assume that $a_1 \in \mathrm{ker}(p_1)$, it follows that $\varphi = 1 \in \pi_0(\mathrm{hAut}_I(M))$, since the inclusion $\mathrm{hAut}_I(M) \hookrightarrow \mathrm{hAut}(M,z_1)$ is a weak homotopy equivalence by Lemma \ref{lem:inclusion-hAut-weq}, so in particular it induces an injection on $\pi_0$. It then also follows that $a = (a_1,\ldots,a_1) \in \mathrm{ker}(p_k)$.

Finally, suppose that $\mathrm{ker}(p_k) \neq \Delta(\mathrm{ker}(p_1))$. Then there must be an element $a = (a_1,\ldots,a_k) \in \mathrm{ker}(p_k) \smallsetminus \Delta(\mathrm{ker}(p_1))$. Since $a_1 \in \mathrm{ker}(p_1)$, we already know that $(a_1^{-1},\ldots,a_1^{-1}) \in \mathrm{ker}(p_k)$, so we also have $b = (b_1,\ldots,b_k) \in \mathrm{ker}(p_k) \smallsetminus \Delta(\mathrm{ker}(p_1))$, where $b_i = a_i a_1^{-1}$, in particular $b_1 = 1$. Choose embedded paths $A_i$ from $z_1$ to $z_i$ for each $i \in \{2,\ldots,k\}$ that are pairwise disjoint except at $z_1$. Also choose embedded loops based at $z_i$ representing $b_i$ (also denoted $b_i$ by abuse of notation) for each $i \in \{2,\ldots,k\}$. We may assume that the loops $b_i$ are pairwise disjoint, and also disjoint from the arcs $A_j$ except at $z_i$. We also assume that the point-pushing automorphism $p_k(b) \in \mathrm{hAut}(M,z)$ has support contained in a small tubular neighbourhood of the union of the loops $b_i$. See Figure \ref{fig:arcs}. By assumption, there is a homotopy $\mathrm{id} \simeq p_k(b)$ of self-maps $(M,z) \to (M,z)$. Restricting this to the embedded path $A_i$, we see that $A_i \simeq p_k(b)(A_i)$ relative to endpoints. Thus, we have
\[
b_i \simeq A_i^{-1} \cdot p_k(b)(A_i) \simeq A_i^{-1} \cdot A_i \simeq *,
\]
where $\cdot$ denotes concatenation of paths. So $b = (1,\ldots,1)$, which is a contradiction.

This finishes the proof of the first statement of the proposition. For the second statement, just as before, we repeat the same arguments with $\mathrm{hAut}$ replaced by $\mathrm{hAut}_\partial$ everywhere to obtain a similar formula, and then apply Proposition \ref{p:kernel-of-p1}.
\end{proof}

\begin{rmk}
The kernel of \eqref{eq:point-pushing-general} for $k=1$, in the $3$-dimensional topological (equivalently smooth) setting, has been understood completely by \cite{Banks2017}. By Proposition \ref{p:kernel-of-pk}, it is therefore also understood completely for \emph{all} $k$ in the $3$-dimensional topological/smooth setting.
\end{rmk}

\begin{rmk}
If $M$ does not necessarily have boundary, but it is equipped with marked points that are required to be fixed under automorphisms, then the corresponding point-pushing map
\[
p_k \colon \pi_1(C_k(M \smallsetminus P)) \longrightarrow \pi_0(\mathrm{Cat}_P(M,z))
\]
is injective when the set $P \subset M$ of marked points is non-empty, just as in the $\partial M \neq \varnothing$ setting. For $k=1$ this follows since $\mathrm{ker}(p_1)$ is the image of the map on $\pi_1$ induced by evaluation $\mathrm{Cat}_P(M) \to M$ at a point $z_1 \in M \smallsetminus P$, which is homotopic to evaluation at a point in $P$, hence nullhomotopic. For higher $k$, the proof above adapts to show that $\mathrm{ker}(p_k) = \Delta(\mathrm{ker}(p_1))$ also in this setting, and hence $p_k$ is also injective.
\end{rmk}

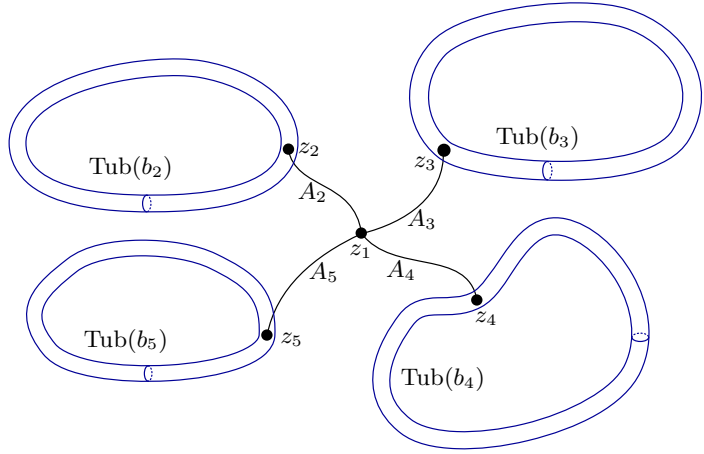
\begin{figure}[t]
\labellist
\small\hair 2pt
 \pinlabel {$z_1$} [ ] at 219 125
 \pinlabel {$z_2$} [ ] at 187 189
 \pinlabel {$z_3$} [ ] at 259 180
 \pinlabel {$z_4$} [ ] at 297 83
 \pinlabel {$z_5$} [ ] at 176 70
 \pinlabel {$A_2$} [ ] at 189 162
 \pinlabel {$A_5$} [ ] at 196 112
 \pinlabel {$A_3$} [ ] at 257 145
 \pinlabel {$A_4$} [ ] at 244 113
 \pinlabel {$\mathrm{Tub}(b_2)$} [ ] at 76 177
 \pinlabel {$\mathrm{Tub}(b_5)$} [ ] at 73 69
 \pinlabel {$\mathrm{Tub}(b_3)$} [ ] at 329 196
 \pinlabel {$\mathrm{Tub}(b_4)$} [ ] at 270 45
\endlabellist
    \centering
    \includegraphics[scale=0.6]{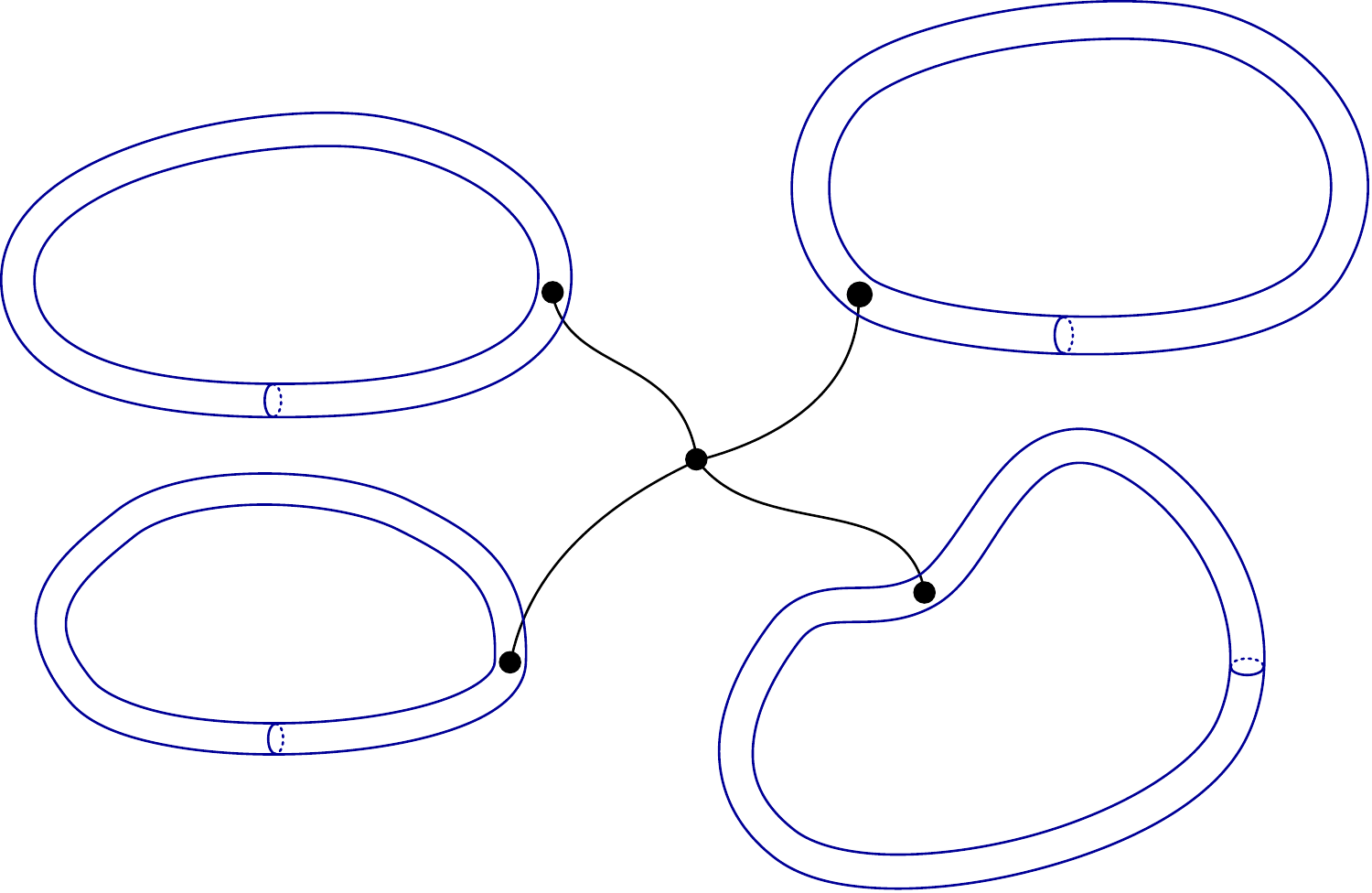}
    \caption{The paths $A_i$ and the support of the point-pushing automorphism $p_k(b)$.}
    \label{fig:arcs}
\end{figure}

\begin{rmk}[\emph{A fake Dehn twist.}]
Lemma \ref{lem:inclusion-hAut-weq} in the setting $(X,B,A) = (S,D,\{*\})$ for a surface $S$ with embedded closed disc $D \subset S$ with centre $* \in D$ has the following potentially counter-intuitive consequence. Let $T_D \in \mathrm{hAut}_D(S)$ be a Dehn twist supported in a small annular neighbourhood of $D$ in $S$. Then the element
\begin{equation}
\label{eq:Dehn-twist-relative-to-disc}
[T_D] \in \pi_0(\mathrm{hAut}_D(S))
\end{equation}
of the mapping class group of $S$ relative to $D$ is trivial: this is because its image in $\pi_0(\mathrm{hAut}_*(S))$ is clearly trivial -- one may simply untwist $T_D$ while keeping the point $*$ fixed -- and the map $\pi_0(\mathrm{hAut}_D(S)) \to \pi_0(\mathrm{hAut}_*(S))$ is injective by Lemma \ref{lem:inclusion-hAut-weq}. In contrast, the element
\begin{equation}
\label{eq:Dehn-twist-remove-disc}
[T_D] \in \pi_0(\mathrm{hAut}_{\partial D}(S \smallsetminus \mathrm{int}(D)))
\end{equation}
is well-known to be non-trivial (and of infinite order) in the mapping class group of $S \smallsetminus \mathrm{int}(D)$ relative to the boundary-component $\partial D$, as long as $S$ is not the $2$-sphere or the $2$-disc.\footnote{To see this, write $\gamma$ for a curve in the interior of $S \smallsetminus \mathrm{int}(D)$ parallel to $\partial D$, so that $T_\gamma = T_D$, and choose an arc $\alpha$ in $S \smallsetminus \mathrm{int}(D)$ with both endpoints on $\partial D$ so that $i(\alpha,\gamma) = 2$, where $i(-,-)$ is the minimal geometric intersection number amongst isotopic representatives. Then $i(T_{\gamma}^k(\alpha),\alpha)$ is strictly increasing as $k\to\infty$. See \cite[Proposition 3.2]{FarbMargalitPrimer} for details (in the case of closed surfaces, which may easily be adapted to compact surfaces with boundary).}

Notice that such an apparent discrepancy cannot occur if $\mathrm{hAut}(-)$ is replaced with $\mathrm{Homeo}(-)$ or $\mathrm{Diff}(-)$, since in these two cases there is a canonical homeomorphism between $\mathrm{Cat}_D(S)$ and $\mathrm{Cat}_{\partial D}(S \smallsetminus \mathrm{int}(D))$ for $\mathrm{Cat} \in \{ \mathrm{Homeo} , \mathrm{Diff} \}$.

The reason for this apparent discrepancy in the $\mathrm{Cat} = \mathrm{hAut}$ setting is illustrated by exhibiting an explicit nullhomotopy of \eqref{eq:Dehn-twist-relative-to-disc}: see Figure \ref{fig:fake-Dehn-twist}. This nullhomotopy depends on the fact that points may be mapped into the disc $D$ (hence why it does not work for \eqref{eq:Dehn-twist-remove-disc}) and also the fact that homotopy equivalences may be non-injective (hence why it does not work for $\mathrm{Cat} \in \{ \mathrm{Homeo} , \mathrm{Diff} \}$).
\end{rmk}

\begin{figure}
    \centering
    \includegraphics[scale=0.3]{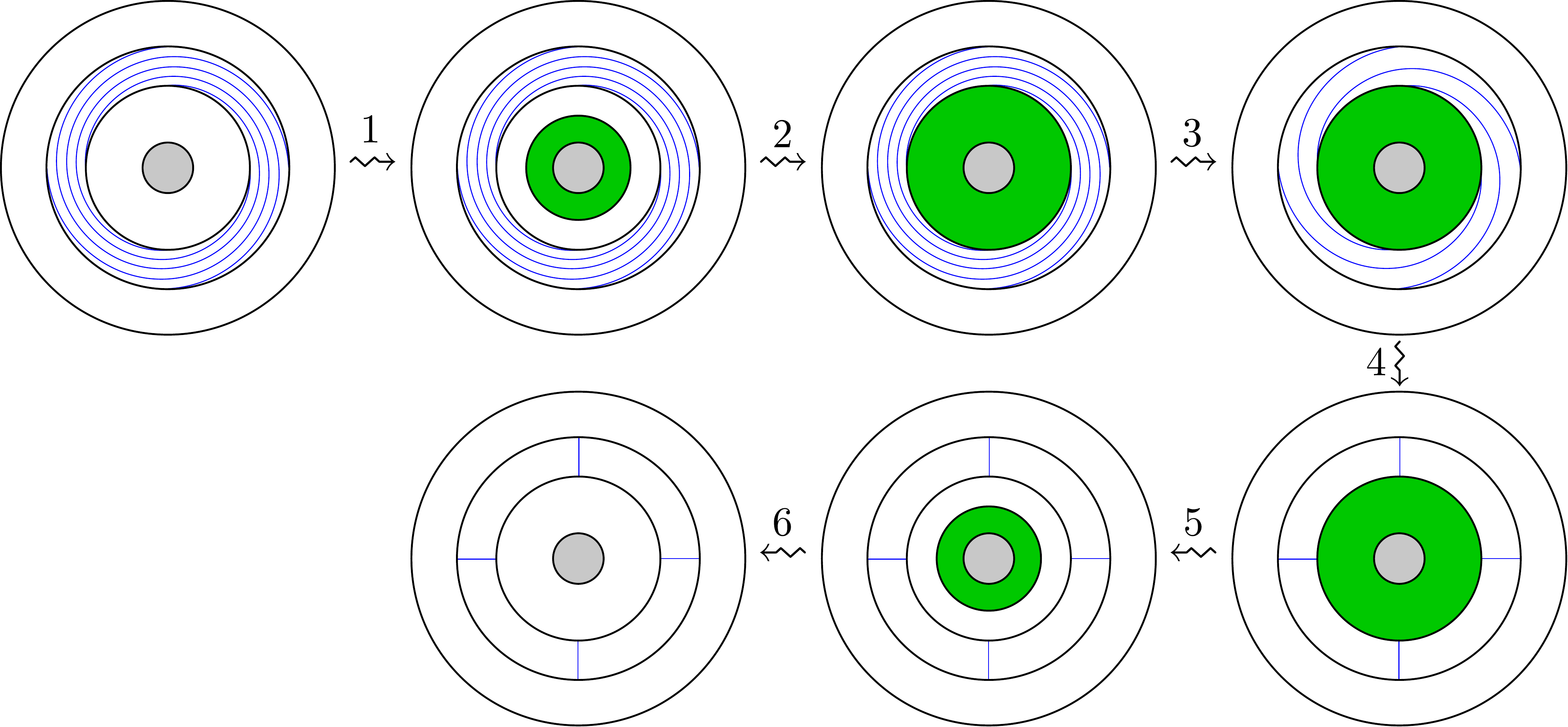}
    \caption{A nullhomotopy of the ``fake Dehn twist'' in the (homotopy automorphism version of the) mapping class group of $S$ relative to $D$. The blue lines indicate twisting in an annular region. The central grey disc is $D$. The green annulus surrounding $D$ in some of the pictures indicates that the inner boundary of the green annulus is mapped to itself by the identity (as it must be), the outer boundary of the green annulus is sent the the midpoint of the disc, and the interior of the green annulus is ``turned inside out'' and mapped onto the grey disc $D$. The untwisting of the Dehn twist in steps 3 and 4 is well-defined exactly because the outer boundary of the green annulus is collapsed to a point. The homotopies in steps 1 and 6 are given by gradually ``folding'' the green annulus inwards, while keeping the grey disc fixed, until the outer boundary of the green annulus is collapsed to the midpoint of the disc. Steps 2 and 5 are not strictly necessary, since one could directly perform the homotopies of steps 1 and 6 with the larger green annulus, but they perhaps make the picture more intuitive.}
    \label{fig:fake-Dehn-twist}
\end{figure}

\begin{rmk}
\label{rmk:relative-hAut}
There is a subtle difference between the space $\mathrm{hAut}_A(X)$ involved in Lemma \ref{lem:inclusion-hAut-weq} and the proof of Proposition \ref{p:kernel-of-pk} and the space $\mathrm{hAut}(X|A)$ defined in \S\ref{s:monodromy-actions}, namely:
\begin{align*}
\mathrm{hAut}(X|A) &= \{ f \in \mathrm{Map}(X,X) \mid f|_A = \mathrm{id}_A \text{ and } f \text{ admits a homotopy inverse relative to } A \} \\
\mathrm{hAut}_A(X) &= \{ f \in \mathrm{Map}(X,X) \mid f|_A = \mathrm{id}_A \text{ and } f \text{ admits a homotopy inverse} \} ,
\end{align*}
so clearly $\mathrm{hAut}(X|A) \subseteq \mathrm{hAut}_A(X)$. In general, if $A \subseteq X$ is a cofibration and $f \colon X \to X$ restricts to the identity on $A$ and admits a homotopy inverse, then one may find both a \emph{left} homotopy inverse for $f$ relative to $A$ and a \emph{right} homotopy inverse for $f$ relative to $A$, but these may not necessarily coincide. On the other hand, if the space $\mathrm{Map}(A,A)$ is simply-connected, then one may always find a two-sided homotopy inverse for $f$ relative to $A$, and so in this case the two spaces $\mathrm{hAut}_A(X)$ and $\mathrm{hAut}(X|A)$ are equal. In particular, this holds if $A=D$ is a disc.
\end{rmk}

%%%%%%%%%%%%%%%%%%%%%%%%%%%%%%%%%%%%%%%%%%%
%%%%%%%%%%%%%%%%%%%%%%%%%%%%%%%%%%%%%%%%%%%
\section{Formulas for associated point-pushing actions on mapping spaces}\label{s:mapping-spaces}
%%%%%%%%%%%%%%%%%%%%%%%%%%%%%%%%%%%%%%%%%%%
%%%%%%%%%%%%%%%%%%%%%%%%%%%%%%%%%%%%%%%%%%%

As an immediate corollary of Proposition \ref{p:formula-sigma}, Lemma \ref{l:reduction-to-one-sphere} Proposition \ref{p:formula-alpha-1} and Lemma \ref{l:associated-point-pushing-comparison}, we obtain (under certain assumptions on $M$) a formula for the associated point-pushing action (Definition \ref{d:point-pushing-action-mapping}) of $\pi_1(C_k(M))$ on the mapping space $\mathrm{Map}_*^c(M \smallsetminus z,X)$, under the identification
\begin{equation}
\label{eq:identification-mapping-space}
\mathrm{Map}_*^c(M \smallsetminus z,X) \simeq \mathrm{Map}_*(M,X) \times (\Omega_c^{d-1}X)^k
\end{equation}
induced by the identification \eqref{construction:equivalence-of-pairs} of $M \smallsetminus z$ with $M \vee \bigvee^k S^{d-1}$. On the right-hand side of \eqref{eq:identification-mapping-space}, $\Omega_c^{d-1}X$ denotes the union of path-components of $\Omega^{d-1}X$ corresponding to the subset $c \subseteq [S^{d-1},X]$.

\begin{rmk}
\label{r:natural-actions}
There are two natural actions on the space $\Omega_c^{d-1}X$. First, there is an action-up-to-homotopy of $\pi_1(X)$ on $\Omega^{d-1}X$, which restricts to an action-up-to-homotopy on the subspace $\Omega_c^{d-1}X$ (this is because the subset $c \subseteq [S^{d-1},X]$ corresponds to a union of $\pi_1(X)$-orbits of $\pi_{d-1}(X)$).

Second, there is an involution of $\Omega^{d-1}X$ given by precomposition with a reflection of $S^{d-1}$ in a hyperplane containing the basepoint; this involution commutes with the action-up-to-homotopy of $\pi_1(X)$. If $c \subseteq [S^{d-1},X]$ is invariant under the corresponding involution of $[S^{d-1},X]$, then this involution restricts to the subspace $\Omega_c^{d-1}X$. In our situation, the involution will only be relevant if $M$ is non-orientable, in which case we have assumed (see Definition \ref{d:cmap-2}) that $c \subseteq [S^{d-1},X]$ is a subset of the fixed points under the involution, so in particular it is invariant under the involution.
\end{rmk}

\begin{coro}
\label{coro:point-pushing-mapping-space}
If $d = \mathrm{dim}(M) \geq 3$ and $M$ satisfies at least one of the following conditions:
\begin{itemizeb}
\item $M$ is simply-connected, or
\item the handle-dimension of $M$ is at most $d-2$;
\end{itemizeb}
then the point-pushing action of $\gamma = (\alpha_1,\ldots,\alpha_k;\sigma) \in \pi_1(C_k(M)) \cong \pi_1(M)^k \rtimes \Sigma_k$ on the mapping space $\mathrm{Map}_*^c(M \smallsetminus z,X)$, under the identification \eqref{eq:identification-mapping-space}, is given as follows \textup{(}see also Figure \ref{fig:point-pushing-mapping-spaces-1}\textup{)}
\begin{equation}
\label{eq:action-formula}
(\alpha_1,\ldots,\alpha_k;\sigma) \cdot (f,g_1,\ldots,g_k) = (f , \bar{g}_1 ,\ldots, \bar{g}_k),
\end{equation}
where $\bar{g}_i = f_*(\alpha_i) . g_{\sigma(i)} . \mathrm{sgn}(\alpha_i)$, and
\begin{itemizeb}
\item for an element $\alpha \in \pi_1(M)$ we write $\mathrm{sgn}(\alpha) = +1$ if $\alpha$ lifts to a loop in the orientation double cover of $M$ and $\mathrm{sgn}(\alpha) = -1$ otherwise,
\item the actions of $\pi_1(X)$ and of $\{\pm 1\}$ on $\Omega_c^{d-1}X$ are as described in Remark \ref{r:natural-actions} above.
\end{itemizeb}
\end{coro}
\begin{proof}
It suffices to check this for elements of the form $(1,\ldots,1;\sigma)$ and $(\alpha,1,\ldots,1;\mathrm{id})$ (symmetric and loop generators), which we denote simply by $\sigma$ and $\alpha$ by abuse of notation.

By Proposition \ref{p:formula-sigma}, the action of $\sigma$ on $M \smallsetminus z \simeq M \vee W_k$ is the identity on the $M$ summand and permutes the $k$ copies of $S^{d-1}$ in $W_k = \bigvee^k S^{d-1}$. Lemma \ref{l:associated-point-pushing-comparison} tells us that the associated point-pushing action of $\sigma$ on $\mathrm{Map}_*(M,X) \times (\Omega_c^{d-1}X)^k$ is induced from its point-pushing action on $M \vee W_k$ by precomposition, so we deduce that it acts by the identity on the $\mathrm{Map}_*(M,X)$ component and the $\Omega_c^{d-1}X$ components are permuted by $\sigma^{-1}$ (the inverse occurs since precomposition is contravariant).

Similarly, Lemma \ref{l:associated-point-pushing-comparison} implies that the point-pushing action of $\alpha$ on $\mathrm{Map}_*(M,X) \times (\Omega_c^{d-1}X)^k$ is induced from the point-pushing action of $\alpha$ on $M \vee W_k$, which is described by Lemma \ref{l:reduction-to-one-sphere} and Proposition \ref{p:formula-alpha-1}, by precomposition. Putting this together, we see that $\alpha$ sends the tuple $(f,g_1,\ldots,g_k)$ to the tuple $(f,f_*(\alpha).g_1.\mathrm{sgn}(\alpha),g_2,\ldots,g_k)$, as desired. Specifically, the $f$ entry in this tuple follows from the left-hand side of \eqref{eq:two-formulas}, the $f_*(\alpha).g_1.\mathrm{sgn}(\alpha)$ entry follows from the right-hand side of \eqref{eq:two-formulas} and the remaining entries follow from Lemma  \ref{l:reduction-to-one-sphere}.
\end{proof}

\definecolor{pushmap}{RGB}{0,150,0}
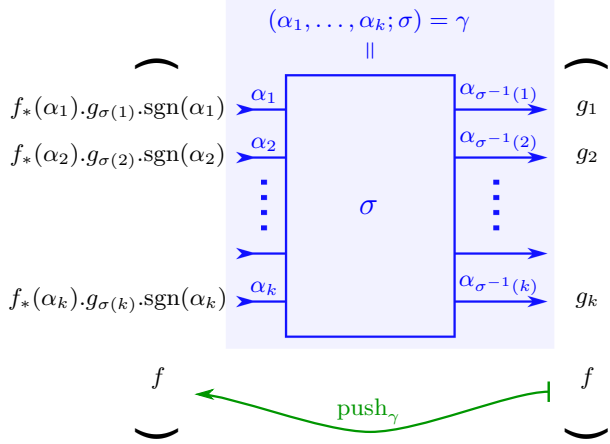
\begin{figure}[t]
\labellist
\small\hair 2pt
 \pinlabel {\textcolor{blue}{\large $\sigma$}} [ ] at 216 542
 \pinlabel {\textcolor{blue}{$\alpha_1$}} [b] at 151 604
 \pinlabel {\textcolor{blue}{$\alpha_2$}} [b] at 151 574
 \pinlabel {\textcolor{blue}{$\alpha_k$}} [b] at 151 484
 \pinlabel {$f_*(\alpha_1).g_{\sigma(1)}.\mathrm{sgn}(\alpha_1)$} [ ] at 60 604
 \pinlabel {$f_*(\alpha_2).g_{\sigma(2)}.\mathrm{sgn}(\alpha_2)$} [ ] at 60 574
 \pinlabel {$f_*(\alpha_k).g_{\sigma(k)}.\mathrm{sgn}(\alpha_k)$} [ ] at 60 484
 \pinlabel {$f$} [ ] at 85 435
 \pinlabel {\textcolor{blue}{$\alpha_{\sigma^{-1}(1)}$}} [b] at 297 604
 \pinlabel {\textcolor{blue}{$\alpha_{\sigma^{-1}(2)}$}} [b] at 297 574
 \pinlabel {\textcolor{blue}{$\alpha_{\sigma^{-1}(k)}$}} [b] at 297 484
 \pinlabel {$g_1$} [ ] at 352 604
 \pinlabel {$g_2$} [ ] at 352 574
 \pinlabel {$g_k$} [ ] at 352 484
 \pinlabel {$f$} [ ] at 352 435
 \pinlabel {\textcolor{pushmap}{$\mathrm{push}_\gamma$}} [b] at 216 403
 \pinlabel {\textcolor{blue}{\rotatebox{270}{$=$}}} [b] at 216 631
 \pinlabel {\textcolor{blue}{$(\alpha_1,\ldots,\alpha_k;\sigma) = \gamma$}} [b] at 216 648
\endlabellist
\centering
\includegraphics[scale=0.6]{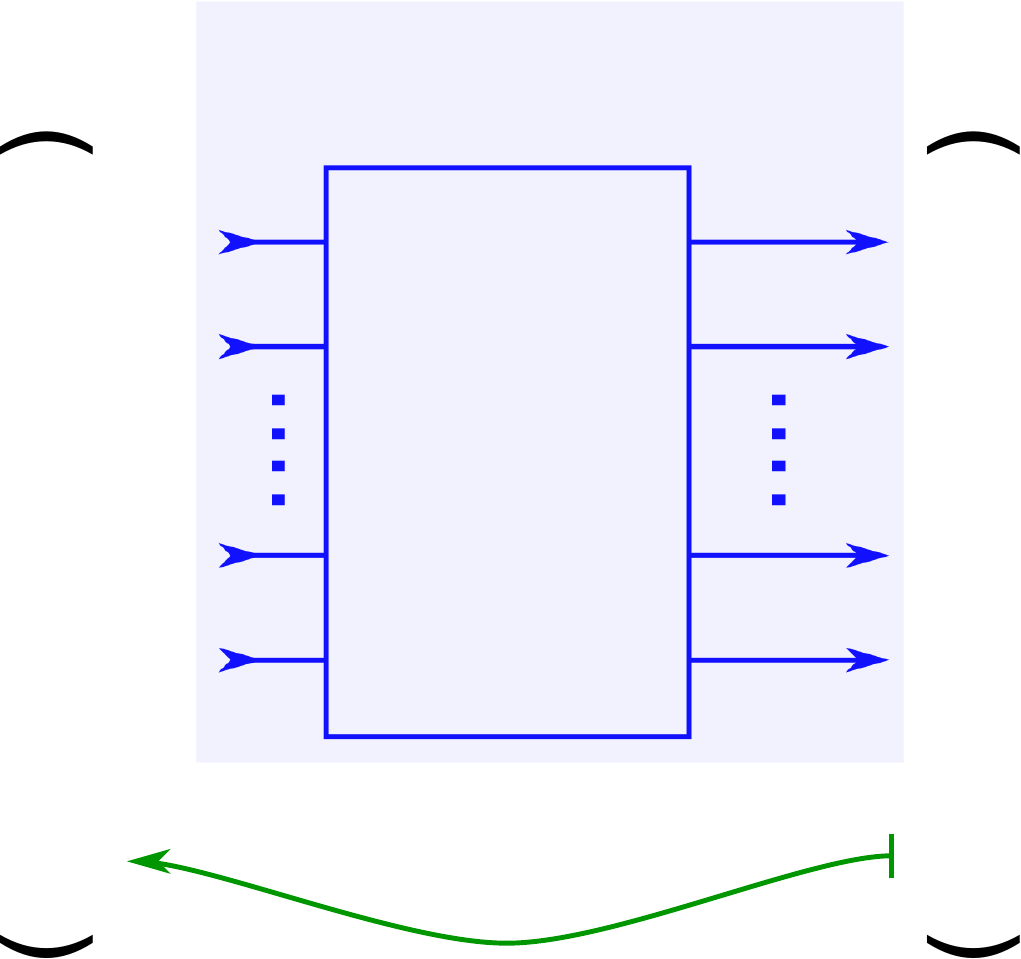}
\caption{The action of the point-pushing map associated to $\gamma = (\alpha_1,\ldots,\alpha_k;\sigma) \in \pi_1(C_k(M))$ on the mapping space $\mathrm{Map}_*(M,X) \times (\Omega_c^{d-1}X)^k$. The loop $\gamma$ is represented in \textcolor{blue}{blue}, the elements of the mapping space in black and the point-pushing map is represented in \textcolor{pushmap}{green}.}
\label{fig:point-pushing-mapping-spaces-1}
\end{figure}

\begin{rmk}
Part of the formula \eqref{eq:action-formula} remains valid without the additional hypothesis on $M$. More precisely, assuming still that $\mathrm{dim}(M) \geq 3$ but removing the second hypothesis (so $M$ is now allowed to be non-simply-connected and to have maximal handle-dimension), the formula for the action of $\gamma = (\alpha_1,\ldots,\alpha_k;\sigma)$ becomes
\begin{equation}
\label{eq:action-formula-2}
(\alpha_1,\ldots,\alpha_k;\sigma) \cdot (f,g_1,\ldots,g_k) = (? , \bar{g}_1 ,\ldots, \bar{g}_k),
\end{equation}
where the entry $?$ is not in general $f$, but rather a based map $M \to X$ that depends in a subtle way on $f$, the loop $\gamma$ and the elements $g_i$. For example, when $\gamma = (\alpha,1,\ldots,1;\mathrm{id})$, the map $? \colon M \to X$ is given by the composition
\[
\mathrm{fold} \circ (f \vee g_1) \circ \overline{\pitchfork}_\alpha \colon M \too M \vee S^{d-1} \too X \vee X \too X,
\]
where $\overline{\pitchfork}_\alpha$ is the map defined in Definition \ref{defn:pf2}. To see this, recall that the equations \eqref{eq:two-formulas} describe the point-pushing action of a loop generator $\alpha$ under the additional assumptions on $M$, and the equations \eqref{eq:two-formulas-2} describe the point-pushing action of $\alpha$ without these assumptions. The right-hand equation of \eqref{eq:two-formulas} agrees with the right-hand equation of \eqref{eq:two-formulas-2}, which is why the tuple $(\bar{g}_1 ,\ldots, \bar{g}_k)$ occurs in \eqref{eq:action-formula-2}, just as in \eqref{eq:action-formula}. However, the left-hand equation of \eqref{eq:two-formulas} is simply $\bar{\pi}^M_\alpha \simeq \mathrm{inc}_M$, whereas the left-hand equation of \eqref{eq:two-formulas-2} is $\bar{\pi}^M_\alpha \simeq \overline{\pitchfork}_\alpha$.
\end{rmk}

\begin{rmk}
\label{r:split-injectivity}
Corollary \ref{coro:point-pushing-mapping-space} is used in \cite[\S 8]{PalmerTillmann2020homologicalstabilityconfigurationsection} to prove a certain split-injectivity result for maps between configuration-mapping spaces. More precisely, there is a natural map of spectral sequences converging to the map on homology induced by the \emph{stabilisation map}
\[
\cmap{k}{c,*}{M}{X} \too \cmap{k+1}{c,*}{M}{X}.
\]
Under the hypotheses on $M$ assumed in Corollary \ref{coro:point-pushing-mapping-space}, this map of spectral sequences is split-injective on $E^2$ pages. For the precise statement, see \cite[Theorem 8.12]{PalmerTillmann2020homologicalstabilityconfigurationsection}.
\end{rmk}

Corollary \ref{coro:point-pushing-mapping-space} may also be used to understand the path-components of configuration-mapping spaces of manifolds of dimension at least $3$. As an example, we have the following.

\begin{coro}
\label{cor:connected_components}
Suppose that $d = \mathrm{dim}(M) \geq 3$, $M$ is orientable and either
\begin{itemizeb}
\item $M$ is simply-connected, or
\item the handle-dimension of $M$ is at most $d-2$.
\end{itemizeb}
Then there is a natural bijection
\begin{equation}
\label{eq:path-components-cmap}
\pi_0(\cmap{k}{c,*}{M}{X}) \;\cong\; \bigsqcup_{f \in \langle M,X \rangle} SP^k(c_f),
\end{equation}
where $\langle M,X \rangle = \pi_0(\mathrm{Map}_*(M,X))$, the notation $SP^k(\phantom{-})$ means $(\phantom{-})^k / \Sigma_k$ and $c_f$ is the pre-image of $c \subseteq [S^{d-1},X]$ under the quotient map
\[
\pi_{d-1}(X) / f_*(\pi_1(M)) \too \pi_{d-1}(X) / \pi_1(X) = [S^{d-1},X].
\]
\end{coro}
\begin{proof}
By the long exact sequence associated to the bundle \eqref{eq:configuration-mapping-space-charge}, the left-hand side of \eqref{eq:path-components-cmap} is naturally in bijection with the set of orbits of
\[
\pi_0(\mathrm{Map}_*^c(M \smallsetminus z,X)) \cong \langle M,X \rangle \times \widetilde{c}^k
\]
under the monodromy (i.e., point-pushing) action of $\pi_1(C_k(M))$, where $\widetilde{c}$ denotes the pre-image of $c \subseteq [S^{d-1},X]$ under the quotient map $\pi_{d-1}(X) \to \pi_{d-1}(X) / \pi_1(X) = [S^{d-1},X]$. Corollary \ref{coro:point-pushing-mapping-space} implies that the elements of $\pi_1(C_k(M))$ act on a tuple $([f],[g_1],\ldots,[g_k])$ by (i) permuting the $[g_i]$'s and (ii) acting on each $[g_i]$ (individually) by $f_*(\pi_1(M)) \leq \pi_1(X)$. The formula \eqref{eq:path-components-cmap} follows.
\end{proof}

\phantomsection
\addcontentsline{toc}{section}{References}
\renewcommand{\bibfont}{\normalfont\small}
\setlength{\bibitemsep}{0pt}
\printbibliography

\vspace{0pt plus 1filll}

\noindent {\itshape Institutul de Matematică Simion Stoilow al Academiei Române, 21 Calea Griviței, 010702 București, Romania}, {\tt mpanghel@imar.ro}

\noindent {\itshape Isaac Newton Institute for Mathematical Sciences, University of Cambridge, Clarkson Road, Cambridge CB3 0EH, UK}, {\tt ut213@cam.ac.uk}

\end{document}